\documentclass[11pt]{article}
\usepackage[mathscr]{eucal}
\usepackage[all]{xy} 
\usepackage{graphicx}
\usepackage{caption}
\usepackage{subcaption}
\usepackage{multicol}
\usepackage{amssymb,amsfonts,amsthm}    
\usepackage{color}                              
\usepackage[centertags]{amsmath}
\usepackage{amsfonts,amssymb,amsthm,newlfont}
\usepackage{geometry}
\usepackage{algorithmic}
\usepackage{hyperref}
\usepackage{mathrsfs}
\usepackage{enumerate}
\usepackage{comment}

\usepackage{nomencl}

\makeglossary

\theoremstyle{plain}
\newtheorem{thm}{Theorem}[section]

\newtheorem{corollary}[thm]{Corollary}

\newtheorem{remark}[thm]{Remark}
\newtheorem{example}[thm]{Example}
\newtheorem{algorithm}{Algorithm}[section]

\numberwithin{equation}{section}
\numberwithin{algorithm}{section}

\def\bfx{\mathbf{x}}
\def\bfv{\mathbf{v}}
\def\bfc{\mathbf{c}}

\begin{document}

\title{A Bivariate Spline Construction of Orthonormal Polynomials over Polygonal  Domains 
and Its Applications to Quadrature}
\author{ Ming-Jun Lai\footnote{mjlai@uga.edu. This author is associated with Department 
of Mathematics, University of
Georgia, Athens, GA 30602, U.S.A. This author is supported by the Simon Foundation by a 
Collaboration Grant \#864439.}  }
\date{}
\maketitle

\begin{abstract}
We present computational methods for constructing orthogonal/orthonormal polynomials over arbitrary polygonal domains in $\mathbb{R}^2$ 
 using bivariate spline functions. Leveraging a mature MATLAB implementation which generates  
spline spaces of any degree, any smoothness over any triangulation, we have exact polynomial representation over the polygonal domain of interest. 
Two algorithms are developed: one constructs orthonormal polynomials of degree $d>0$ 
 over a polygonal domain, and the other constructs orthonormal polynomials of degree 
$d+1$ in the orthogonal complement of $\mathbb{P}_d$. 
Numerical examples for degrees $d=1--5$
illustrate the structure and zero curves of these polynomials, providing evidence against the 
existence of Gauss quadrature on centrally symmetric domains. In addition, we introduce 
polynomial reduction strategies based on odd- and even-degree orthogonal polynomials, 
reducing the integration to the integration of its residual quadratic or linear polynomials. 
These reductions motivate new quadrature schemes, which we further extend through polynomial 
interpolation to obtain efficient, high-precision quadrature rules for various polygonal domains.
\end{abstract}

\section{Introduction}
We are interested in constructing multivariate orthogonal polynomials over any bounded 
domain $\Omega\subset \mathbb{R}^n, n \ge 2$
which can be approximated by a finitely many triangles in $\mathbb{R}^2$ or by a finitely 
many tetrahedra in  $\mathbb{R}^3$ and etc.. Mainly we know  orthonormal polynomials over 
a bounded  domain play a similar role like the well-known Fourier series over $[-\pi, 
\pi]$ or $[-\pi, \pi]^n$ for $n=2$ or $n\ge 3$  which are extremely useful in practice, 
e.g. cell phone 
communication, internet transmission of images and videos.  Another
significant application of  orthogonal polynomials is Gaussian quadrature based on the 
zeros of Legendre polynomials over $[-1, 1]$ so that the quadrature will be exact for 
polynomials of degree $\le 2d-1$ than the degree $d$ of the polynomials used for 
constructing 
the quadrature.  The usefulness of the quadrature formulas can be seen that 
they have been used  for constructing a projection of function $f$ in a finite dimensional 
space, 
e.g. polynomial space $\mathbb{P}_d$ with $d\ge 1$ and infinite dimensional 
space, e.g. for reproducing kernel Hilbert space. Multivariate version of these studies have been 
actively studied for many decades.   
Although there is a result that no Gaussian quadrature exists 
over centrally symmetric domains (cf. \cite{DX14} and \cite{X21}), 
we can find some replacement as discussed in this paper.    
Another interesting paper is \cite{AM16} which pointed out the connection of 
multivariate orthogonal polynomials to the study of integrable systems.   
   
Although ChaptGPT, Gemini 3,  Deepseek, and etc can help find much more literature on the study of 
multivariate orthogonal polynomials (MOP), 
we shall only explain some works closely related to our study in this paper.   
Let us start with the two monographs written by C. Dunkl and Y. Xu \cite{DX01} and \cite{DX14} 
on explaining how to write down formulas for orthogonal polynomials in multivariate 
settings and their properties.  Many formulas for  orthonormal 
polynomials over different standard domains such as the unit ball in $\mathbb{R}^n, n\ge 
2$, the simplex in $\mathbb{R}^n, n\ge 2$ and  $[-1,1]^n, n\ge 2$ were presented. 
For a general domain, explicit formulas for
orthogonal polynomials were given in terms of the moment functional ${\cal L}$. 
See Theorem 3.2.12 in \cite{DX14}.  We refer to these two monographs for details. 
Next in \cite{OST20}, Olver, Slevinsky, Townsend published a survey paper of 120 pages in 2021 
on fast algorithm of orthogonal polynomials including various versions of orthogonal polynomials 
over a standard  triangle.  In addition, we can find many interesting results, e.g. 
recurrence relations, and many interesting literature, e.g. Gauss quadrature in \cite{OST20}.  
Another recent survey paper  \cite{X21} summarizes several explicit formulas and many 
properties of orthogonal polynomials of several variables. 
Furthermore, see the classic literature, e.g. \cite{HS56},  
\cite{CH87}, \cite{S71}, and \cite{C97} for more formulas for cubature. Additional cubature formulas were 
added into literature, see \cite{DVX09}, \cite{HKA12}, and \cite{P16}, and etc..    

It is known that there are several difficulties in obtaining a clear and clean formula 
of multivariate orthogonal/orthonormal polynomials (MOP). 
For convenience, we shall use MOP in the rest of the paper.  One of them is the 
order of multi-indices to numerate these multivariate polynomials, e.g. the graded 
lexicographic order or  lexicographic order, which lead to two different systems of 
orthogonal polynomials. Another difficulty is a bounded domain in the 2D setting has infinitely 
many types. 
Indeed, when extending the concept of a bounded domain $[a, b]$ in the univariate setting to 
the bivariate  
setting, there are infinitely many different types of bounded domains, e.g. any polygonal domain 
with $n$ sides, $n\ge 3$. Each bounded domain has its own formulas of MOP. 
When we formulate the orthogonal conditions in terms of determinants,  it is difficult 
to find the determinant of matrices of large sizes with variables $x, y, ...$ as 
their entries.  Another difficulty is to compute the moment functional ${\cal L}$ over domains 
of irregular shape as one can use the moment functionals to generate MOP.  
In \cite{MOPS18}, the researchers have implemented some of these formulas in MAPLE.  
In \cite{DX14}, the researchers also studied the zeros of orthogonal polynomials. Many 
interesting phenomena on the common zeros of orthogonal polynomials of degree exactly $d$ 
were found. For example, in the bivariate setting, orthogonal polynomials $P_{d,j} \in 
\mathbb{P}_d, j=0, \cdots, d$ have no common zeros when $d$ is even and have one common 
zero when $d$ is odd over a centrally symmetric domain $\Omega$.  

Based on the difficulties mentioned above,  
MOP over a general polygonal domain, e.g. an L-shaped domain 
in $\mathbb{R}^2$ or a general polyhedron in $\mathbb{R}^3$ are hard to compute. Carrying 
out the computation based on the given formula in \cite{DX14} looks straightforward, but tedious. 
Due to the advances of computer and computational techniques, it is possible to derive 
algorithms to let computer find these 
polynomials for us.  That is, one can compute them by using numerical integrations 
and numerical computation of Gram-Schmidt orthogonalization. 
Then a serious difficulty is the accuracy 
of numerical integrations. 
Also carrying out numerical integration over an arbitrary polygon may not be 
so easy. All these difficulties can be overcome by using multivariate
splines which were summarized in \cite{LS07} and the references therein. 
See also \cite{A2000}, \cite{ALW06}, \cite{L2025}, and etc.. 
Mainly, there is an exact formula for 
inner product of polynomials over any simplex in $\mathbb{R}^n, n\ge 2$. 
We refer to \cite{CL91} and \cite{LS07} and 
hence, the inner product of two polynomials over any polygonal/polyhedral domain can 
be computed exactly and hence, the integration has no computational error up to a 
machine epsilon.   
It is the main purpose of this paper  to present a computational 
method for  finding MOP over any polygonal domains/polyhedral domains. 
Once we have these polynomials, we will discuss how to use them for generating numerical 
quadrature  
over arbitrary polygons.  For convenience, we shall focus on the bivariate setting and leave the 
discussion of trivariate setting in a future publication. 

Indeed,  when a domain $\Omega$ of interest is a polygon, 
we first use a triangulation, i.e.  collection of triangles to decompose the domain into a 
union of triangles which leads to a triangulation $\triangle$.  
Then we can use bivariate/trivariate splines
which are polynomials over each triangle with global smoothness $r$ over $\Omega$. 
When the smoothness order $r$ of these splines is high enough, say the order of the 
smoothness is equal to the degree of splines,  the spline functions are polynomials.  
More precisely,  given a polygonal domain $\Omega$, let $\triangle$ be a triangulation of 
$\Omega$.  We define by 
\begin{equation}
\label{SS}
S^r_d(\triangle) =\{ s\in C^r(\Omega), s|_{T}\in \mathbb{P}_d, T\in \triangle\}
\end{equation}
the spline space of degree $d$ and smoothness $r$ over  $\triangle$, where $d\ge r$. We 
refer to \cite{LS07} for its properties and \cite{ALW06} and 
\cite{S15} for its computation and applications. Traditionally, we used multivariate splines 
of degree $d$ and $r$ with $d>r$ for numerical solutions of PDE's, data interpolating and 
fitting, surface reconstructions, and construction of wavelets/framelets. See, e.g. \cite{GL13}, \cite{S19}, \cite{LL22}, \cite{L2025}, and etc.. 
In this paper, we use $r=d$ and use the spline functions in $S^r_d(\triangle)$ for 
constructing orthogonal polynomials. In this setting, the above spline space 
$S^d_d(\triangle)$ is the space of polynomials of degree $d$. Our MATLAB implementation of
$S^r_d(\triangle)$ gives us an exact representation of polynomials over any polygonal domain 
$\Omega$.

Instead of using a Gram-Schmidt orthogonalization procedure,  
we reformulate the construction problem into a minimization problem and then 
solve it numerically. 
We shall explain the minimization in the next section. For convenience, 
let us first show quadratic orthonormal polynomials over an L-shaped domain with its 
vertices $[0,0;2, 0;2, 1;1, 1;1, 2;0, 2]$ in 
Figure~\ref{quadraticop0}.  

\begin{example}
Consider a polygon $P$ which is an L-shaped domain vertices $[0,0;2,0;2, 1;1, 1;1, 2;0, 2]$.  For $d=2$, 
we find bivariate quadratic orthogonal polynomials. Let us present
their graphs  in Figure~\ref{quadraticop0}.  These coefficients in power format will 
be given in a later section.  
	
\begin{figure}[htpb]
		\includegraphics[width=0.3\linewidth]{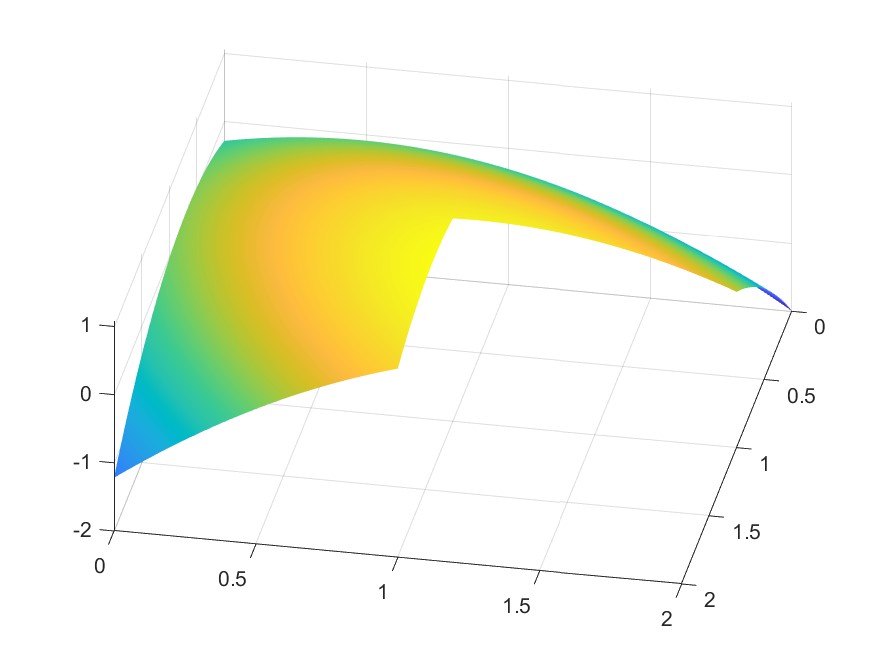}
		\includegraphics[width=0.3\linewidth]{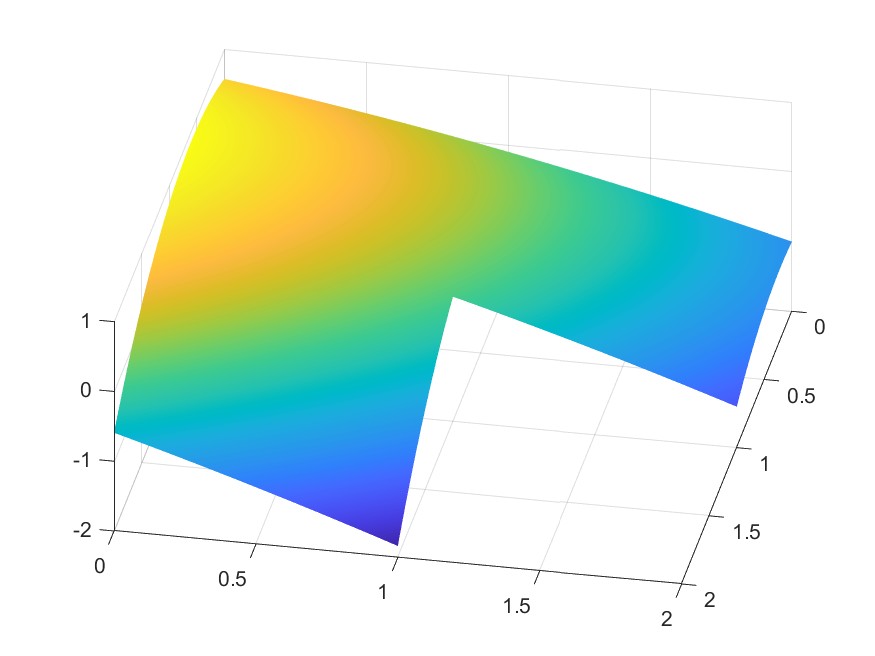}
		\includegraphics[width=0.3\linewidth]{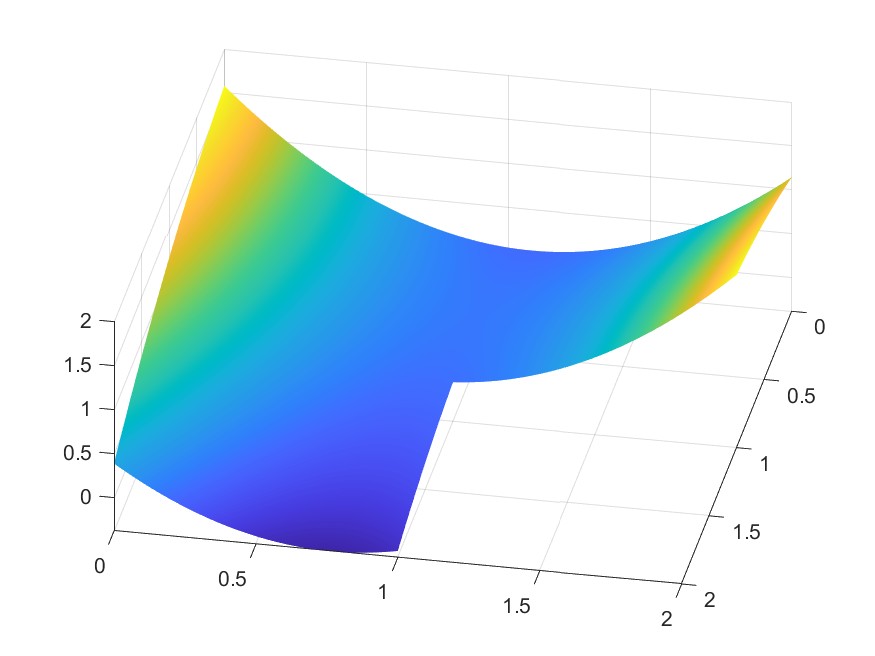}
		\includegraphics[width=0.3\linewidth]{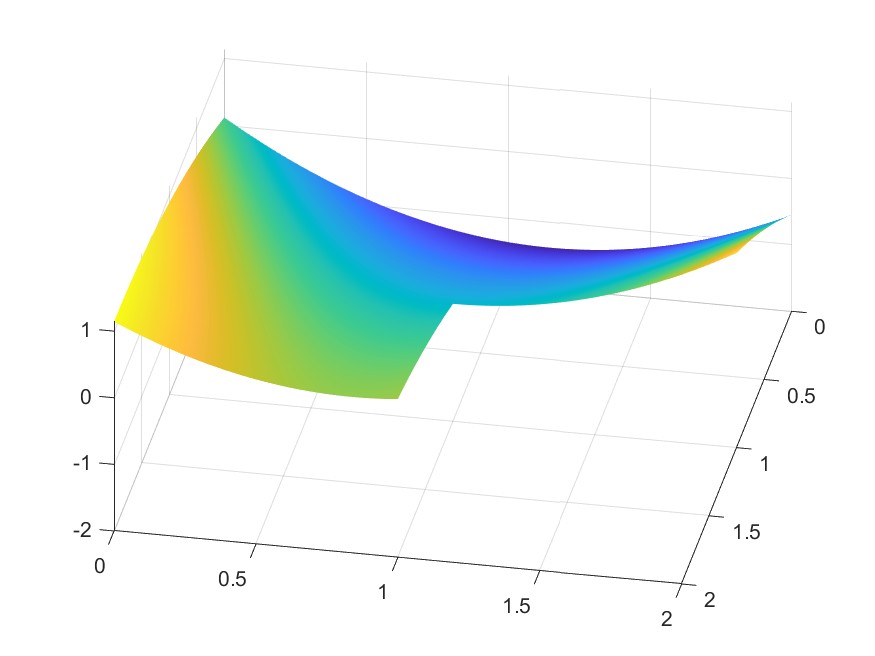}
		\includegraphics[width=0.3\linewidth]{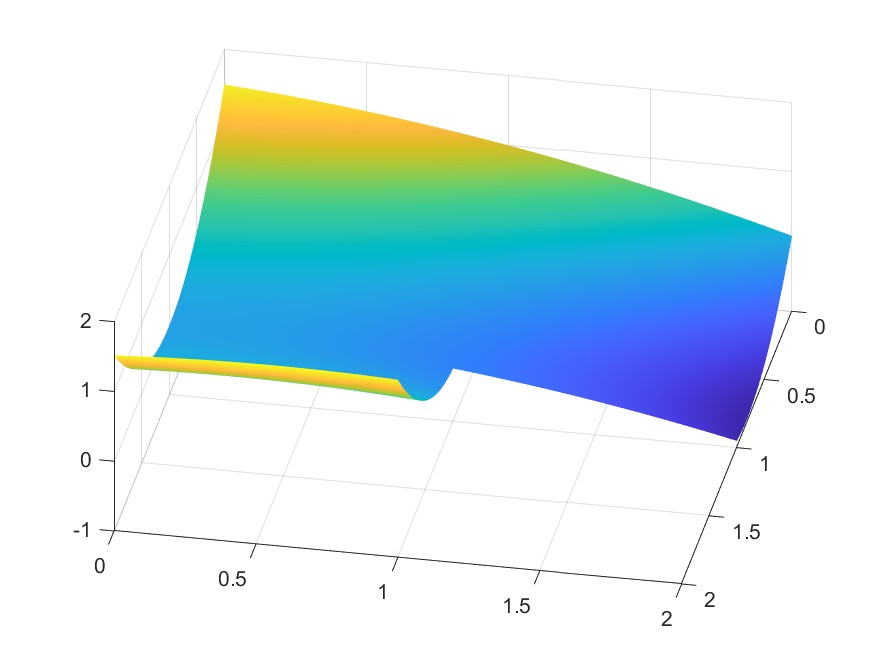}
		\includegraphics[width=0.3\linewidth]{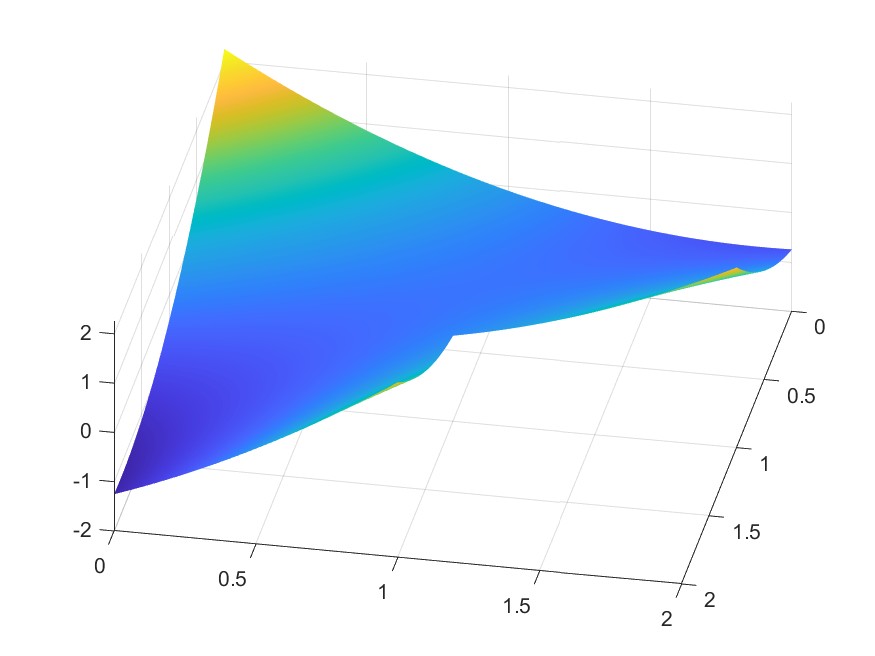}
		\caption{Quadratic Orthonormal Polynomials over L-shaped Domain \label{quadraticop0}}
	\end{figure}
\end{example}
 
 In addition to our computational algorithm, we also add a verification step. 
That is, we always resize the domain of interest
into $[0, 1]^2$ and use the values of the  orthonormal polynomials from our algorithm over
$1001^2$ equally spaced points of $[0, 1]^2$ and inside the given domain $\Omega$ 
to compute
a Riemann sum and check whether the matrix of the inner product products of these polynomials 
is close  to the identity or not.  Many examples will be given and their zeros, 
especially,  common zeros will be shown. 
The associated MATLAB program will be available from the author upon request. 

There are a few other contributions in this paper. One significant result is that   
one point quadrature  is exact for many more functions than linear polynomials. 
It is exact for not only odd functions, but also many even functions over $[-1, 1]$ as well.
Another significant result is a new strategy for numerical integration. 
Indeed, we introduce a polynomial reduction procedure based on orthogonal polynomials of odd degree
(similar for orthogonal polynomials of even degree). Then the integration of a polynomial of any 
degree is  the integration of the residual quadratic polynomial or linear polynomial after the 
polynomial reduction procedure.  See Theorems~\ref{mainresult1} and \ref{mainresult2} for details.
These lead to a brand-new integration approach: computing the integrals of linear polynomials or
quadratic polynomials by a brute force method beforehand once for all, the integration of any 
polynomial is done by doing a polynomial reduction by either orthogonal polynomials of even degree or odd degree.  We further use the
polynomial interpolation to obtain some numerical quadrature formulas which has a higher polynomial
precision using a less number of function values.

The paper is organized as follows. We first revisit the construction of Gaussian quadrature and 
the well-known one point formula in \S 2. Then we use the zeros of Legendre polynomials of odd
degree to give an explanation for the one point formula and 
show that the one point formula will be exact for all odd 
functions and many even functions. Next we explain a construction to obtain Gauss type
quadrature based on Legendre polynomials of odd degree. 
This motivates us to construct orthogonal polynomials in the 
multivariate settings.  In \S 3,   we explain how to  compute an orthogonal basis for 
polynomial space $\mathbb{P}_d$ for $d\ge 1$ by using bivariate splines. Two computational 
procedures will be outlined, one is to compute an orthonormal basis for $\mathbb{P}_d$ 
and the other one is 
to construct an additional basis of orthonormal polynomials in $\mathbb{P}_{d+1} \ominus \mathbb{P}_d$.    
Many numerical examples will be given, including the zeros of these orthogonal polynomials.  We 
numerically verify a classic result of orthogonal polynomials in $\mathbb{R}^2$ that there is no 
common zeros when degree $d$ is even and there is one common zeros when the degree $d$ is odd as
proved in \cite{DX14} when the domain is centrally symmetric. Furthermore, in \S 4 we explain a 
polynomial reduction approach based on orthogonal polynomials of odd degree $d\ge 3$ with 
the residual term $R$ which is a 
quadratic polynomial. Then the  
integration of a polynomial function of any degree is equal to the integration of its quadratic residual $R$ which leads to a one-point formula when the residual $R$ is a linear polynomial. 
In the same fashion, we can reduce a polynomial $p$ by using orthogonal polynomials of even 
degree with a residual linear polynomial.  The the integration of $p$ is the same as the integration of the residual linear polynomial. Since we can pre-calculate the integration of 
$x, y, 1$ over the domain of interest, the integration of $p$ is the coefficients of the residual linear polynomial with the pre-calculated values.  
In addition, we explain a polynomial interpolation method to obtain a quadrature based 
on function values. Furthermore, the study leads to some quadrature with a higher polynomial 
precision by using a less number of function values over triangular/rectangular/hexagonal 
domains.  Finally, in \S 5 we end the paper
with several remarks on the generalizations and extensions of the construction in this  paper. 
In particular,  the computation in \S 3 and new integration methods in \S 4 
can be extended into the 3D settings and the details are left to a future publication. 
Also, the construction of orthogonal polynomials over $\Omega$ 
in  this paper can be generalized to the construction of orthogonal spline functions over 
$S^r_d(\triangle)$ for  any $r<d$.

\section{Preliminary and Motivations}
Let us recall the ideas for Gauss quadrature over $[-1, 1]$ first. 
It uses Legendre orthogonal polynomials $P_n$ of degree $\le n$ with leading coefficient $1$ for 
$n=0, 1, 2, \cdots, $. These polynomials $P_n$'s have some nice properties: they satisfy a 
simple recurrence relation, all the zeros of $P_n$ are inside $[-1, 1]$ and all polynomial
$P_{2n-1}$ of odd degree has a zero at $0$.
Thus, for any polynomial $P$ of degree $2n-1$,  we do a long division to get
\begin{equation}
\label{Gauss}
 P (x) = q(x) P_{n}(x) + R(x),
 \end{equation}
 where $q(x)$ is the quotient and $R(x)$ is the residual with degree $\le n-1$.  Due to the 
orthogonality, we have
 \begin{equation*}
 \int_{-1}^1 P(x)dx = \int_{-1}^1  q(x) P_{n}(x) dx + \int_{-1}^1 R(x)dx = \int_{-1}^1 R(x)dx.
 \end{equation*}
 If we use the zeros $x_1, \cdots, x_n$ of the polynomial $P_n$ as interpolation locations, 
$R(x)$ can be written in terms of Lagrange basis function, i.e. 
 \begin{equation*}
  R(x) = \sum_{j=1}^n R(x_j) \prod_{i=1\atop i\not=j}^n \frac{(x-x_i)}{(x_j-x_i)}.
  \end{equation*}
  Note that $R(x_j) = P(x_j)$ due to the zeros of $P_n$.  Then integration 
  \begin{equation*}
  \int_{-1}^1 R(x)dx = \sum_{j=1}^n P(x_j) \int_{-1}^1 \prod_{i=1\atop i\not=j}^n \frac{(x-x_i)}{(x_j-x_i)}dx.
  \end{equation*}
  By choosing the weights $c_j = \int_{-1}^1 \prod_{i=1\atop i\not=j}^n \frac{(x-x_i)}{(x_j-x_i)}dx$, the quadrature 
  \begin{equation*}
  \int_{-1}^1 P(x)dx = \sum_{j=1}^n P(x_j) c_j
  \end{equation*}
is exact for all polynomials of degree $2n-1$.  We now iteratively apply the ideas above to 
write any polynomial $P$ of any degree $n\ge 3$, say $2d-1\le n \le 2d$,   
\begin{equation}
\label{laisidea}
P(x) = q_1(x) P_{2d-1}(x) + q_2(x) P_{2d-3}(x) + \cdots + q_d(x) P_1(x) +R(x)  
\end{equation}
by using Legendre polynomials of odd degree, 
where $q_i(x)$ are linear polynomials and so is the residual polynomial $R(x)$. 
By using the orthogonal properties of Legendre polynomials, it follows that 
\begin{equation}
\label{laisidea2}
\int_{-1}^1 P(x)dx = \sum_{j=1}^d \int_{-1}^1q_j(x) P_{2(d-j)+1}(x)dx + 
\int_{-1}^1R(x)dx = \int_{-1}^1 q_d(x) P_1(x)dx + \int_{-1}^1 R(x)dx.
\end{equation}
If $q_d(x)\equiv c$, i.e. $q_d(x)$ is a constant function,  we will have $\int_{-1}^1 
P(x)dx  = \int_{-1}^1R(x)dx.$ Recall that 
$0$ is the common zero of all Legendre polynomials of odd degree. So $R(0)=P(0)$ and 
$\int_{-1}^1 R(x)dx= 2P(0)$ and hence, $\int_{-1}^1P(x)dx= 2 P(0)$ which is the one point 
quadrature. We write this case as in the following theorem. 
\begin{thm}
\label{onepointquadraure}
For any polynomial $P$ of degree $n\ge 1$, write $P$ in the format of (\ref{laisidea}). 
If the linear polynomial $q_d$ is a constant, then 
\begin{equation}
\label{mainresult0}
\int_{-1}^1P(x)dx= 2 P(0).
\end{equation}
\end{thm}
\begin{proof}
When $q_d= c$, we can combine $q_d(x) P_1(x)$ and $R(x)$ together as a new linear 
polynomial $\widetilde{R}$. The two integrals on the 
right-hand side of (\ref{laisidea2}) are $\int_{-1}^1 \widetilde{R}(x)dx = 
2\widetilde{R}(0)$.   Note that 
$P_1(0)=0$ and hence, we still have $\widetilde{R}(0)=R(0)=P(0)$. This completes the 
proof.  
\end{proof}

The quadrature formula in (\ref{mainresult0}) is a well-known one point quadrature 
formula. It is known that the quadrature is exact for linear polynomials. 
Now we can see that it is exact for a lot of more functions. Let us define the following space 
\begin{equation}
\label{laispace}
\mathcal{L}([-1,1])= \hbox{span}\{ \sum_{j=2}^d q_j(x) P_{2j-1}(x) +\ell(x) : 
\hbox{ all } q_j \hbox{ and } \ell(x) \hbox{ are linear polynomials }, d\ge 2\}.
\end{equation}
It is easy to see that the one point quadrature is exact for all functions in 
$\mathcal{L}([-1, 1])$.  
Note that $P_{2j-1}, j\ge 1$ are odd functions and clearly belong to $\mathcal{L}$. Since 
any odd continuous functions 
$f$ can be approximated by odd degree Legendre polynomials, the one point quadrature 
formula is exact for all odd functions. If we choose $q_j(x)=c_jx$ for all $j\ge 2$ and 
$\ell(x)=c_0$, 
$f(x) =\sum_{j=2}^d x P_{2j-1}(x)+\ell$ is an even function which is also in 
$\mathcal{L}([-1,1])$.    Thus, $\mathcal{L}[-1,1]$ in (\ref{laispace}) 
is a very large space. It is 
clear that $\mathcal{L}[-1,1]$ is not the whole space
$C([-1,1])$ since $x^2\not\in \mathcal{L}([-1,1])$ as the one point quadrature  does not 
work.   It is natural to add all quadratic polynomials into $\mathcal{L}[-1,1]$ and define
\begin{equation}
\label{Laispace2}
\mathbb{L}[-1,1]= \hbox{span}\{ \mathcal{L}[-1,1], \mathbb{P}_2\},
\end{equation}
where $\mathbb{P}_2$ is the space of all quadratic polynomials. We are now ready to establish 
 another main result in this section.
\begin{thm}
\label{mjlai12233035}
The closure of the space $\mathbb{L}[-1,1]$ in (\ref{Laispace2}) 
under the maximal norm is $C[-1,1]$. 
\end{thm}
\begin{proof}
As we have already seen that all polynomials of odd degree are in $\mathcal{L}[-1,1]\subset 
\mathbb{L}[-1,1]$. Also, we know all polynomimals of even degree $d\ge 4$ are in $\mathcal{L}[-1,1]$
and hence, in $\mathbb{L}[-1,1]$. Now $\mathbb{P}_2\subset \mathbb{L}[-1,1]$. Then the closure 
of the linear span of all polynomials of any degree is $C[-1, 1]$ by the well-known Weierstrass theorem
and hence, the closure of $\mathbb{L}[-1,1]$ is $C[-1, 1]$. 
\end{proof}

The result above tells us that the integration of any continuous function $f$ can be done by first  
approximate $f$  by using the functions in $\mathbb{L}[-1,1]$ within a given tolerance $\epsilon$, 
say $f \approx f_1+ f_2$ with $f_1\in \mathbb{L}[-1,1]$ and $f_2\in \mathbb{P}_2$,  
and then the integration is $\int_{-1}^1 f_2(x)dx$ which is trivial. It follows 
\begin{corollary}
Let $f\in C[-1, 1]$ be a given continuous function. For any given tolerance $\epsilon>0$, there 
exists an approximation $f_1+f_2$ with $f_1\in \mathbb{L}[-1,1]$ and $f_2\in \mathbb{P}_2$ such 
that $\|f- f_1-f_2\|_\infty\le \epsilon$. Then 
$$
\int_{-1}^1 f(x)dx = \int_{-1}^1 f_2(x)dx.
$$
If $f_2(x)= ax^2+ bx+ c$ then $\int_{-1}^1 f(x)dx= 2a/3+2c$. 
\end{corollary}
In the same fashion as above, we can use all Gauss-Legendre orthogonal polynomials of even degree to form 
the space 
\begin{equation}
\label{Laispace3}
\mathcal{L}_e[-1, 1]=\{ \sum_{k=1}^d c_k(x) P_{2k}(x) + \ell(x), 
\hbox{ all } c_k(x) \hbox{ and } \ell(x) \hbox{ are linear polynomials }, d\ge 1\}.
\end{equation}
By using the same arguments above, we have
\begin{thm}
\label{mjlai12232025}
The closure of $\mathcal{L}_e[-1,1]$ under the maximal norm is $C[-1,1]$. 
\end{thm}
\begin{proof}
Clearly, all polynomials of even degree are in $\mathcal{L}_e[-1,1]$ in (\ref{Laispace3}) 
as all the linear polynomials
$c_k(x)$ and $\ell(x)$ are constants.  For any polynomial $f$ of odd degree, say 
$$
f(x)= \sum_{k=0}^d a_k x^{2k+1}, 
$$
we can write $f_1(x)= f(x)- a_dx P_{2d}(x)$ which is a polynomial of odd degree $2d-1$ with 
leading 
coefficient $\tilde{a}_{d-1}$, where $P_{2d}$ is the Gauss-Legendre polynomial of degree $2d$ with 
leading coefficient $1$. We refer to \cite{L78} for MATLAB code to compute these polynomials. Then
we let $f_2(x)= f_1(x)- \tilde{a}_{d-1} xP_{2d-2}(x)$ which is a polynomial of odd degree $2d-3$. 
We repeat the iteration until the remainder $f_{d}$ is a linear polynomial. This shows that all 
polynomials of odd degree are in  $\mathcal{L}_e[-1,1]$. Hence, the closure of 
$\mathcal{L}_e[-1,1]$ in the maximal norm is $C[-1, 1]$.   
\end{proof}
It follows 
\begin{corollary}
Let $f\in C[-1, 1]$ be a given continuous function. For any given tolerance $\epsilon>0$, there 
exists an approximation  $f_1\in \mathbb{L}_e[-1,1]$ with a linear term $\ell_{f,\epsilon}$ such 
that $\|f- f_1\|_\infty\le \epsilon$. Then 
$$
\int_{-1}^1 f(x)dx = \int_{-1}^1 \ell_{f,\epsilon}(x)dx.
$$
If $\ell_{f,\epsilon}(x)= bx+ c$, then $\int_{-1}^1 f(x)dx=2c$. 
\end{corollary}

We now discuss how to use the expansions of Gauss-Legendre polynomials of odd degree or even 
degree 
for numerical integration of continuous functions. For any function $f\in C[-1, 1]$, we can write
$f(x)= f_o(x)+ f_e(x)$ with $f_o(x)=(f(x)-f(-x))/2$ and $f_e(x)= (f(x)+f(-x))/2$. Since the integration
of the odd part $f_o$ over $[-1, 1]$ is zero, we only need to consider $\int_{-1}^1 f_e(x)dx$. 
For any given tolerance $\epsilon>0$, we need to find a polynomial $p_{f, \epsilon}$ of even degree
such that $\|f_e(x)- p_{f,\epsilon}(x)\|_\infty \le \epsilon$. How to find $p_{f,\epsilon}$ is still 
a good research problem in general, in particular, when $f$ may not be very smooth.  
Let us discuss how to convert an even degree polynomial $p_{2d}$ in the expansion of Gauss-Legendre polynomials of odd degree. Similarly, we can expand $p_{2d}$ in Gauss-Legendre polynomials of even degree. 
Without loss of generality, we may write a polynomial function $p_{2d}(x) = \sum_{j=0}^d f_j x^{2j}$.  
 we can derive an algorithm to find coefficients $q_j, j=1, \cdots, d$  such that 
 \begin{equation*}
 p_{2d}(x) = \sum_{j=2}^{d} q_jx P_{2j-1}(x)+ q_1 x P_1(x) + p_{2d}(0), 
\end{equation*}
 where $q_j$ are real numbers for all $j=1, \cdots, d$. Indeed, as discussed above, 
we can iteratively solve for $q_j,j=d-1, \cdots, 1$ if we know $f_j, j=1, \cdots, d$ as $f_0=p_{2d}(0)$. 
Then $\int_{-1}^1 p_{2d}(x)dx = 2q_1/3+ 2p_{2d}(0)$. 
Another way to find $q_1$ is to use polynomial interpolation method. Also, in practice we do not know
$p_{2d}$. Instead, we only know $f$. Let us solve the following interpolation problem: letting 
$x_i, i=1, \cdots, d$ be distinct points over $(0, 1]$, find $q_1, \cdots, q_d$ 
satisfy the following interpolation conditions:
 \begin{equation*}
  \sum_{j=2}^{d} q_j P_{2j-1}(x_i)+ q_1  P_1(x_i) =(f(x_i)- f(0))/x_i, \qquad i=1, \cdots, d.  
\end{equation*}
Since Gauss-Legendre polynomials $P_{2j-1}, j=1, \cdots, d$ are linearly independent, 
such a linear system has a unique solution. As we need $q_1$, we simply use the Cramer rule to get 
\begin{equation}
\label{solution}
q_1 = \hbox{det}[ f(:), P_3(:), \cdots, P_{2d-1}(:)] / \hbox{det}[ P_1(:), P_3(:), \cdots, P_{2d-1}(:)],
\end{equation} 
where 
$$
\hbox{det}[ P_1(:), P_3(:), \cdots, P_{2d-1}(:)] = \begin{bmatrix} P_1(x_1) & P_3(x_1) & \cdots & 
P_{2d-1}(x_1)\cr
P_1(x_2) & P_3(x_2) & \cdots & P_{2d-1}(x_2)\cr
\vdots & \ddots & \cdots & \vdots \cr
P_1(x_d) & P_3(x_d) & \cdots & P_{2d-1}(x_d)\cr
\end{bmatrix} 
$$
and similar for $\hbox{det}[ f(:), P_3(:), \cdots, P_{2d-1}(:)]$. 
  
  \begin{thm}
  \label{betterthanGaussianQuadrature}
For any continuous even function $f$ and for any set of distinct points $x_1, \cdots, x_d\in (0, 1]$, 
compute $q_1$ as in (\ref{solution}).  Then  the quadrature 
  \begin{equation}
  \label{Q1}
  \int_{-1}^1 f(x) dx \approx \frac{2}{3}q_1+ 2f(0)
\end{equation}
which  is exact for all even polynomial functions $f$ of degree $\le 2d$.   For any continuous function
$f$, let $f= f_e+ f_o$ with even part $f_e$ and odd part $f_o$. We compute $q_{1,e}$ with 
$f_e$ in the place of $f$ in (\ref{solution}). Then 
  \begin{equation}
  \label{Q2}
  \int_{-1}^1 f(x) dx \approx \frac{2}{3}q_{1,e}+ 2f(0).
\end{equation}
 \end{thm}
 \begin{proof}
 When $f$ is a polynomial of even degree $\le 2d$, the solution vector $(q_1, q_2, \cdots, q_d)$ gives
 us the exact representation of $f$ in terms of the Gauss-Legendre polynomials of degree $\le 2d$. Then 
 $$
\int_{-1}^1 f(x)dx = \sum_{j=2}^d \int_{-1}^1 q_j x P_{2j-1}(x)dx + \int_{-1}^1 f(0)dx 
=\int_{-1}^1 q_1 x P_1(x)dx+2f(0)= \int_{-1}^1 q_1 x^2 dx +2f(0)
$$
which is the right-hand side of (\ref{Q1}). Similarly, $f(x)= (f(x)+f(-x))/2+ (f(x)-f(-x))/2$ with 
$f_e= (f(x)+f(-x))/2$ and $f_o= (f(x)-f(-x))/2$. Since $\int_{-1}^1 f_o(x)dx=0$, we use $f_e$ for 
$f$ in (\ref{Q1}) to get (\ref{Q2}) as $f_e(0)=f(0)$.  
 \end{proof} 



We remark that when $f$ is even, the quadrature (\ref{Q1}) uses only $d$ function 
values to make the quadrature to be exact for all polynomials of degree $2d$ which is 
slightly better than the well-known Gaussian quadrature mentioned in the beginning of the section.  
Also,  the quadrature (\ref{Q2}) does not need to find the zeros of Legendre polynomials and can 
be built based on any distinct $d$ points in $(0, 1]$.  
Finally, we remark that the set of distinct points can be chosen equally-spaced points over $[0, 1]$ with
spacing $1/(d+1)$. 

\begin{example}
Consider $d=4$. We know the Legendre polynomials  $P_1(x)=x$ and $P_3(x)= x^3-3/5 x$. When $f(x)= 
ax^4+bx^2+c$, we can rewrite  $f(x)=q_2 x P_3(x)+q_1 xP_1(x)+c$ with $q_2=a$ and 
$q_1=b+3a/5$.   Note that we can solve  $a, b$ from the equations  $a x_i^4+ bx_i^2 = 
f(x_i)-f(0)$ for $i=1, 2$. Taking any two distinct points in $(0, 1]$, say $x_1=1$ and 
$x_2=1/\sqrt{2}$, 
we have $a= 2f(1)- 4f(1/\sqrt{2})+2f(0)$ and $b=4f(1/\sqrt{2})- f(1)-3f(0)$.  So $q_1= 
4f(1/\sqrt{2})- f(1) +3(2f(1)- 4f(1/\sqrt{2}))/5 = \frac{1}{5} f(1)+ 
\frac{8}{5}f(1/\sqrt{2})$.  Thus, the quadrature 
$$
\int_{-1}^1 f(x) dx = \frac{2}{15}  f(1)+ 
\frac{16}{15}f(\frac{1}{\sqrt{2}}) +\frac{4}{5} f(0) 
$$   
for any  even function $f$.  This quadrature shows that one only uses $3$ function values 
to make the quadrature exact for all even polynomials of degree $4$.   
\end{example}

\begin{example}
We now give a MATLAB code for  even polynomials of any degree $2d$. That is, this code
will produce a quadrature based on $d+1$ nodes which is exact for even polynomial of degree
$\le 2d$. 
\begin{verbatim}
%This demo is to find integration of even function f over [-1, 1]. 
%It is written by Dr. Ming-Jun Lai on Nov. 1, 2024. 
%Choose your polynomial of even degree 2d, say.
d=4;
c=rand(d+1,1); 
f=@(x) c(1)+ c(2)*x.^2+ c(3)*x.^4+c(4)*x.^6+c(5)*x.^8;
%f=@(x) cos(x); %or you can replace polynomials by a function f=cos(x);
x=[1/d:1/d:1]; %nodes over (0, 1];
A=[x.^2];
for i=2:d
P=Legendre(2*i-1);
y=polyval(P,x);
A=[A; x.*y]; 
end
A=A';
b=f(x)-f(0); b=b';  
q=A\b; q1=q(1);
myvalue=2*q1/3+2*f(0);
exact= 2*c(1)+2*c(2)/3+2*c(3)/5+2*c(4)/7+2*c(5)/9;
%exact=2*sin(1);
[exact, myvalue]
\end{verbatim}
In the above, we use Legendre.m  which is a MATLAB code to generate Legendre polynomials with leading coefficient 1. 
The code is available from the author upon request or can be obtained from \cite{L78} and from ChaptGPT.   
I purposely put a function $f(x)=cos(x)$ in the command line above 
and one can check again the exact value $2\sin(1)$ which gives an excellent approximation 
with error $1.182169251379150e-09$.   The coefficients for the corresponding quadrature 
formula can be found \textcolor{blue}{ in } the first row of the inverse matrix of $A$.   
\end{example}

\section{Construction of Orthogonal Polynomials over Polygonal Domains}
We shall use bivariate  splines to construct  orthogonal polynomials
over polygons in detail.  Let $\Omega$ be a polygon  in $\mathbb{R}^2$ 
and let $\triangle$ be a triangulation of $\Omega$ 
(cf. \cite{LS07} for definition of triangulation). 
Recall the bivariate spline space on $\triangle$ defined in (\ref{SS}).  
In particular, we shall use the smoothness $r=d$ in our construction. 

Let us formulate the construction as a minimization problem as follows. We write any 
polynomial $p\in \mathbb{P}_d$ as a spline function
in $S^r_d(\triangle)$ with $r=d$.  Recall the Bernstein-B\'ezier form as follows: for 
each triangle $T=\langle \bfv_1, \bfv_2, 
\bfv_3\rangle$, let $b_1, b_2, b_3$ be the barycentric coordinates of point $\bfx$ based 
on $T$. We define  
$$
B^T_{ijk} = \frac{d!}{i!j!k!} b_1^i b_2^j b_3^k
$$
to be the polynomial of degree $(i,j,k)$ over $T$. As $b_1, b_2, b_3$ are linear 
polynomials of $\bfx$, $B_{ijk}^T$ is a polynomial of 
degree $d=i+j+k$.  Thus, $p=\sum_{i+j+k=d} c_{ijk}^T B_{ijk}^T$ for 
some coefficients $c_{ijk}^T$. Letting $\bfc^T=
(c_{ijk}^T, i+j+k=d)$ be the vector of $p$ with respect to $T$,  let 
$$
\bfc=( \bfc^T, T\in \triangle)
$$ 
be the coefficient vector of $p$ over the polygon $\Omega$ 
in terms of the Bernstein-B\'ezier form.  Note that $p\in C^r(\Omega)$ for 
$r\ge 0$ for all positive integer $r$. It is well-known that there is a smoothness matrix 
$H_r$ such that $H_r\bfc=0$ if and only if the spline 
function $s$ with coefficient vector $\bfc$ is in $C^r(\Omega)$ for each $r$. 
See \cite{LS07} for detail. The computation of these
smoothness matrices $H_r, r\ge 1$ are implemented in MATLAB according to \cite{ALW06} \textcolor{blue}{.}   
One of the main reasons that we are able to construct the orthogonal polynomials is the 
formulas for inner products of two polynomials which was given in \cite{CL91} 
or \cite{LS07}.      That is, we know the formulas for
\begin{equation}
\label{keyformula}
\int_T B^T_{ijk} B^T_{i',j',k'} dxdy = 
\frac{{i+i'\choose i}{j+j'\choose j}{k+k'\choose k}}{{2d\choose d}{2d+2\choose 2}}A_T
\end{equation}
for all $i+j+k=d$ and $i'+j'+k'=d$ (cf. \cite{CL91} or \cite{LS07}), where $A_T$ is the 
area of triangle $T$.  Note that these integrations are exact.  
Recall $D=(d+1)(d+2)/2$ is the dimension of $\mathbb{P}_d$.  There are 
D polynomials $p_i =  \hbox{diag}(B_{ijk}^T,i+j+k=d,  T\in \triangle)\bfc_i$, 
$i=1, \cdots, D$ such that 
$$
[p_1,\cdots,p_D]^\top  [p_1,\cdots,p_D]  = C^\top \hbox{diag}( \int_T B^T_{ijk} 
B^T_{i',j',k'} dxdy, i+j+k=d,i'+j'+k'=d,T\in \triangle)
C = I_\triangle,
$$
where $I_\triangle$ is the identity matrix of size $D\times D$ and 
$C$ be a matrix of size $D\times (N \times D)$ with $N$ being the number of
triangles in $\triangle$.  
Finding $C$ may not be easy. One usual approach is to start with constant $1$ with a 
normalization of
the area of $\Omega$ and then use the Gram-Schmidt orthogonalization procedure.  
In this way,  one can obtain a set of orthogonal polynomials  iteratively.  

Instead, in this paper, we consider the following minimization problem: letting 
\begin{equation}
\label{mainMatrix}
M_\triangle^d =
\hbox{diag}( \int_T B^T_{ijk} B^T_{i',j',k'} dxdy, i+j+k=d,i'+j'+k'=d,T\in \triangle)
\end{equation} 
be the diagonal block matrix of size $N\times D$ time $N\times D$, more precisely, $M^d_\triangle$ 
has $N$ blocks on the diagonal and each block is of size $D\times D$, 
\begin{equation}
\label{min}
\min  \frac{1}{2} \| C^\top M_\triangle^d C- I_\triangle\|_F,  \quad H_\triangle C =0, 
\end{equation}
where $H_\triangle$ stands for the matrix associated with all the smoothness conditions 
cross all interior edges
of the triangulation $\triangle$ up to the smoothness order $d$. We 
refer to  \cite{LS07} for the smoothness condition and to \cite{ALW06} for its MATLAB 
implementation. 

Certainly, the above minimization is not easy to solve.  
More studies are needed.  Let us simplify and reformulate it so that we can present 
a computational algorithm to construct MOP in the following two subsections. One is the construction of orthogonal 
polynomials on $\mathbb{P}_d$ and one is the construction 
in $\mathbb{P}_{d+1}\ominus \mathbb{P}_d$.

\subsection{Orthonormal Polynomials for $\mathbb{P}_d$}
Let us first simplify the minimization (\ref{min}) as follows. 
First we can compute $M^d_\triangle$ in (\ref{mainMatrix}) using our MATLAB code 
(cf. \cite{ALW06}). Then we compute the standard $L {\cal D} L^\top$ 
decomposition of $M^d_\triangle$ in MATLAB. That is, writing 
$M^d_\triangle = L {\cal D} L^\top$, we let $\sqrt{\cal D}$ be the square 
root of ${\cal D}$ and 
$C_1= \sqrt{\cal D} LC$.  Writing ${\cal H}=  H_\triangle L^{-1} \sqrt{\cal D}^{-1}$,  
we have obtained a new minimization:
\begin{equation}
\label{min2}
\min  \frac{1}{2} \| C_1^\top C_1- I_\triangle\|_F,  \quad {\cal H} C_1 =0. 
\end{equation}
Let us point out that the computation $M^d_\triangle = L {\cal D} L^\top$ is done blockwisely since 
$M^d_\triangle$ is a diagonal block matrix consisting of $N$ blocks of size $D\times D$ and the computation
can be done in parallel. 

It turns out that 
this minimization can be easily solved by using the 
null space property of ${\cal H}$.  In fact, $C_1=\hbox{null}({\cal H})$ 
in terms of MATLAB language. That is, the MATLAB command {\tt null} automatically returns $C_1$ 
which satisfies $C_1^\top C_1
=I_\triangle$. The discovery above leads to the following computational algorithm. See 
Algorithm~\ref{alg2} below. Before presenting the algorithm,  
let us show the contour plot of all orthonormal polynomials of degree $2$ over a standard 
rectangular domain $[0, 1]^2$. See Figure~\ref{quadraticop} for their 3D view and 
Figure~\ref{quadraticop2} for their contour view.   
More numerical examples of orthonormal polynomials over an L-shapded domain and more 
complicated domains will be shown in a section later. 

\begin{figure}[htpb]
\centering
	\begin{tabular}{ccc}
	\includegraphics[width=0.2\linewidth]{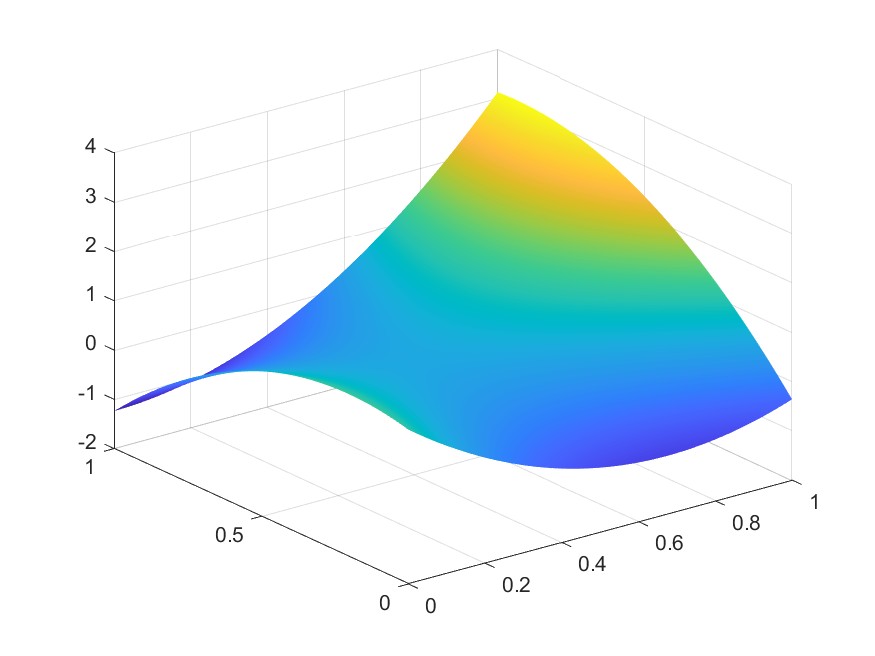}&
	\includegraphics[width=0.2\linewidth]{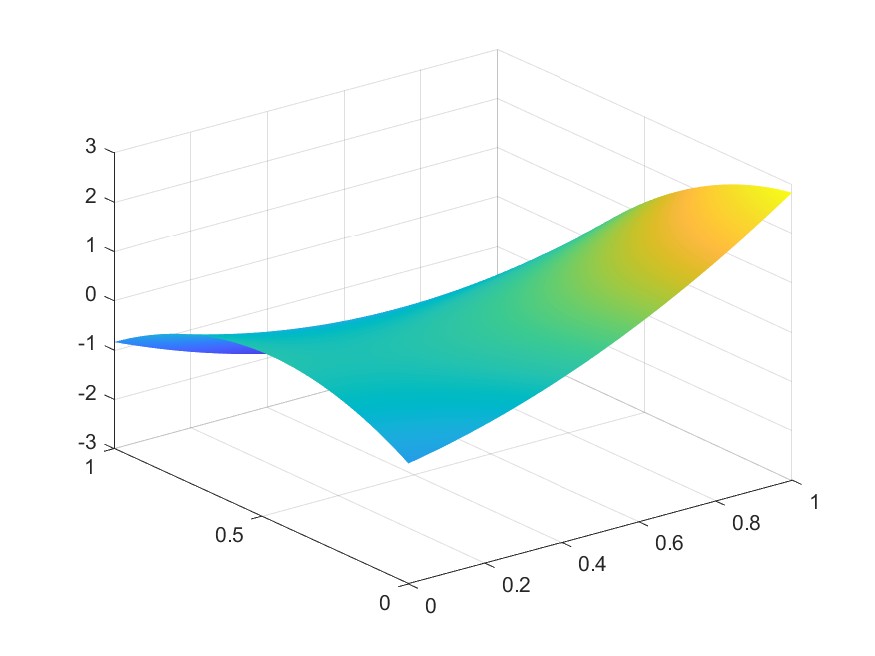}&
	\includegraphics[width=0.2\linewidth]{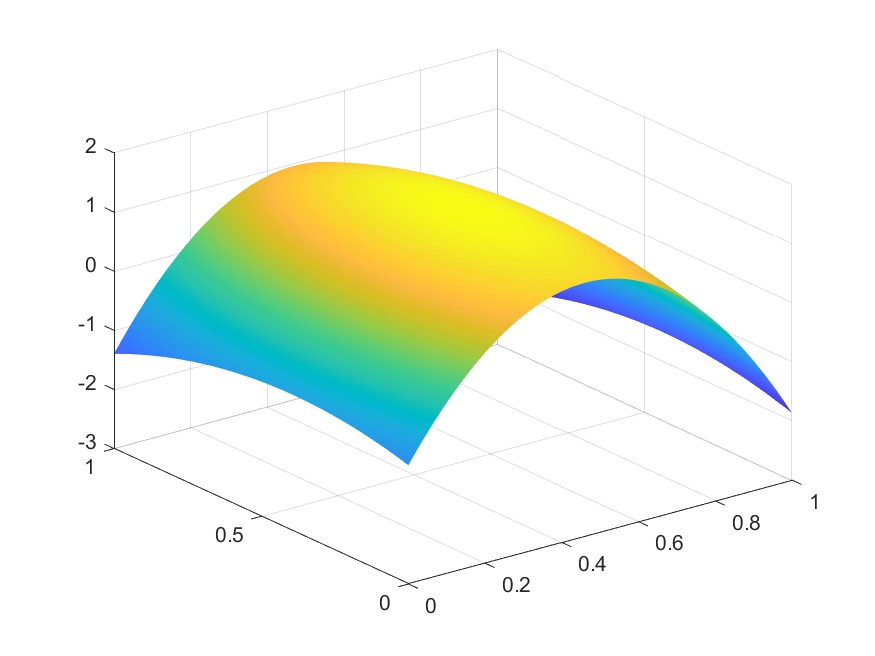}\\
	\includegraphics[width=0.2\linewidth]{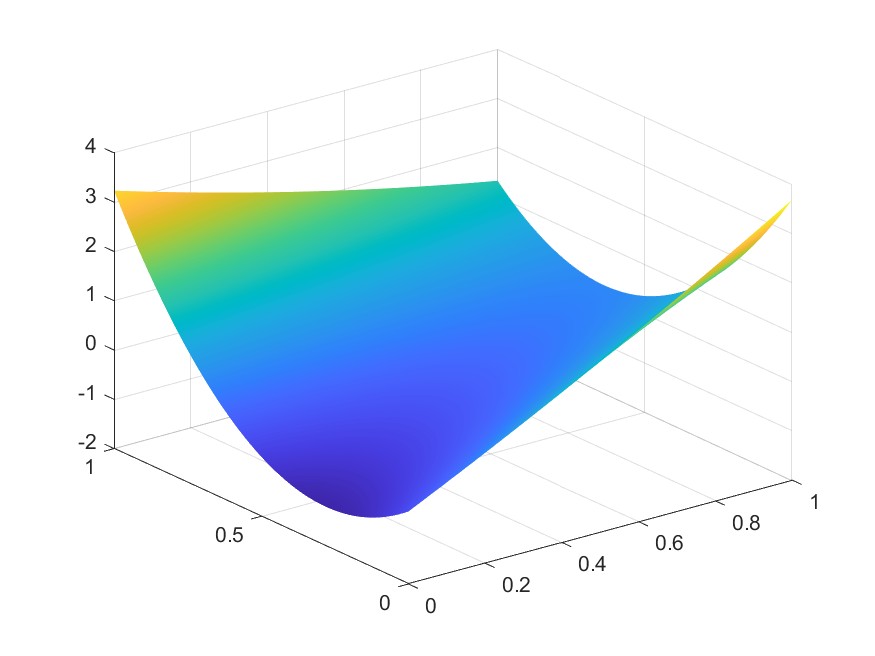}&
	\includegraphics[width=0.2\linewidth]{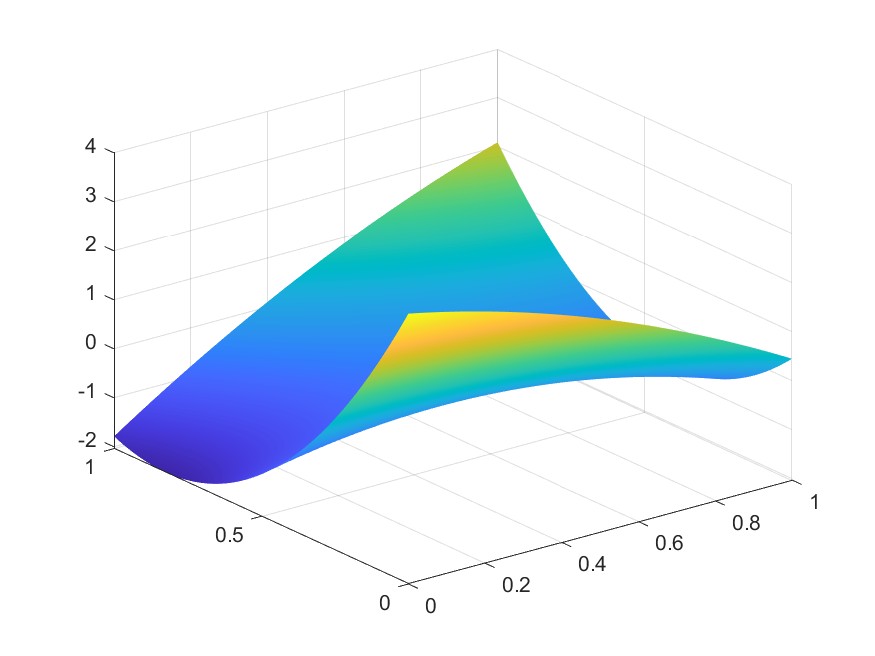}&
	\includegraphics[width=0.2\linewidth]{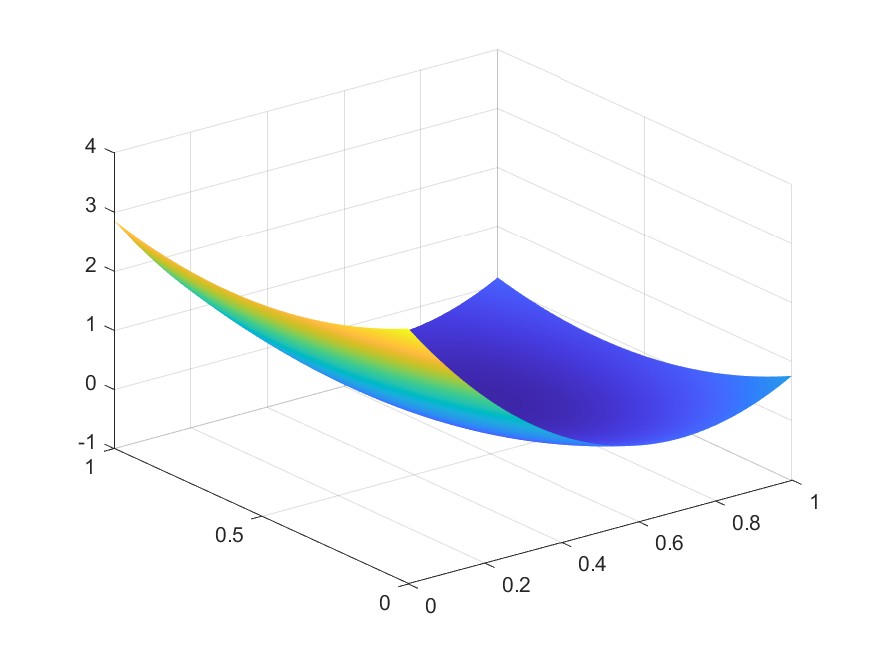}\\
\end{tabular}
	\caption{Quadratic Orthogonal Polynomials over Domain $[0, 1]^2$ \label{quadraticop}}
\end{figure}

\begin{figure}[htpb]
\centering
	\begin{tabular}{ccc}
		\includegraphics[width=0.3\linewidth]{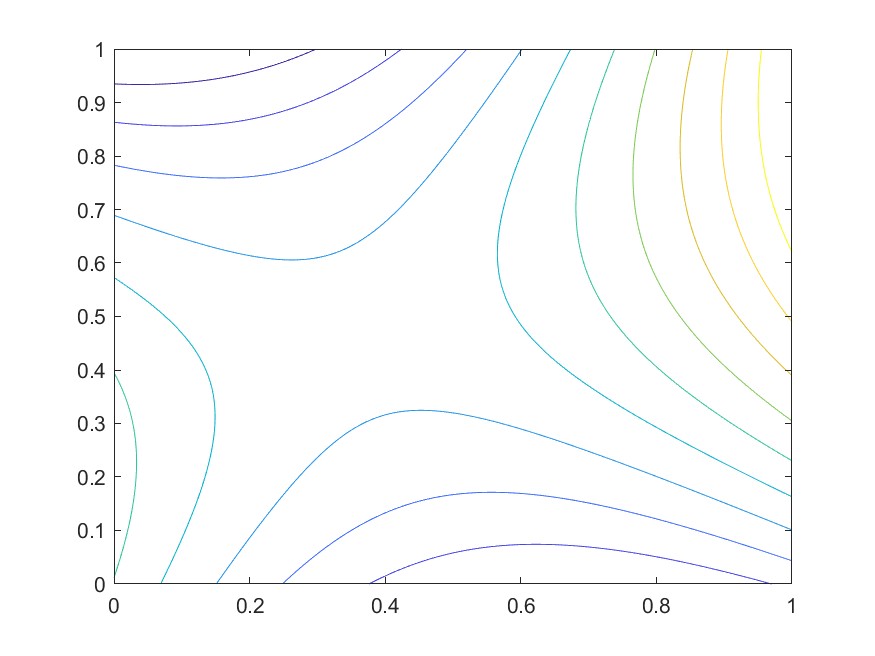}&
		\includegraphics[width=0.3\linewidth]{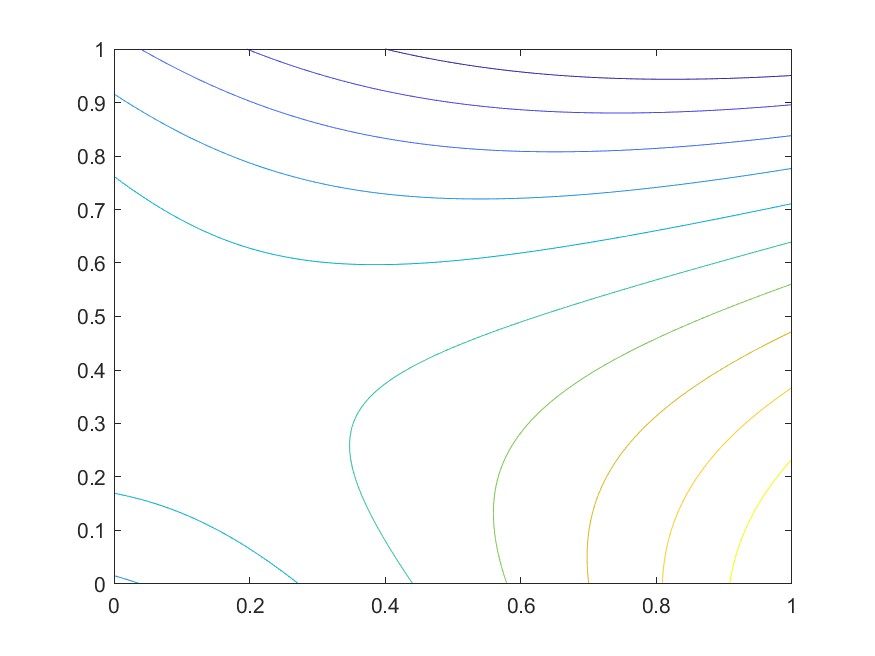}&
		\includegraphics[width=0.3\linewidth]{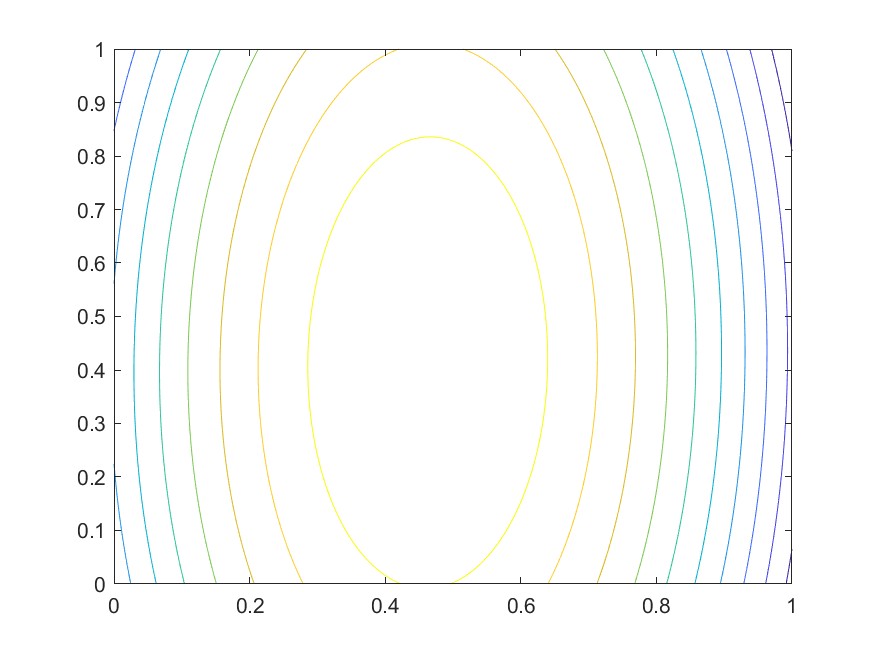}\\
		\includegraphics[width=0.3\linewidth]{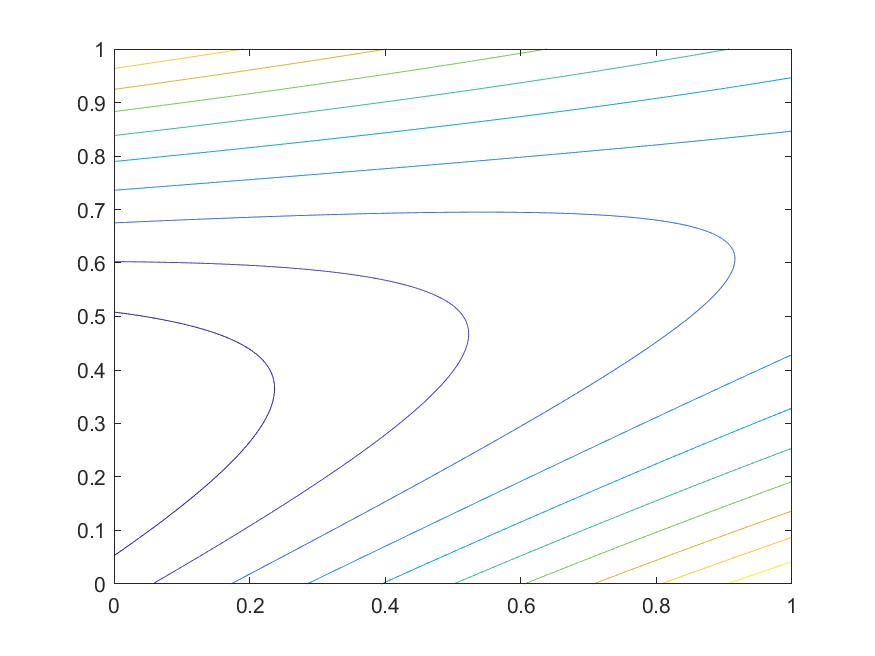}&
		\includegraphics[width=0.3\linewidth]{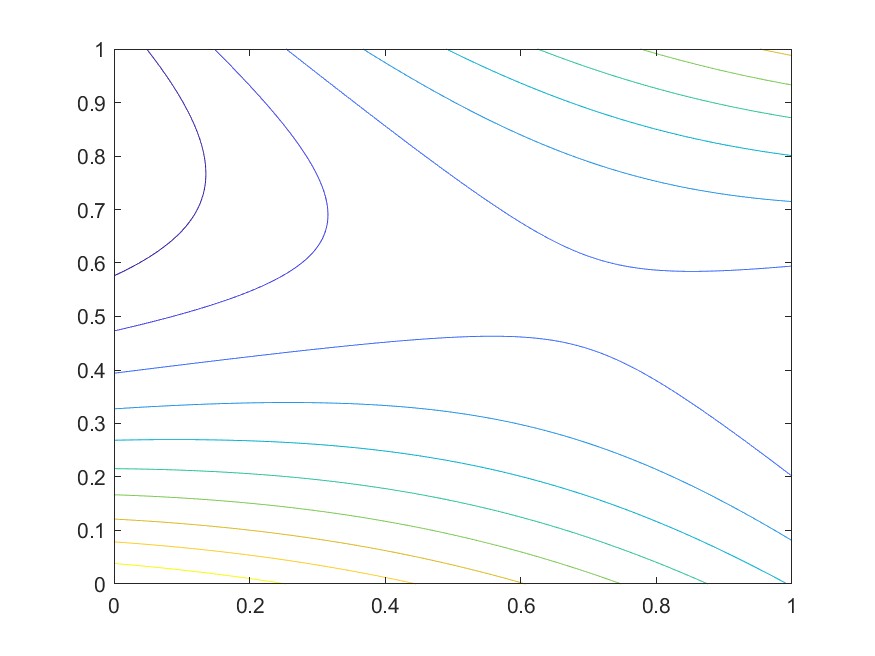}&
		\includegraphics[width=0.3\linewidth]{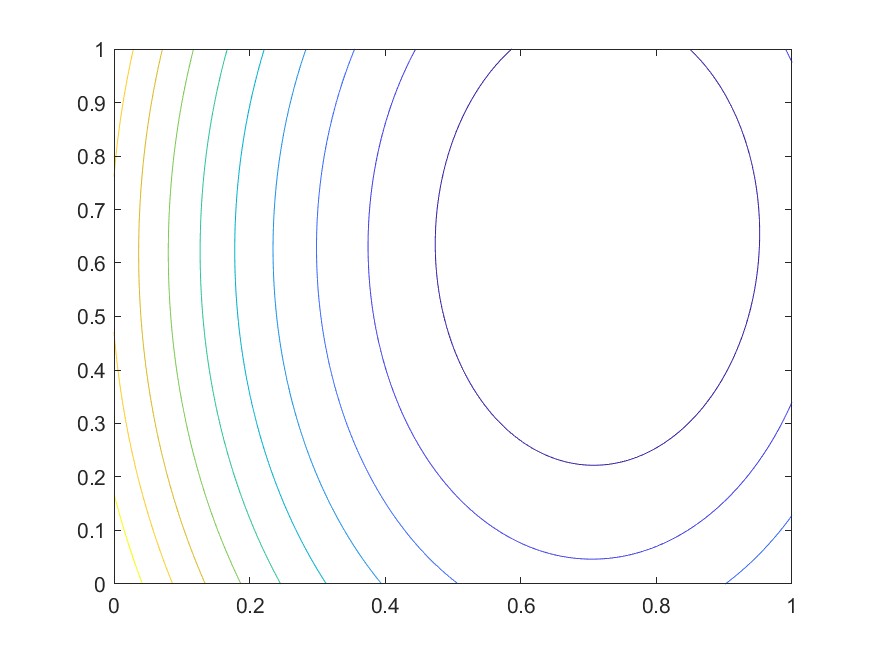}\\
	\end{tabular}
	\caption{Contours of Quadratic Orthogonal Polynomials over Domain $[0, 1]^2$ 
\label{quadraticop2}}
\end{figure}

\begin{algorithm}{Construction of Bivariate Orthonormal Polynomials} 
\label{alg1}
\begin{itemize}
\item{1.} {\bf Inputs:} Given a polygon $P$ and given a positive integer $d\ge 1$;
\item{2.} {\bf Initialization: }Compute  a triangulation $\triangle$ of $P$  and compute the diagonal block matrix 
$M^d_\triangle$  as in (\ref{mainMatrix}). 
\item{3.} Find $L {\cal D} L^\top$ decomposition  of $M_\triangle$ to get $L$ and $\sqrt{\cal D}$ which is done 
blockwisely and in parallel.
\item{4.} Compute the full smoothness matrix $H_\triangle$ according to the smoothness conditions in \cite{LS07}.
\item{5.} Form ${\cal H}= H_\triangle L^{-1} \sqrt{\cal D}^{-1}$. 
\item{6.} Compute $C_1=\hbox{\tt null}({\cal H})$ using MATLAB command {\tt null.m}.
\item{7.} Find $C= \sqrt{\cal D}^{-1} L^{-1} C_1$.  
\item{8.} Then ${\cal B}_N C$ are the orthonormal polynomials of degree $d$ over polygon $P$, 
where 
$$
{\cal B}_N=  \hbox{diag}(B_{ijk}^T,i+j+k=d,  T\in \triangle)).
$$ 
\item{9.} {\bf Output:} Output the polynomial $(B_{ijk}^{T_1},i+j+k=d)\bfc_1^\top$ in Taylor format, 
where $T_1$ is the first triangle of $\triangle$ and $\bfc_1$ is the first row of $C$. 
\end{itemize}
\end{algorithm}

Similarly, we can do the construction in the 3D setting as we have the 3D formula for (
\ref{keyformula}) and MATLAB code for generating $M_\triangle$ in (\ref{mainMatrix}) in 
the 3D setting, see \cite{ALW06} or \cite{LL22}.  We leave the detail to a future 
publication. 

\subsection{Iterative Orthogonal Polynomials}
In addition to have constructed orthogonal polynomials in $\mathbb{P}_d$, we now explain 
how to construct 
orthonormal polynomials iteratively. That is, assume that we have found the orthonormal 
polynomials for $\mathbb{P}_d$, 
say $P_{d,1}, \cdots, P_{d, D_d}$, we now construct $P_{d+1, 1}, \cdots, P_{d+1, 
D_{d+1}-D_d}$ which are orthonormal 
polynomials for $\mathbb{P}_{d+1} \ominus \mathbb{P}_d$. Such a \textcolor{blue}{sequence } of orthonormal 
polynomials will be more 
convenient to use than the ones constructed in the previous subsection. That is, we first 
use the orthonormal polynomials
in $\mathbb{P}_d$ to approximate a given function $f$ over a polygon $P$. Then we can add 
$D_{d+1}- D_d$ new terms to
approximate $f$. We then repeat to add more and more terms for a better accurate 
approximation.  

We now explain how to do. Assume that we have spline coefficient vectors $C_1, \cdots, 
C_{D_d}$ for orthonormal polynomials
 $P_{d,1}, \cdots, P_{d, D_d}$. We use the degree raising operator (cf. \cite{LS07}) to 
express these coefficient 
 vectors in terms of splines of degree $d+1$. Let us denote $C_{d+1}$. The spline 
coefficients of the orthonormal polynomials in $\mathbb{P}_{d+1}$ must be orthogonal to  
these newly degree raised spline coefficients. That is, we seek the coefficient vectors 
$C_1$ which satisfy
 \begin{equation}
 	\label{orthcond}
 	  C_{d+1} M_\triangle^{d+1} C_1 = 0,
 \end{equation} 
 where $M_\triangle^{d+1}$ is the mass matrix for spline of degree $d+1$.  We need to solve
 a new minimization:
 \begin{equation}
 	\label{min3}
 	\min  \frac{1}{2} \| C_1^\top C_1- I_\triangle\|_F,  \quad {\cal H} C_1 =0 \hbox{ and }  C_{d+1} M_\triangle(d+1) C_1 = 0. 
 \end{equation} 
 
 \begin{figure}[htpb]
 \centering
 	\begin{tabular}{ccc}
 	\includegraphics[width=0.3\linewidth]{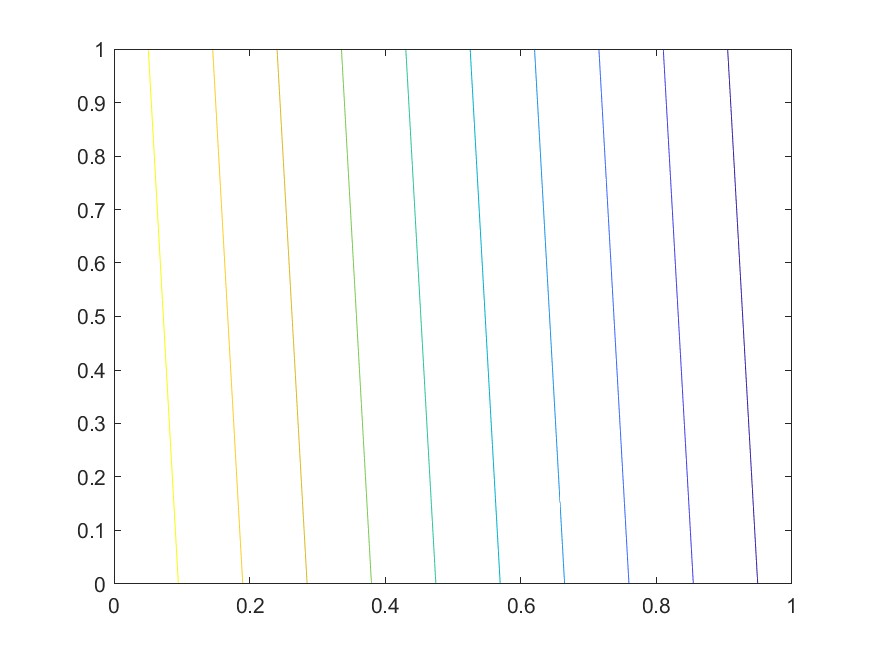}&
 	\includegraphics[width=0.3\linewidth]{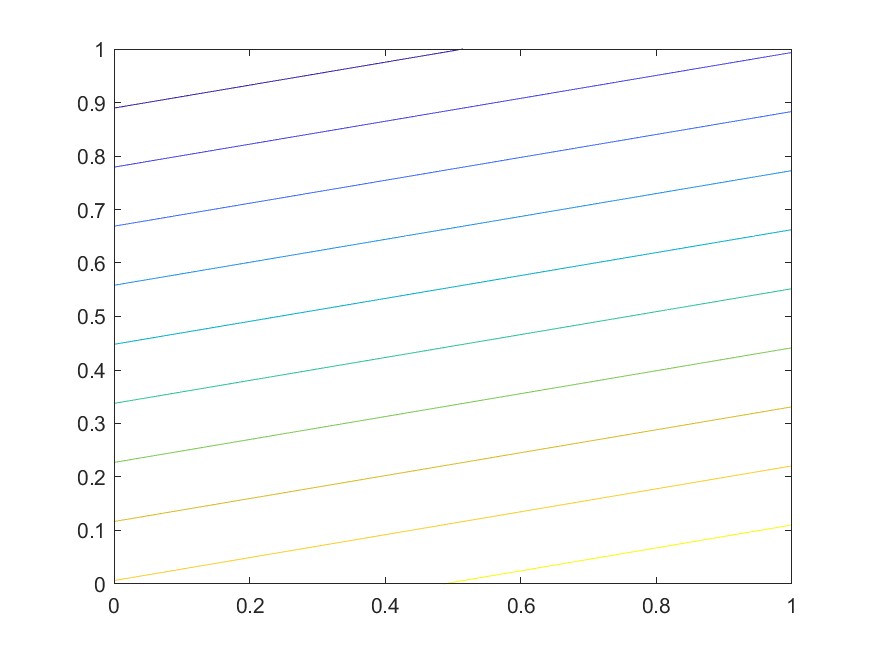}&
 	\includegraphics[width=0.3\linewidth]{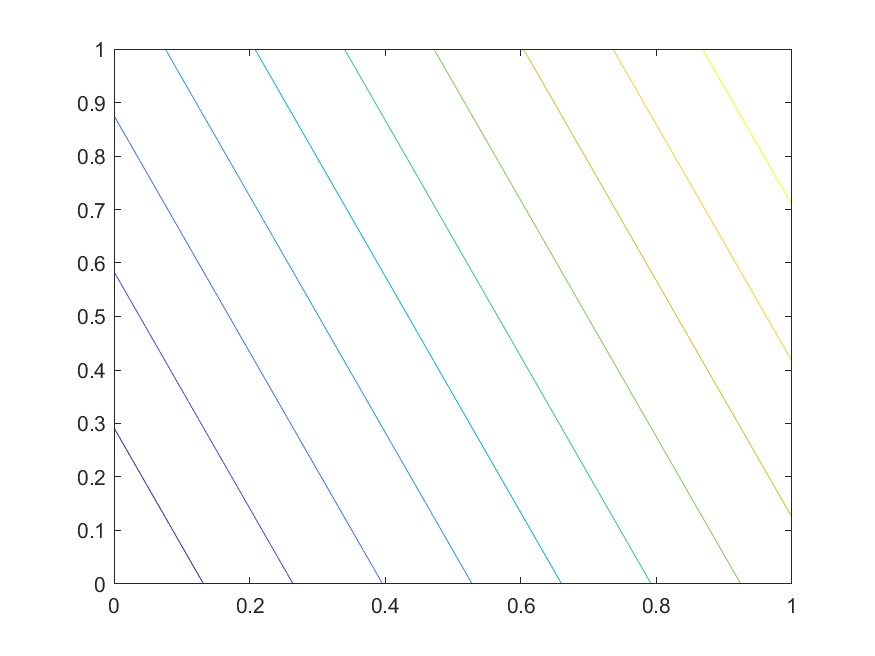}\cr
 	\includegraphics[width=0.3\linewidth]{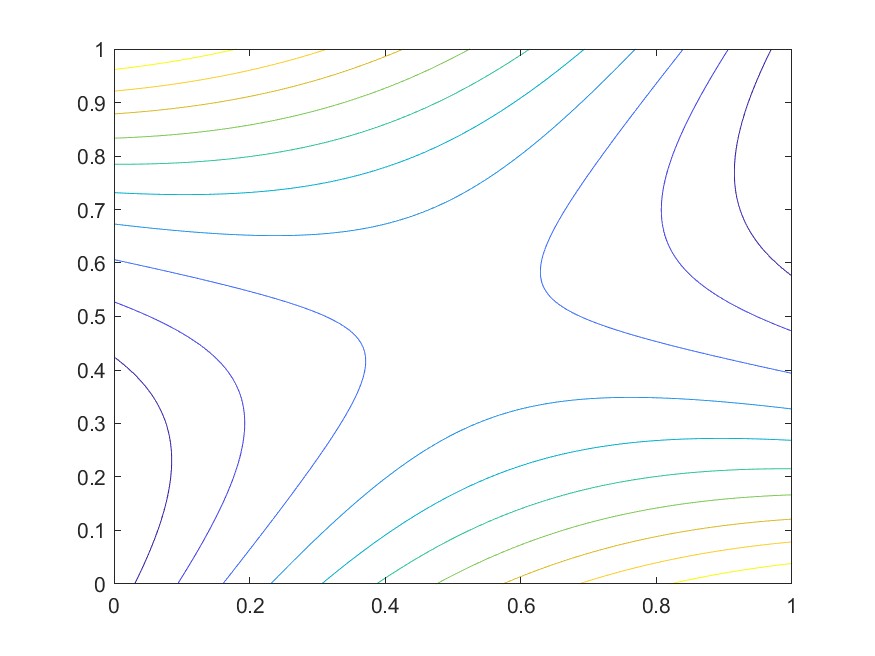}&
 	\includegraphics[width=0.3\linewidth]{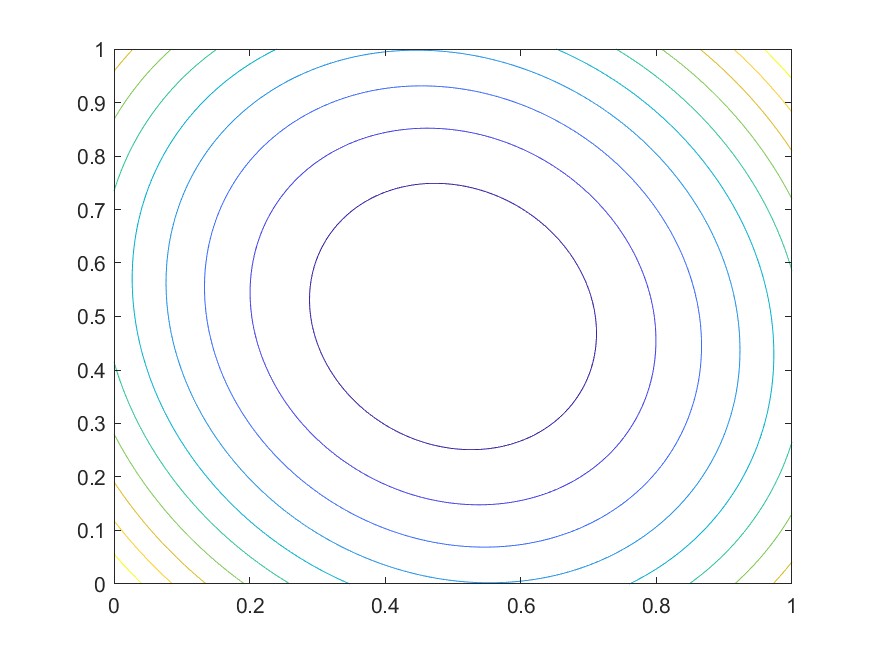}&
 	\includegraphics[width=0.3\linewidth]{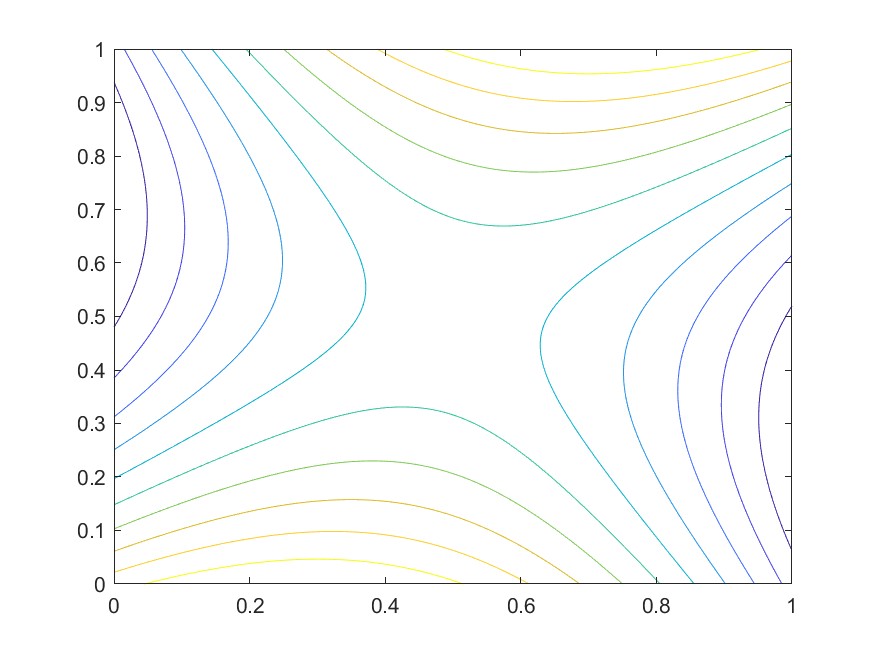}
 \end{tabular}
 	\caption{Contours of Orthogonal Polynomials of Degree 1 (top row) and Degree 2 (bottom row) over Domain $[0, 1]^2$ \label{quadraticContour}}
 \end{figure}

 The above discussion leads to the following computational algorithm:
 
 \begin{algorithm}{Iterative Construction of Bivariate Orthonormal Polynomials}
 \label{alg2}
 	\begin{itemize}
 		\item{1.} {\bf Inputs:} Given a polygon $P$ and given a positive integer $d\ge 1$; Input the coefficient vectors of all orthonormal polynomials of degree $d$ over $P$.  
 		\item{2.} {\bf Initialization: }Compute  a triangulation $\triangle$ of $P$  and compute the diagonal matrix 
\textcolor{blue}{ $M^{d+1}_\triangle$ } 
 		as in (\ref{mainMatrix}). 
 		\item{3.} Find $L {\cal D} L^\top$ decomposition  of $M^{d+1}_\triangle$ 
to get $L$ and $\sqrt{\cal D}$.
 		\item{4.} Compute the full smoothness matrix $H_\triangle$.
 		\item{5.} Form ${\cal H}= H_\triangle L^{-1} \sqrt{\cal D}^{-1}$.
 		\item{6.} Use the degree raising algorithm to rewrite the spline coefficient vectors of the orthonormal polynomials of degree $d$ to be a polynomial of degree $d+1$ and then denote them by $C_{d+1}$. 
 		\item{7.} Compute ${\cal C}=C_{d+1} M^{d+1}_\triangle$ and form 
${\cal H}_{d+1}=[{\cal H}; {\cal C}]$; 
 		\item{8.} Compute $C_1=$ {\tt null} $({\cal H}_{d+1})$ using MATLAB command {\tt null.m}.
 		\item{9.} Find $C= \sqrt{\cal D}^{-1} L^{-1} C_1$.  
 		\item{10.} Then ${\cal B}_N C$ are the orthonormal polynomials of degree $d$ over polygon $P$, where 
 		$$
 		{\cal B}_N=  \hbox{diag}(B_{ijk}^T,i+j+k=d+1,  T\in \triangle)).
 		$$ 
 		\item{9.} {\bf Output:} Output the polynomial $(B_{ijk}^{T_1},i+j+k=d)\bfc_1^\top$ in Taylor format, where 
 		$T_1$ is the first triangle of $\triangle$ and $\bfc_1$ is the first row of $C$. 
 	\end{itemize}
 \end{algorithm}
 
 Let us  show the contour plot of the orthonormal polynomials of degree $1$  and the 
orthonormal polynomials of
 degree $2$ which are orthogonal to all polynomials of degree $1$ over a standard 
rectangular domain  $[0, 1]^2$.  See Figure~\ref{quadraticContour}. 
 We can compare the contours in Figure~\ref{quadraticContour} with the contours in 
Figure~\ref{quadraticop2} to see that the graphs look
 better when using the iterative orthonormal construction in Algorithm~\ref{alg2} 
than using Algorithm~\ref{alg1}. The contour lines on the top three graphs in 
Figure~\ref{quadraticContour} are linear polynomials and the bottom three graphs are quadratic 
polynomials. Also, among the bottom three graphs, the first and last one look similar in the sense
that the last one is a rotation of the first one.

\subsection{Numerical Examples}
In this section, we show many examples in the 2D setting.  However, examples in the 3D 
settings are difficult to compute and visualize. The detail is left to a future 
publication.   
Let us present some examples of orthogonal polynomials based on a triangle, 
rectangular domain, an L-shaped domain and a flower domain which has  a hole in the 
middle. 
 
\begin{example}
We first start with examples of orthonormal polynomials over a triangle $T$ with vertices
$(0,0),(1,0), (0.5,1)$.  For convenience, we only present their explicit formula in Taylor format for degree $d=1,2$. 
We start with $P_0= \sqrt{2}$ and then compute two linear orthogonal polynomials $P_{1,0}(x,y)=a_0 + b_0 x+ c_0 y$ and 
$P_{1,1}(x,y)= a_1 + b_1 x+ c_1 y$, where $a_0, b_0, c_0$ is the first row while $a_1, b_1, c_1$ is the second row below.
\begin{eqnarray}
  3.997836537223834 & -5.882837001131618 & -3.169254109974075\cr
   0.131540950422845& -3.659539427047720 &  5.094686289303045.
\end{eqnarray}

$P_{2,0}(x,y)= a_2+ b_2x+c_2 y+ d_2 x^2+e_2 xy +f_2 y^2$, and similar for $P_{2,1}, P_{2,2}$. Their coefficients $a_i, b_i, c_i, d_i, e_i, f_i$ for 
$i=0,1,2$ are given in Table~\ref{tab2}. 
\begin{table}[htpb]
\tiny  
\begin{tabular}{cccccc}\hline 
  0.568231172946732 & -19.049005993186132 &  24.027659606205336&  23.574328321631199 & -22.626611642225328 & -12.946151191161945\cr \hline
   5.482587063060269& -22.420544765433238 & -10.229879356332292 & 15.244458227262641 &  35.880432690852999&-7.732363957959931\cr \hline
   4.859871647487797& -16.924968038160749 &-13.491521122661251 & 17.083838489671507  & -0.794352257553776 & 19.496351375506610\cr \hline
   \end{tabular}
   \caption{Polynomial Coefficients for $P_{2,0}, P_{2,1}, P_{2,2}$. \label{tab2}}
   \end{table}
   We  have verified numerically that the polynomials with these coefficients are orthogonal to each other within the accuracy of $1e-16$.  

For convenience, we only show them in graph for $d=0, \cdots, 5$ in 
Figure~\ref{Torthpoly13}
and Figure~\ref{Torthpoly45}, respectively, where $T$ with vertices $(0,0),(1,0), (1,1)$.
 
\begin{figure}[htpb]
\centering
	\begin{tabular}{cccc}
\includegraphics[width=0.2\linewidth]{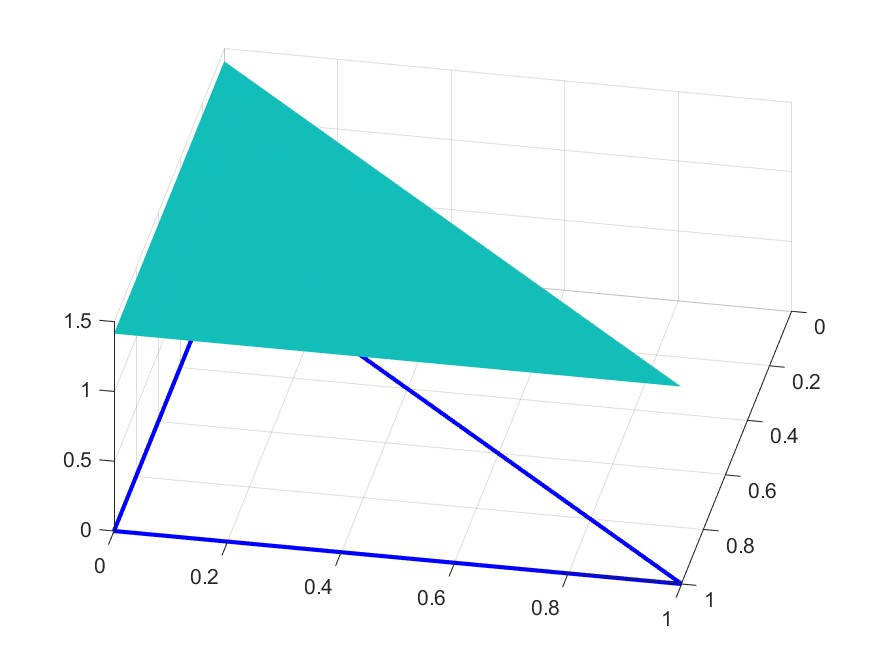}&
\includegraphics[width=0.2\linewidth]{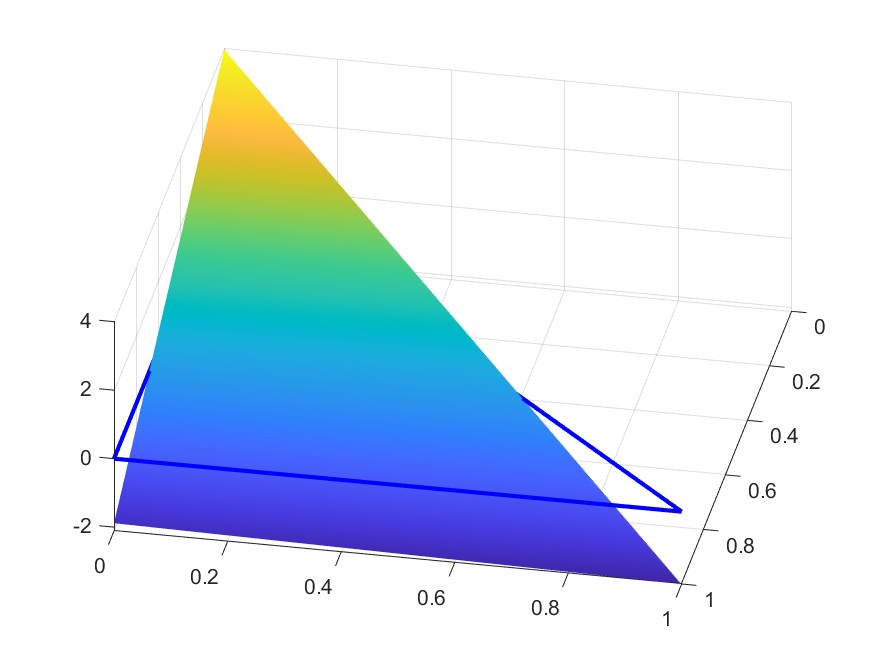}&
\includegraphics[width=0.2\linewidth]{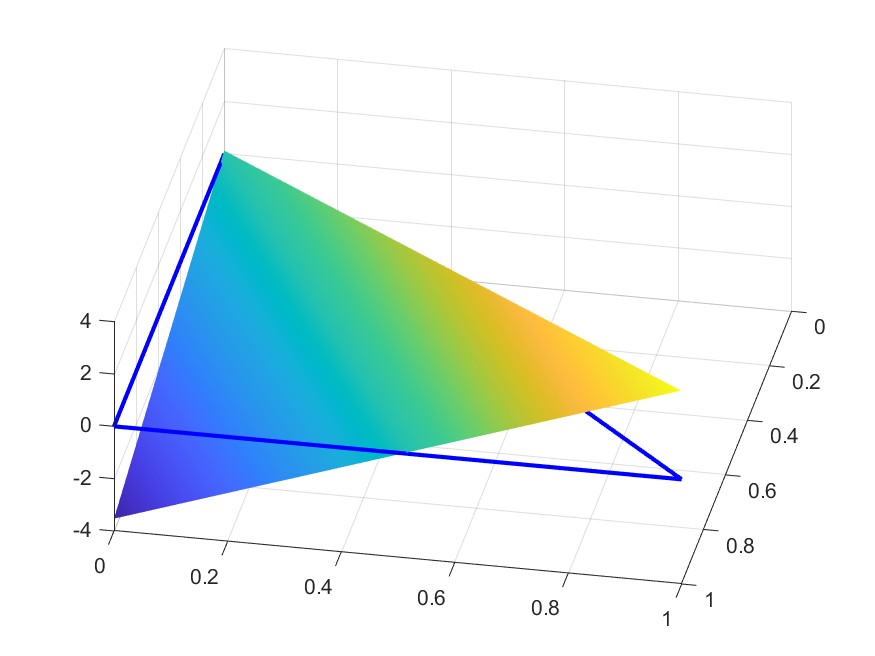}&
\includegraphics[width=0.2\linewidth]{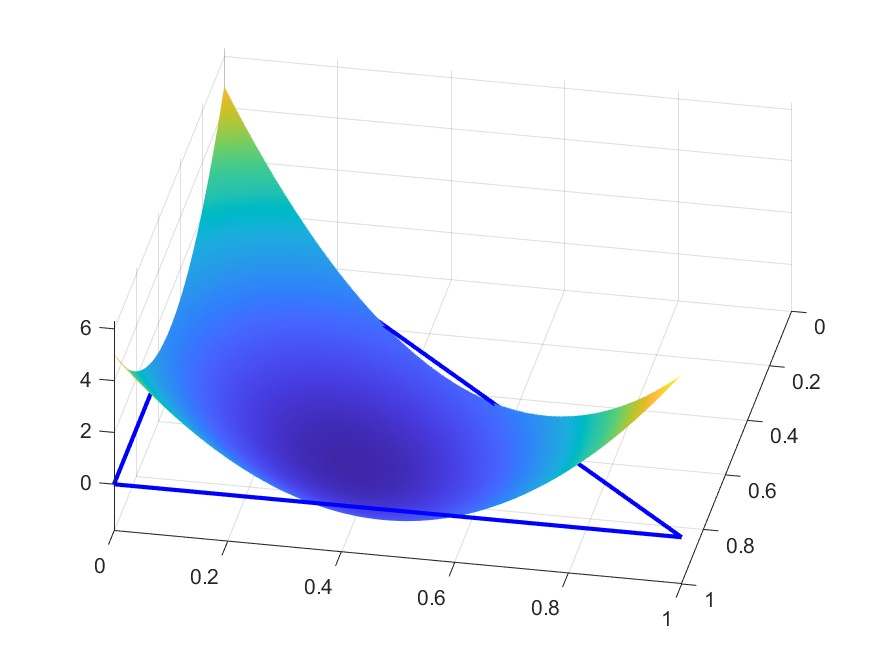}\\
\includegraphics[width=0.2\linewidth]{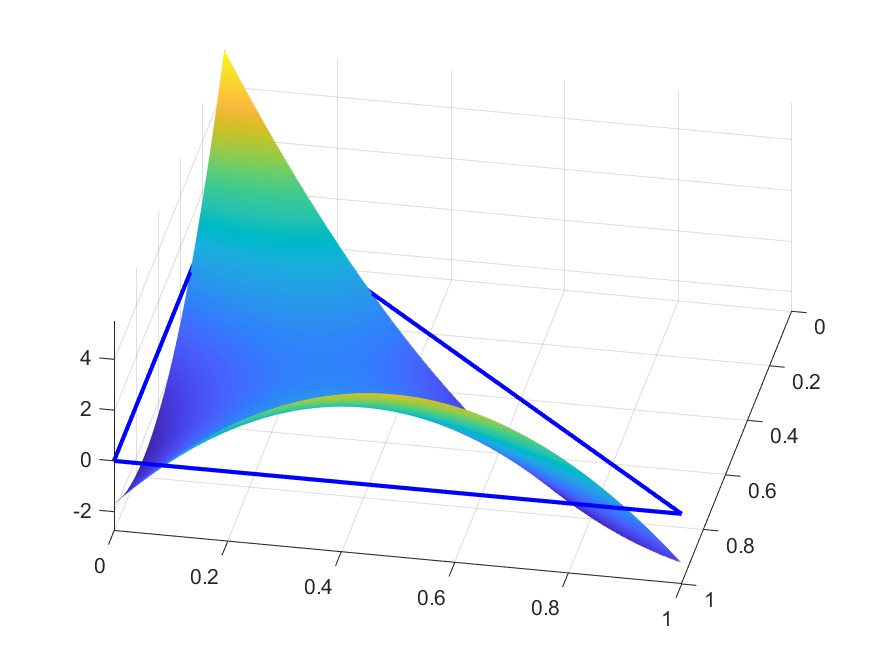}&
\includegraphics[width=0.2\linewidth]{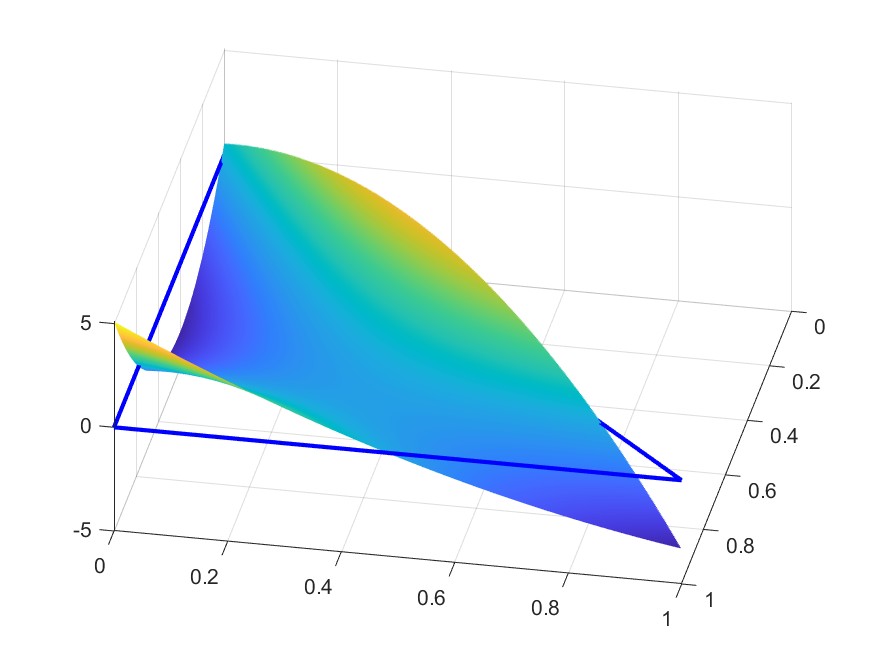}&
\includegraphics[width=0.2\linewidth]{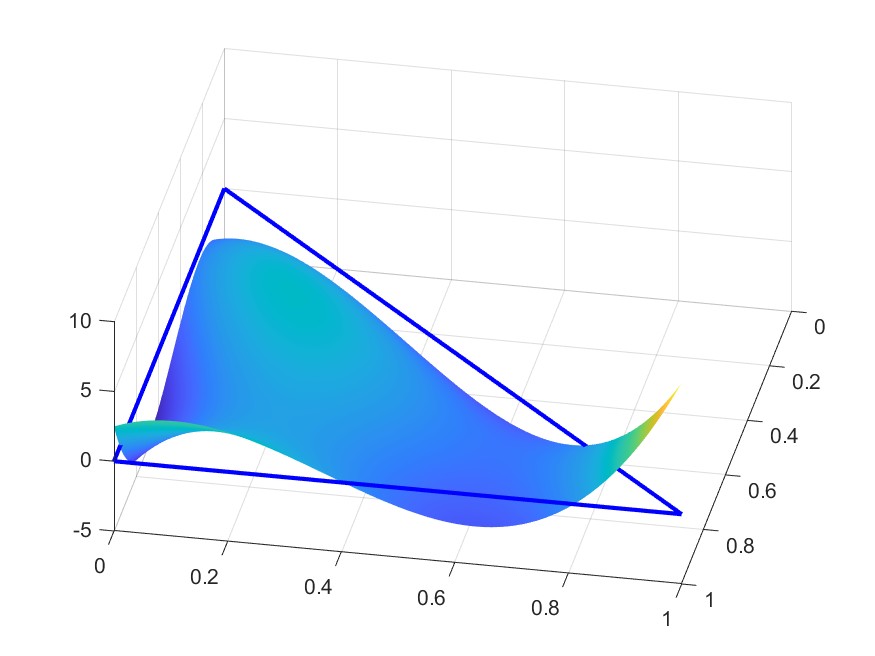}&
\includegraphics[width=0.2\linewidth]{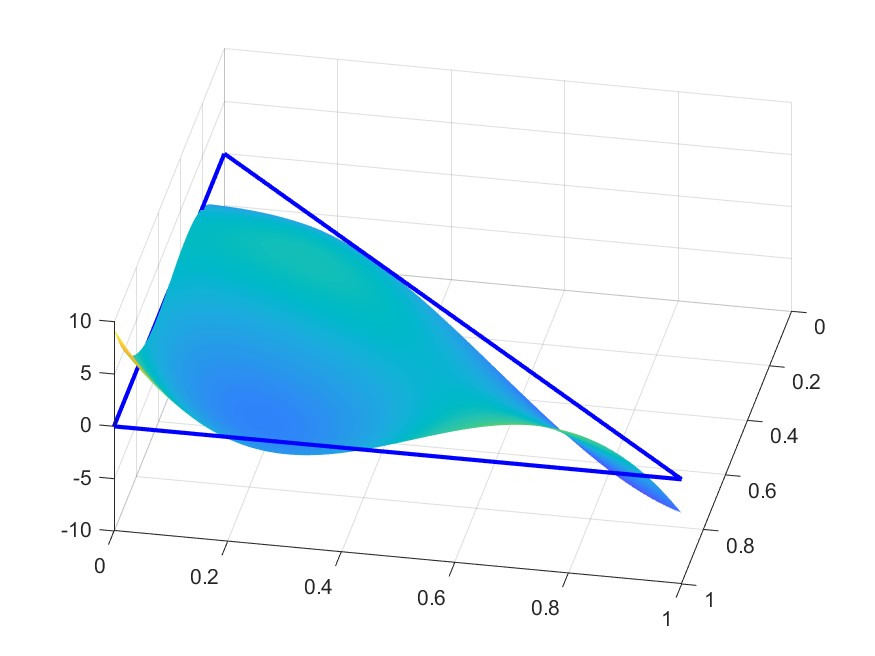}\\
\includegraphics[width=0.2\linewidth]{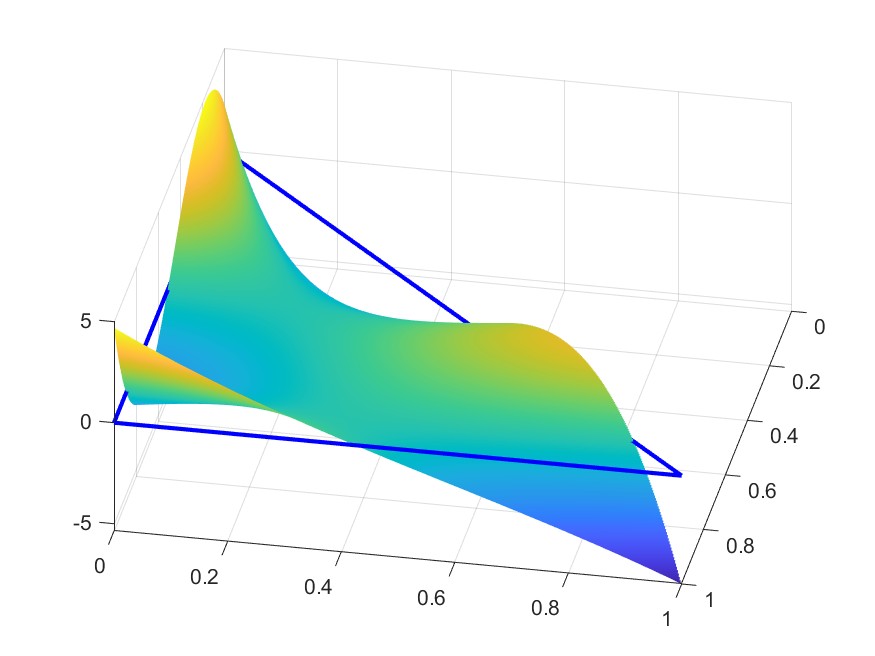}&
\includegraphics[width=0.2\linewidth]{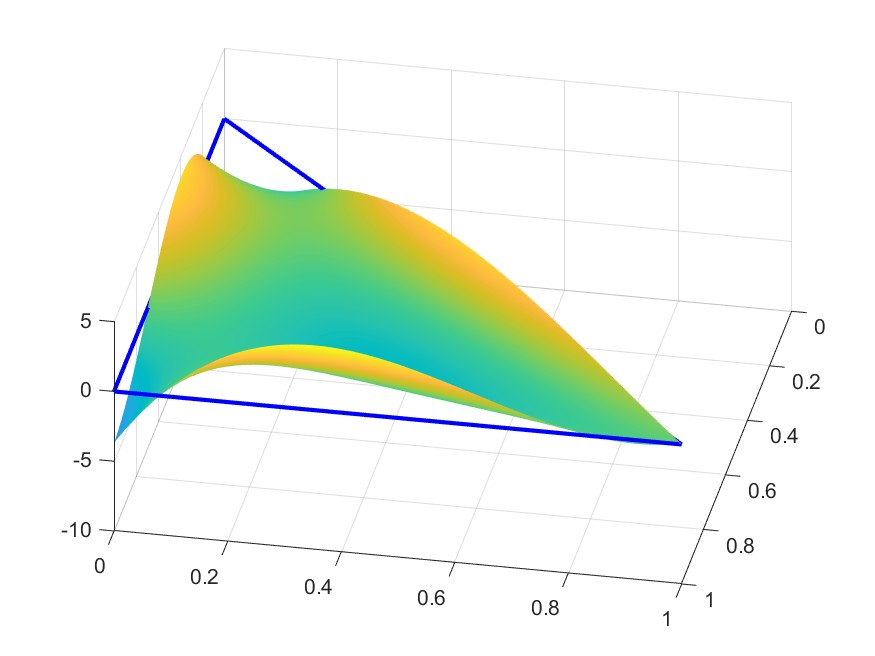}&
\end{tabular}
\caption{Orthonormal Polynomials of degrees $0, 1, 2, 3$ \label{Torthpoly13}}
\end{figure}

\begin{figure}[htpb]
\centering
	\begin{tabular}{cccc}
\includegraphics[width=0.2\linewidth]{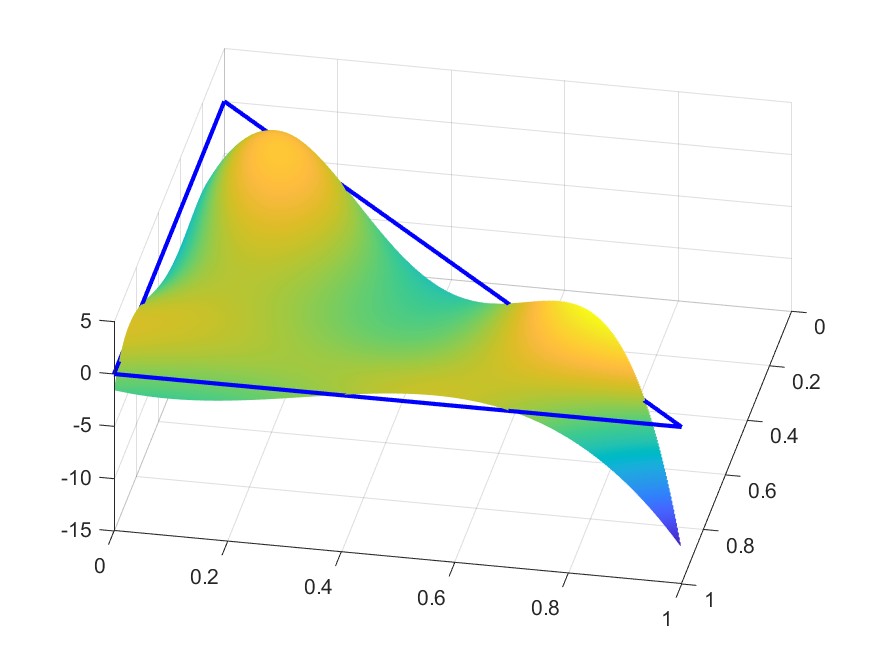}&
\includegraphics[width=0.2\linewidth]{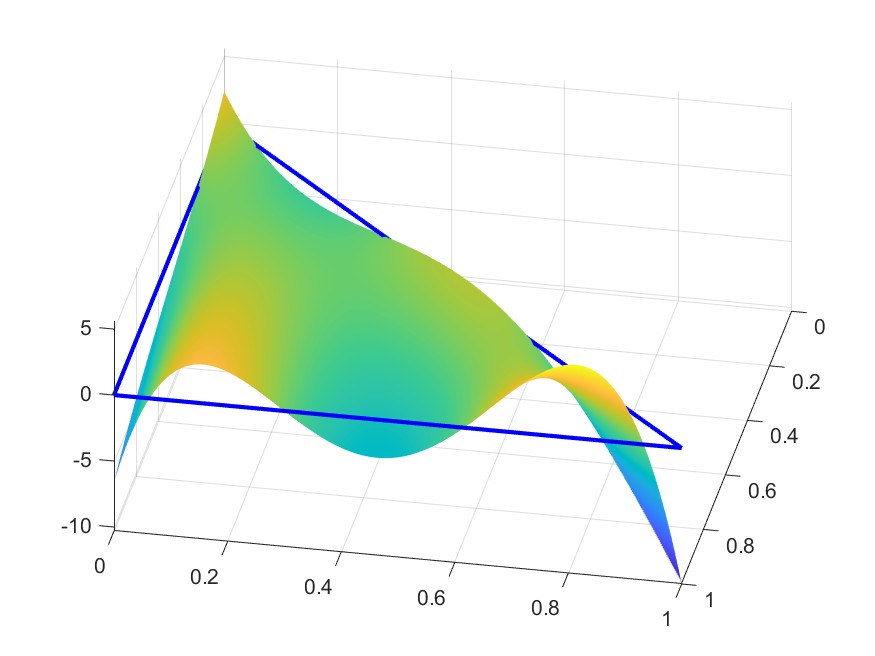}&
\includegraphics[width=0.2\linewidth]{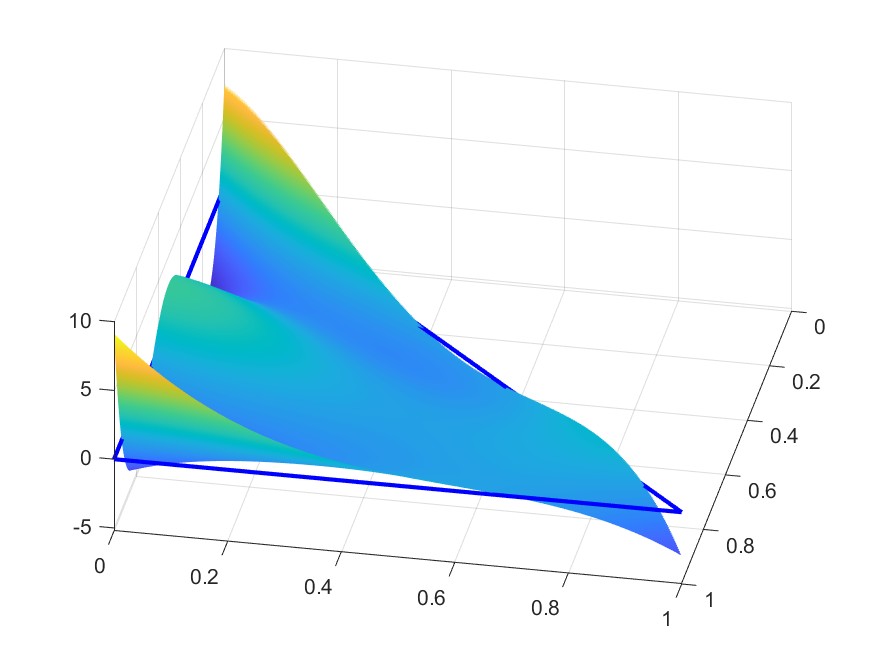}&
\includegraphics[width=0.2\linewidth]{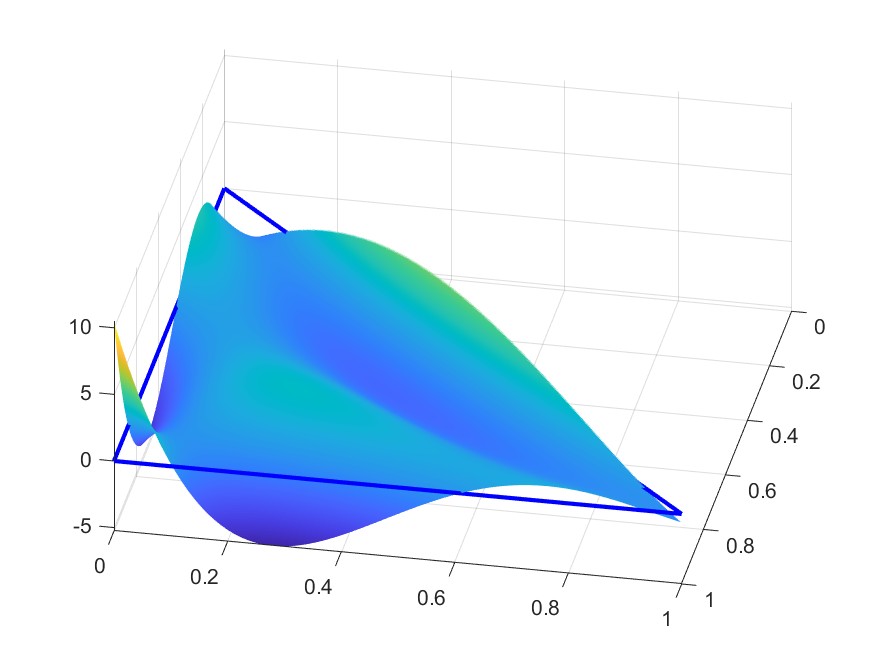}\\
\includegraphics[width=0.2\linewidth]{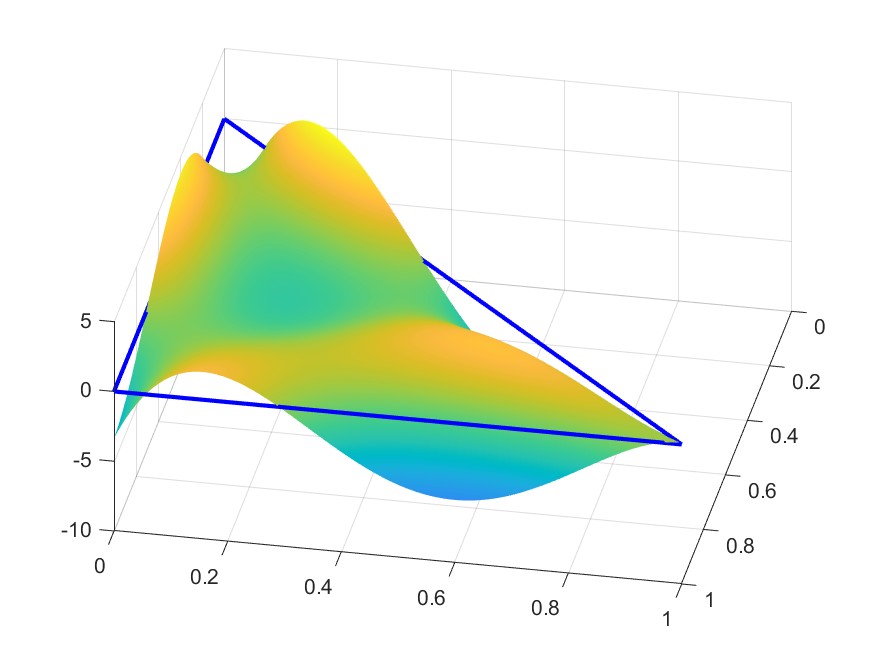}&
\includegraphics[width=0.2\linewidth]{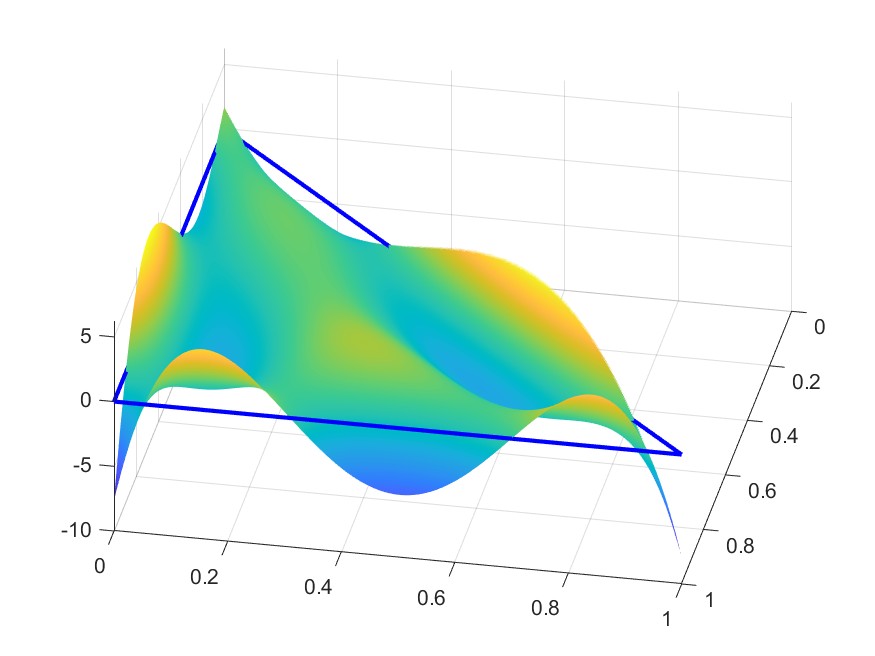}&
\includegraphics[width=0.2\linewidth]{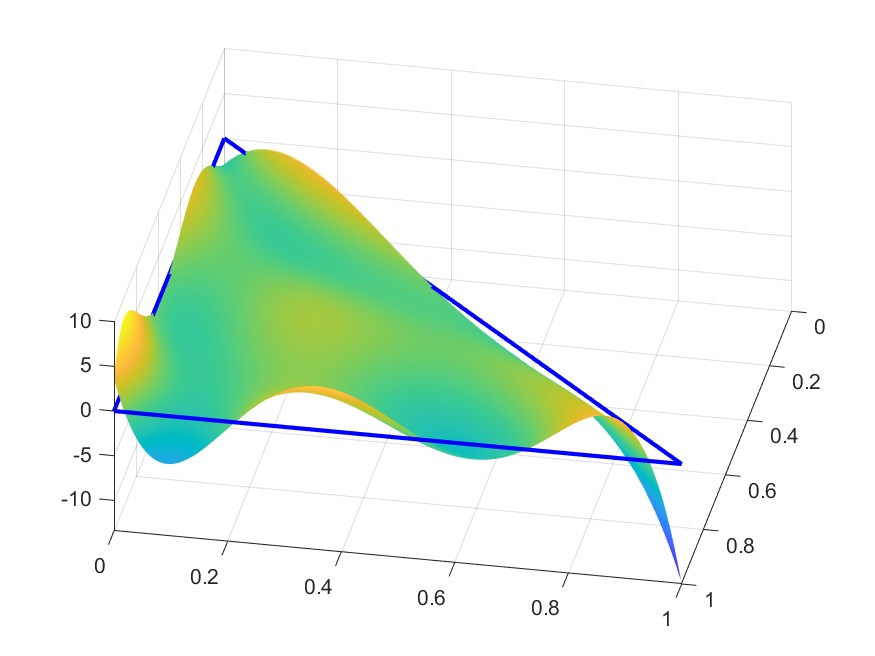}&
\includegraphics[width=0.2\linewidth]{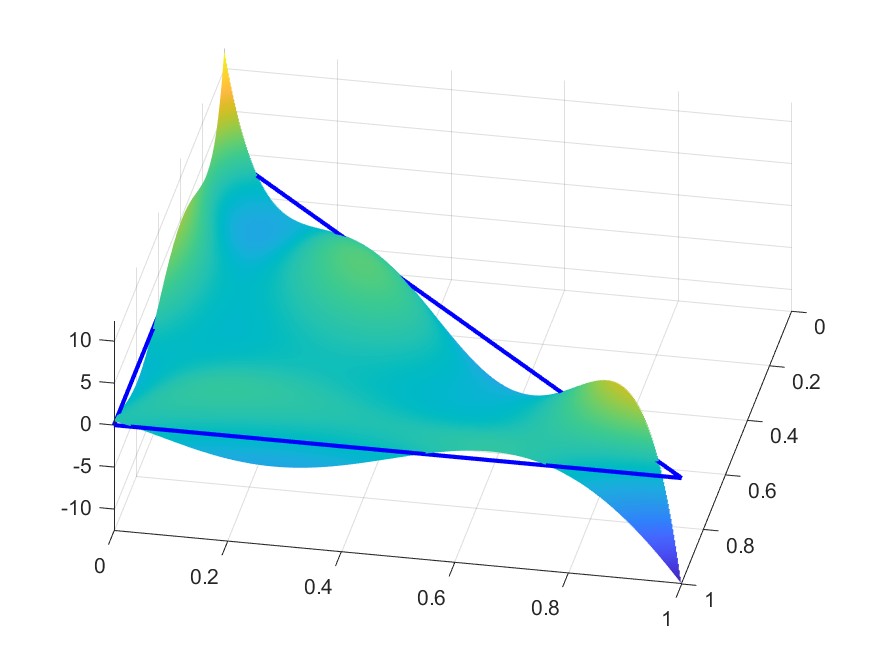}\\
\includegraphics[width=0.2\linewidth]{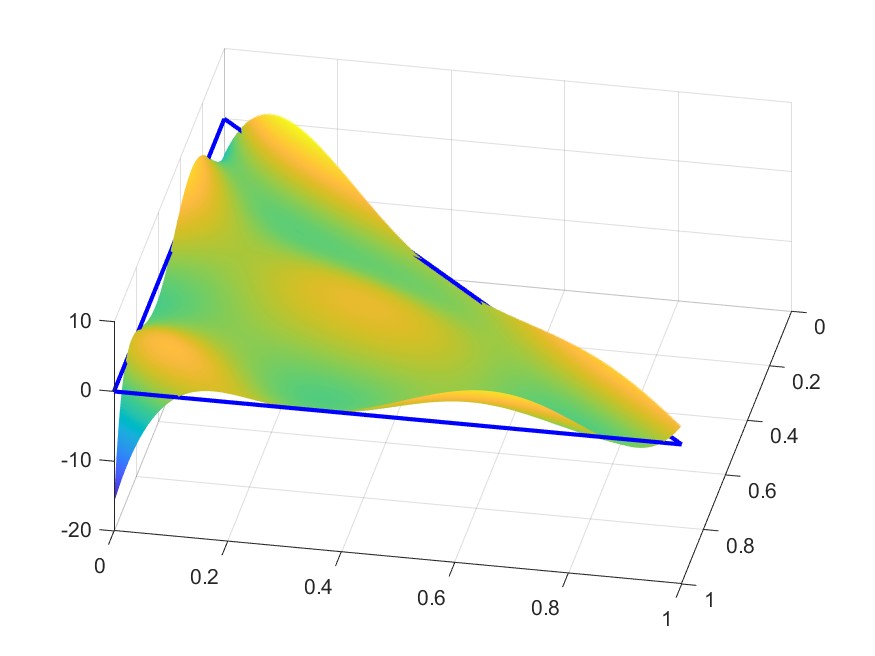}&
\includegraphics[width=0.2\linewidth]{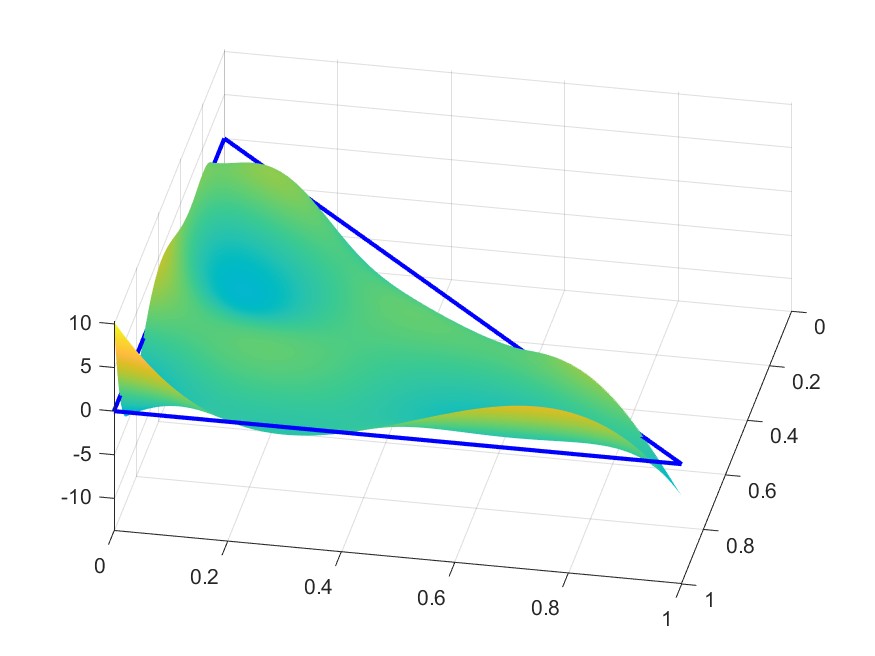}&
\includegraphics[width=0.2\linewidth]{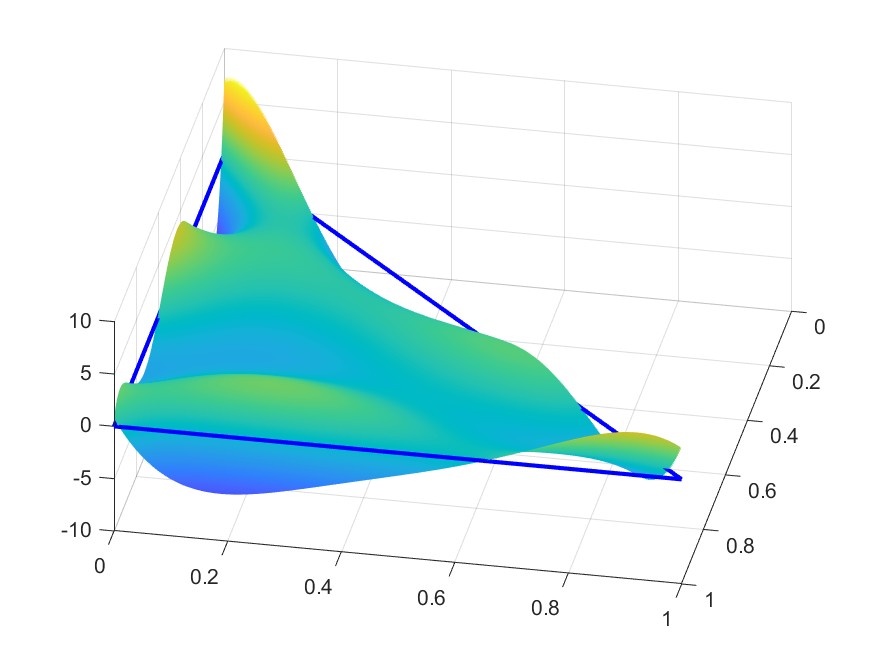}&
\end{tabular}
\caption{Orthonormal Polynomials of degrees $4$ and $5$ \label{Torthpoly45}}
\end{figure} 
\end{example} 

\begin{example}
Let us consider the standard domain $[-1, 1]^2$. The orthonormal polynomials of 
degree $1$ are easy to obtain. $p_0(x,y)=1/2, p_1(x,y)=\sqrt{3}/2 x, p_2(x,y) 
=\sqrt{3}/2 y$.  The orthonormal polynomials of degree 2 are given in terms of their
coefficients below in the order of $1, x, y, x^2, xy, y^2$
\begin{equation*}
X=
\begin{bmatrix}
    35/116  &   0 &   0 &           -229/221 &     -2630/2241 &   557/4251 \cr  
 -171/737 &   0   &   0 &        -598/879  &     1353/2242 & 373/271  \cr  
      88/127 & 0  &   0 &       -1325/1173 &      517/725 &   -3435/3619  \cr
      \end{bmatrix}
   \end{equation*}
 A MATLAB code  to check whether these 6 polynomials are orthonormal to each other is given 
 below.
 \begin{verbatim}
 >> syms x y fname1 fname2 fname11 fname12 fname13 ff
>> fname1=@(x,y) sqrt(3/2)*x;
>> fname2=@(x,y) sqrt(3/2)*y;
>> fname11=@(x,y) X(1,1)+X(1,4)*x.^2+X(1,5)*x.*y+X(1,6)*y.^2;
>> int(int(fname11,x,-1,1),y,-1,1)
 
ans =
 
13/13510798882111488
>> fname12=@(x,y) X(2,1)+X(2,4)*x.^2+X(2,5)*x.*y+X(2,6)*y.^2;
>> int(int(fname12,x,-1,1),y,-1,1)
 
ans =
 
-29/27021597764222976
 
>> fname13=@(x,y) X(3,1)+X(3,4)*x.^2+X(3,5)*x.*y+X(3,6)*y.^2;
>> int(int(fname13,x,-1,1),y,-1,1)
 
ans =
 
11/6755399441055744
 
>> ff=@(x,y) fname12(x,y).*fname11(x,y);
>> int(int(ff,x,-1,1),y,-1,1)
 
ans =
 
458907104604566899/7301667457314601352621010462965760
 
>> 458907104604566899/7301667457314601352621010462965760

ans =

   6.2850e-17

>> ff=@(x,y) fname12(x,y).*fname13(x,y);
>> int(int(ff,x,-1,1),y,-1,1)
 
ans =
 
-232997285764840411/3650833728657300676310505231482880
 
>> -232997285764840411/3650833728657300676310505231482880

ans =

  -6.3820e-17

>> ff=@(x,y) fname12(x,y).*fname1(x,y);
>> int(int(ff,x,-1,1),y,-1,1)
 
ans =
 
0
 
>> ff=@(x,y) fname13(x,y).*fname1(x,y);
>> int(int(ff,x,-1,1),y,-1,1)
 
ans =
 
0

 \end{verbatim}

However, there are many different sets of orthonormal polynomials of degree $1$. 
For example, $p_0(x,y)=1/2$ while  $p_1(x,y)
a_1x+ b_1 y$ and $p_2=a_2x + b_2 y$, where $[a_1,b_1]$ is the first 
row and $[a_2,b_2]$ is the second row.  
\begin{eqnarray}
\label{opd1a}
  &  0.119407995182024 & -0.857753887013408 \cr
  & -0.857753887013408 & -0.119407995182024 
\end{eqnarray}
One can easily check that $p_0, p_1, p_2$ are orthonormal to each other over $[-1, 1]^2$ 
which have been done by the author by using MATLAB commands similar to the one above.  

Another set of linear orthonormal polynomials are $p_0(x,y)=a_0+b_0 x+ c_0 y;
p_1(x,y)=a_1+b_1x+ c_1y$ and $p_2(x,y)= a_2 + b_2 x+ c_2 y$ with coefficients 
$a_0,b_0,c_0$; $a_1,b_1, c_1$; and $a_2, b_2, c_2$ given in the following data array:
\begin{eqnarray*}
\label{opd1b}
0.736626476725820 &  1.067183625864925 &  0.482973508637176 \cr
0.191228098068285 &  0.356328591835220 & -1.662325292524989\cr
0.648701201090063 & -1.316870928878955 & -0.058405580096484\cr
\end{eqnarray*}
The orthonormality of these linear polynomials are verified by the author with accuracy 
$1e-15$ in MATLAB. 
\end{example}

\begin{example}
Let us consider a polygon $[0, 1]^2$. The orthonormal polynomials of degree $1$ are 
given as follows. The first one is $p_0=1$ defined over $[0, 1]^2$. Write $p_1(x,y)=
a_1+b_1 x+ c_1 y$ and $p_2=a_2+ b_2 x+ c_2 y$, where $[a_1,b_1,c_1]$ is the first 
row and $[a_2,b_2,c_2]$ is the second row.  
\begin{eqnarray}
\label{opd1}
1.476691783662768 &  0.477631980728096 & -3.431015548053630\\
   1.954323764390864 & -3.431015548053631 & -0.477631980728095
\end{eqnarray}
One can easily check that $p_0, p_1, p_2$ are orthonormal to each other over $[0, 1]^2$. 
These polynomials will be used to generated quadrature to be given in a later example.  
\end{example}

\begin{example}
Consider a polygon $P$ which is an L-shaped domain \textcolor{blue}{with vertices $[0,0;2, 0;2, 1;1, 1;1, 2;0, 2]$ 
which is within in $[0, 2]^2$. }   
For degree $d=2$, 
	we use Algorithm~\ref{alg1} to find bivariate quadratic orthonormal polynomials. 
Their graphs were shown 
	 in Figure~\ref{quadraticop0}. We now present their coefficients in the power format 
in Table~\ref{tab9}. 
For example, $P1=	-0.27759+ 1.00778x+ 0.81207y -0.73939x^2+ 1.00649xy -0.83594y^2$. 
Note that these coefficients are chopped after 5 digits 
for the decimal part.  The numerical integration of their inner products of these 
polynomials based on polynomial values at $101\times 101$ equally-spaced points  of 
$[0, 2]^2$ gives about $1e-4$ error away from the exact orthonormality
	  $I_{6\times 6}$. 
	 
	\begin{table}[htpb]
		\centering
		\begin{tabular}{|l|c|c|c|c|c|c|} \hline 
		P1 &	-0.27759 & 1.00778 & 0.81207 & -0.73939 & 1.00649 & -0.83594\cr 
			\hline
			P2 &	0.57155 & 0.91303 & -0.54342 & -0.74757 & -0.28896 & -0.12163\cr 
			\hline
			P3 &	1.58044 & -0.21017 & -2.85523 & -0.19324 & 0.42186 & 1.31931\cr 
			\hline
			P4 &	0.11336 & 1.01927 & -2.61386 & -0.25244 & 0.35354 & 1.14126\cr 
			\hline
			P5 &	1.49075 & -3.02600 & -0.27774 & 1.51963 & 0.26120 & -0.21360\cr 
			\hline
			P6 &	2.24632 & -3.74764 & -3.18077 & 0.99871 & 2.57303 & 0.84324\cr 
			\hline
		\end{tabular}
		\caption{Coefficients(in rows) of six quadratic orthonormal polynomials \label{tab9}}
		\end{table} 
\end{example}

\begin{example}
Consider a polygon $P$ which is an L-shaped domain.  For degree $d=3$, 
we again use Algorithm~\ref{alg1} in the previous section 
to find orthonormal polynomials. Let us present
their graphs  in Figure~ \ref{cubicop2}.   
Their coefficients in the power format are given in Table~\ref{tab0}. We only show their 
coefficients in 5 digit chopping. 

\begin{table}[htpb]
	\centering
\scriptsize 
	\begin{tabular}{|l|c|c|c|c|c|c|c|c|c|c|} \hline 
	P1 &-0.41821 & 2.23299 & 1.68804 & -2.34761 & 1.26810 & -2.37918 &0.71637 &-0.04421 
&-1.46598 &1.05805\cr 
\hline
	P2 &1.03950 & -3.61418 & -1.01216 & 5.01519 & 3.68612 & -1.26161 &-1.91878 &-1.34261 
&-1.17910 &0.78674\cr 
\hline
	P3 &-2.25562 & 9.10083 & 6.34619 & -10.00452 & -4.83676 & -6.64088 &2.98589 &1.34142 
&1.87440 &1.87401\cr 
\hline
	P4 &-1.12493 & 4.31404 & 4.34727 & -1.97392 & -11.92040 & -1.86506 &0.20965 &3.34780 
&3.78191 &0.03874\cr 
\hline
	P5 &2.40821 & -5.22071 & -1.97223 & 2.33479 & 5.03458 & -0.73131 &-0.40797 &-0.36993 
&-2.21076 &0.79109\cr 
\hline
	P6 &2.41601 & -1.95840 & -7.34403 & 0.24409 & 3.59120 & 7.87856 &0.13064 &-0.89908 
&-0.74196 &-2.73171\cr 
\hline
	P7 &-0.21502 & 6.63545 & -5.47597 & -7.05251 & -0.86886 & 5.21139 &1.85168 &2.14921 &-1.39347 &-1.16061\cr 
\hline
	P8 &0.62764 & 0.60790 & -2.34192 & 0.63250 & -3.70790 & 2.68535 &-0.77545 &2.25788 &0.88506 &-0.66557\cr 
\hline
	P9 &2.42865 & -4.96796 & -3.20560 & 2.44912 & 5.39874 & -0.25506 &0.00987 &-2.51434 &-0.20070 &0.54400\cr 
\hline
	P10 &1.36544 & -0.76606 & -7.32284 & -1.31750 & 7.59777 & 6.79411 &0.83367 &-2.63263 &-3.02784 &-1.62395\cr 
\hline
	\end{tabular}
\caption{Coefficients(in rows) of 10 cubic orthonormal polynomials \label{tab0}}
\end{table} 		
 
\begin{figure}[htpb]
\centering
	\begin{tabular}{cccc}
\includegraphics[width=0.2\linewidth]{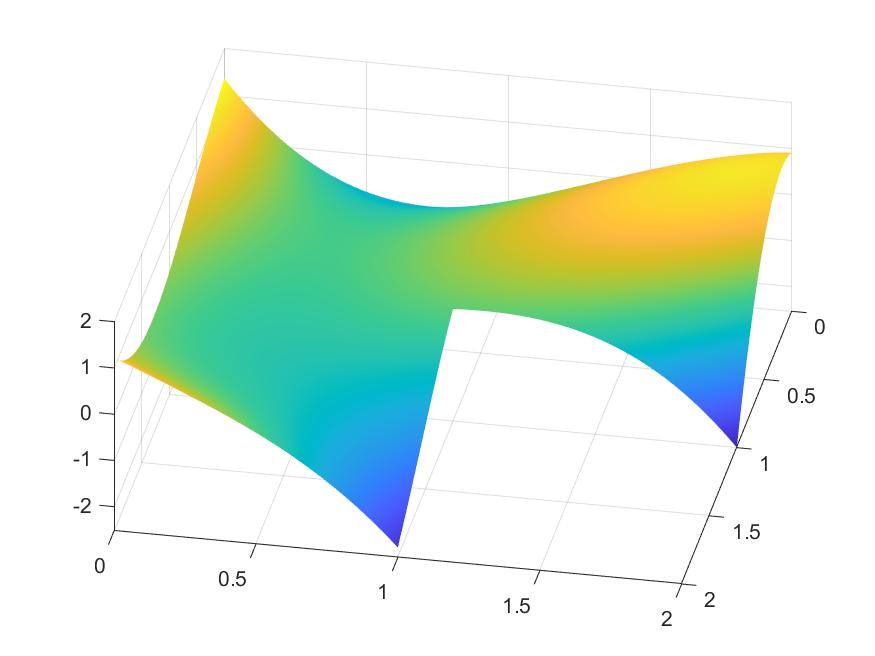}&
\includegraphics[width=0.2\linewidth]{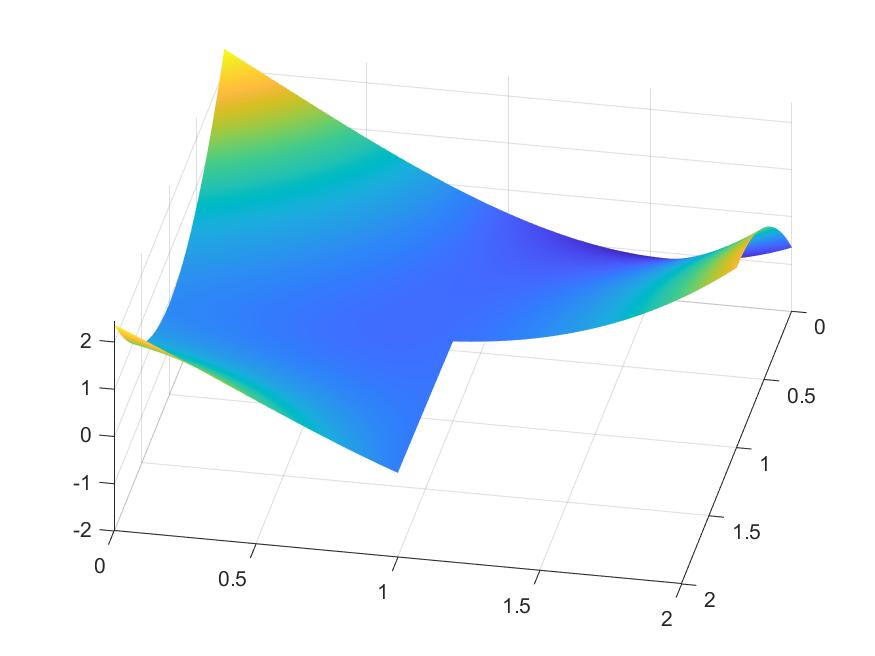}&
\includegraphics[width=0.2\linewidth]{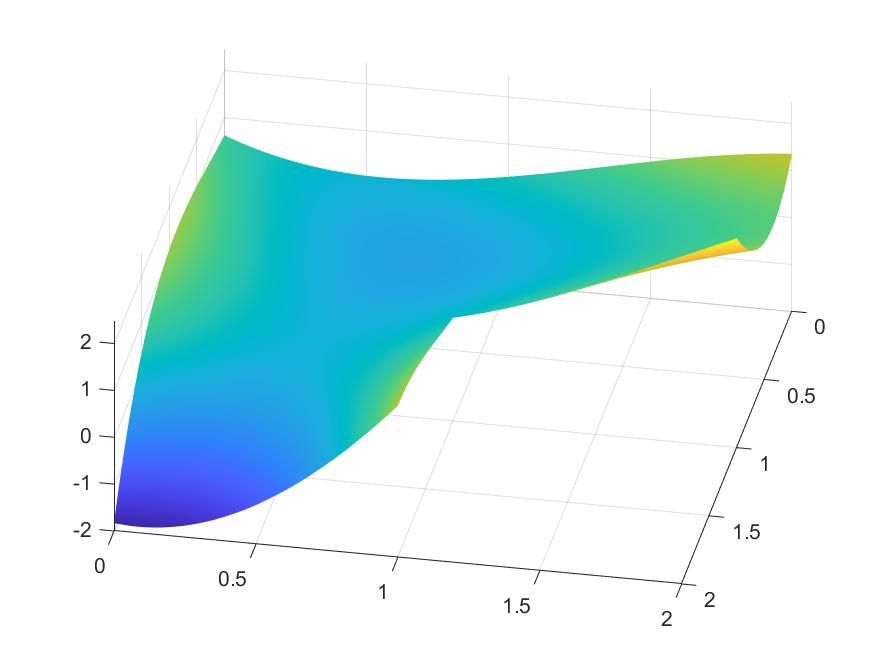}&
\includegraphics[width=0.2\linewidth]{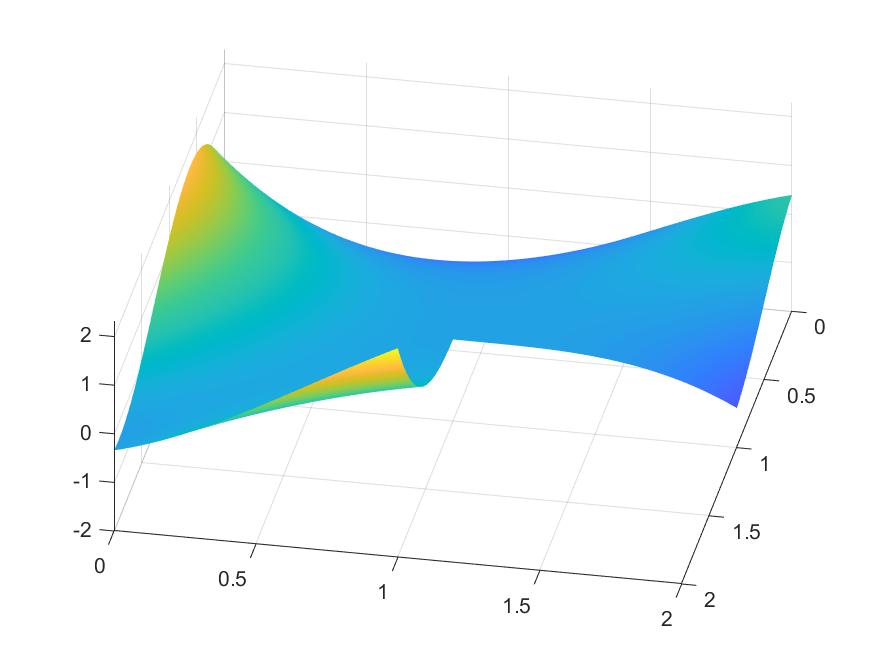}\\
\includegraphics[width=0.2\linewidth]{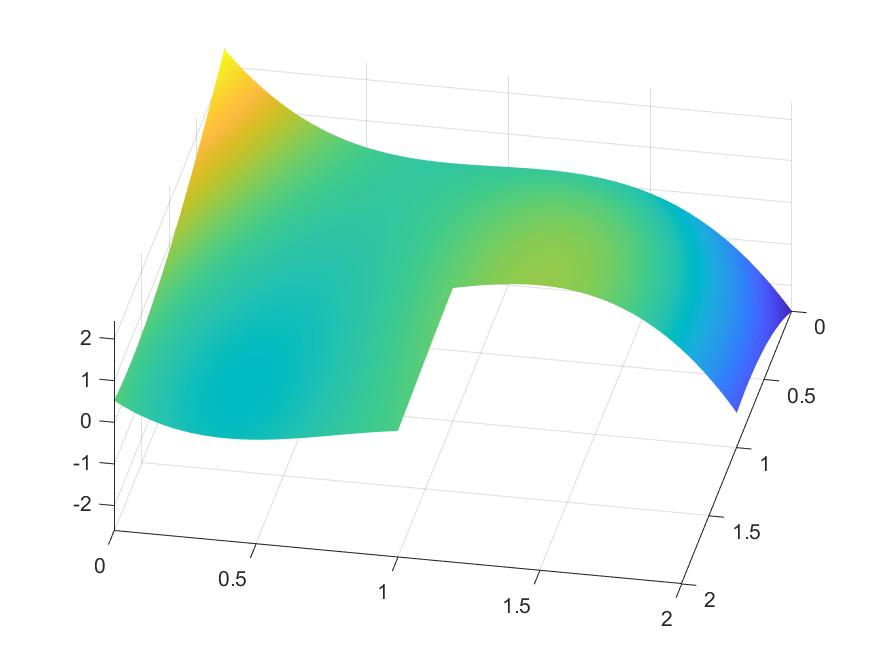}&
\includegraphics[width=0.2\linewidth]{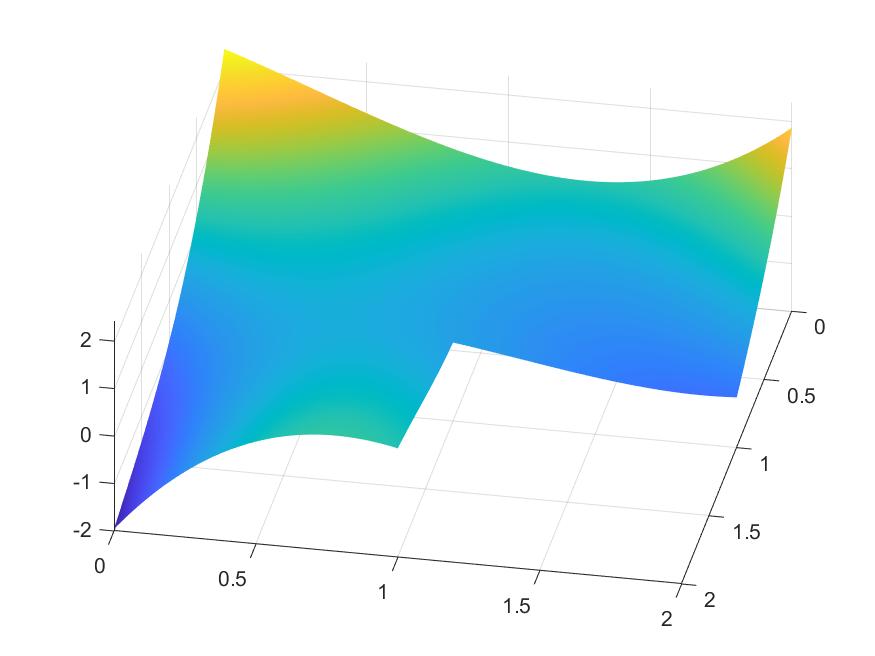}&
%
\includegraphics[width=0.2\linewidth]{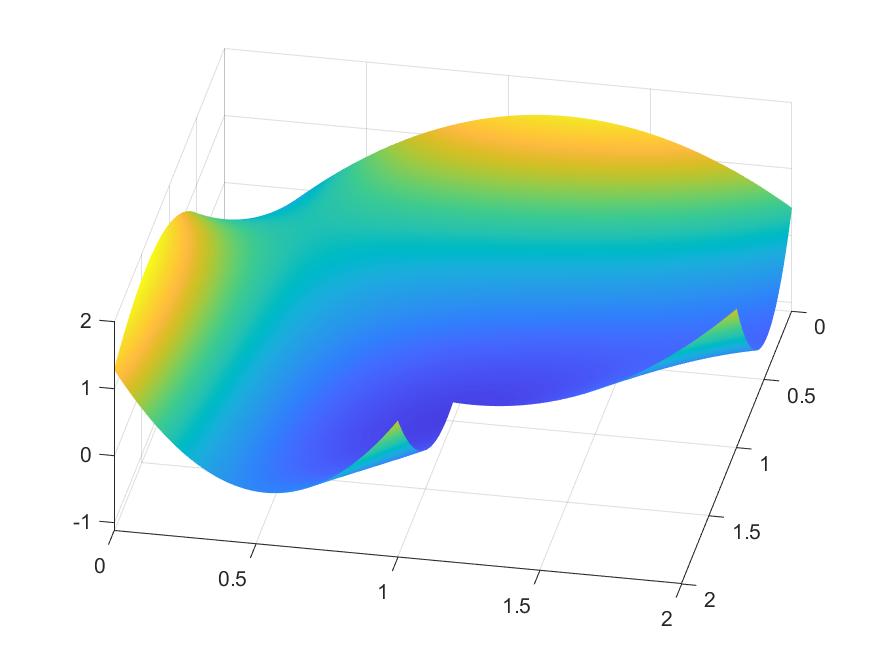}&
\includegraphics[width=0.2\linewidth]{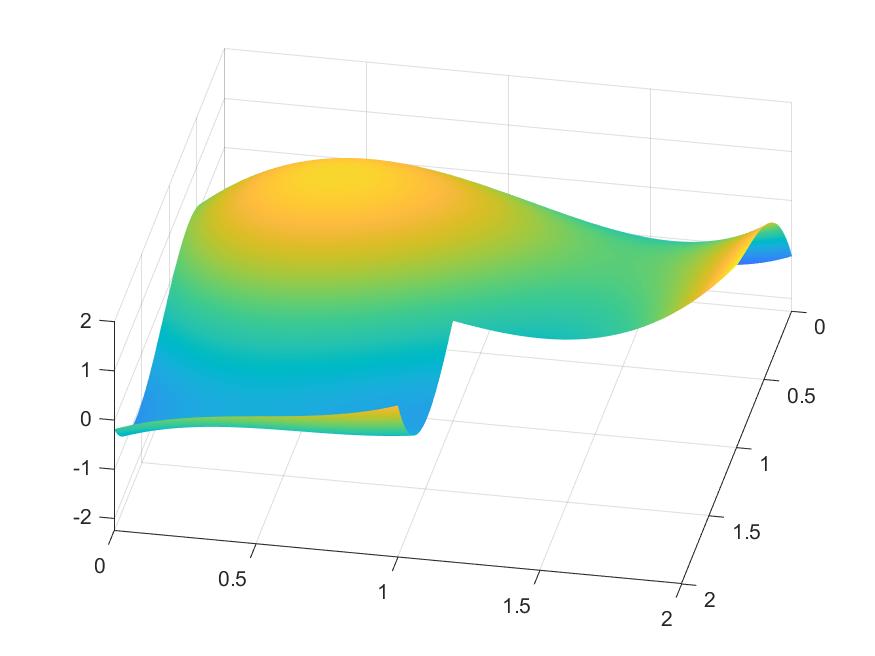}\\
\includegraphics[width=0.2\linewidth]{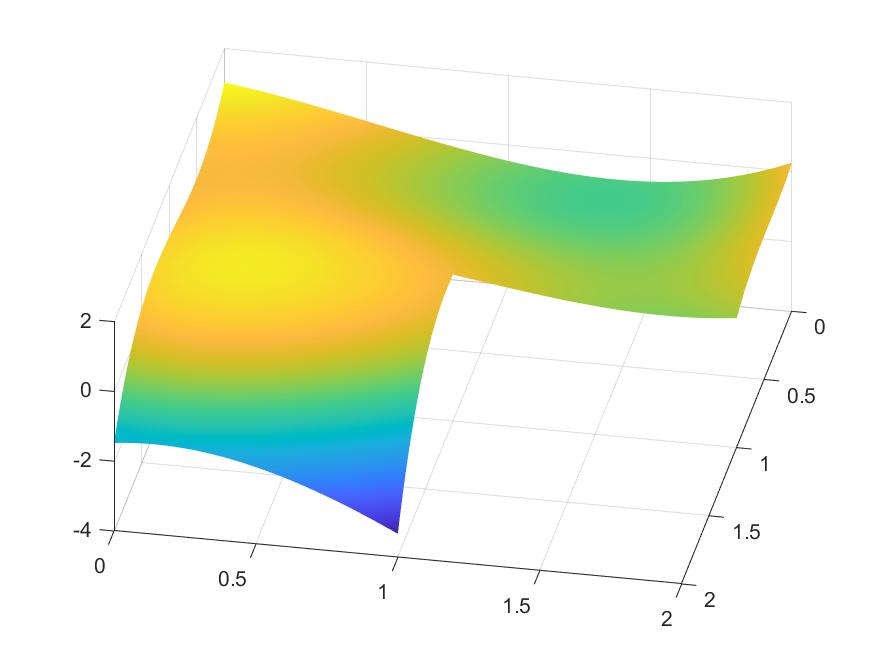}&
\includegraphics[width=0.2\linewidth]{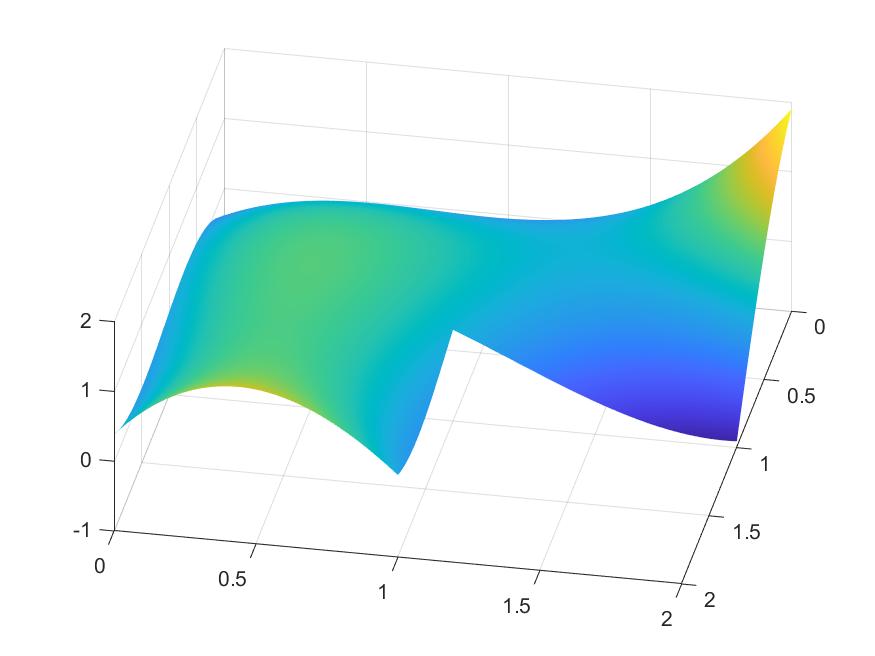}&
\end{tabular}
\caption{Cubic Orthogonal Polynomials over L-shaped Domain (Continued) \label{cubicop2}}
\end{figure}

\end{example}

\begin{example}
	Consider a polygon $P$ which is an L-shaped domain again.  For degree $d=4$, 
	we find bivariate quartic orthonormal polynomials by using Algorithm~\ref{alg1}. 
Let us present	their graphs in Figure~\ref{quaticop}.  
	
	\begin{figure}[htpb]
	\centering
		\begin{tabular}{ccc}
		\includegraphics[width=0.2\linewidth]{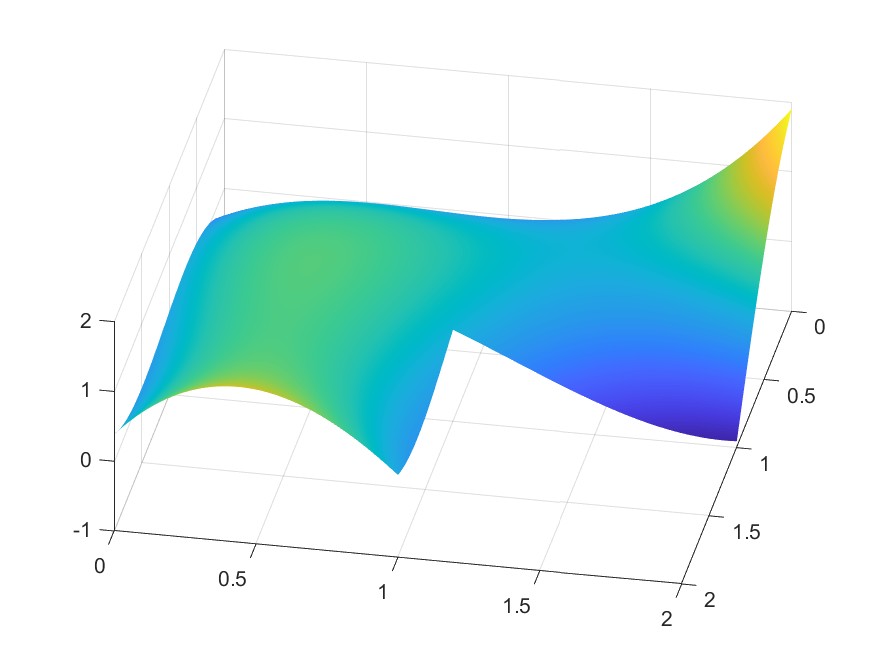} &
		\includegraphics[width=0.2\linewidth]{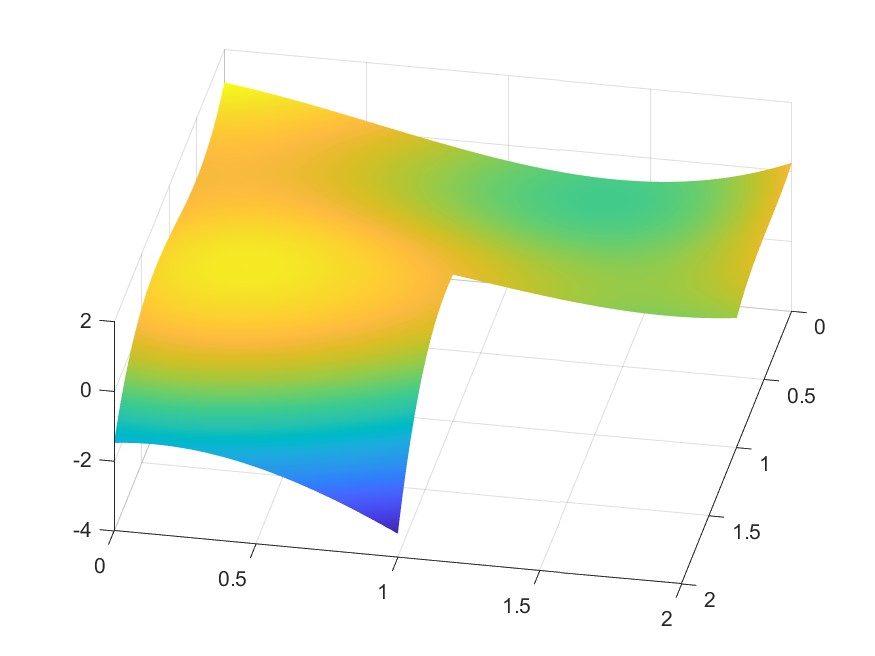} &
		\includegraphics[width=0.2\linewidth]{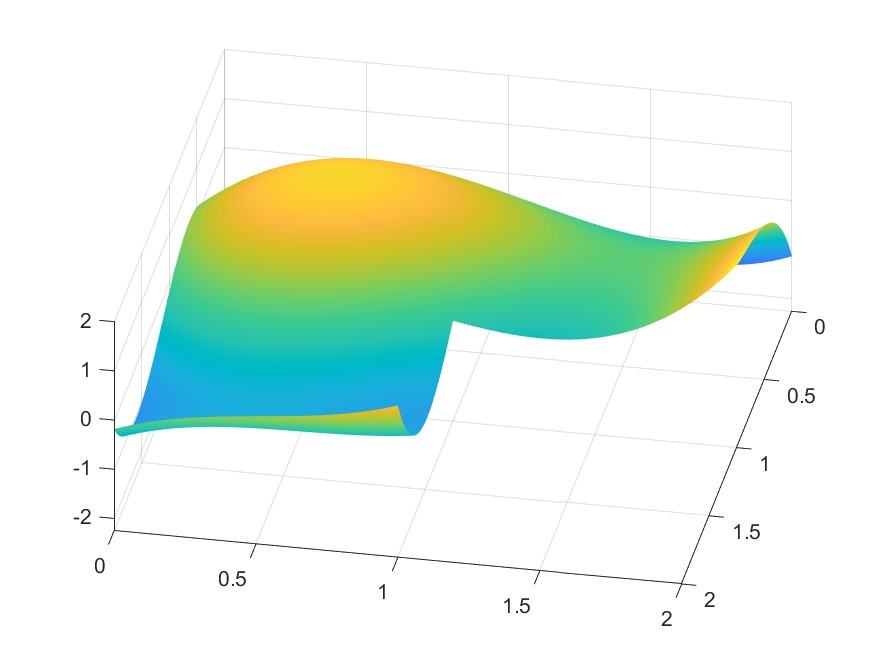}\cr
		\includegraphics[width=0.2\linewidth]{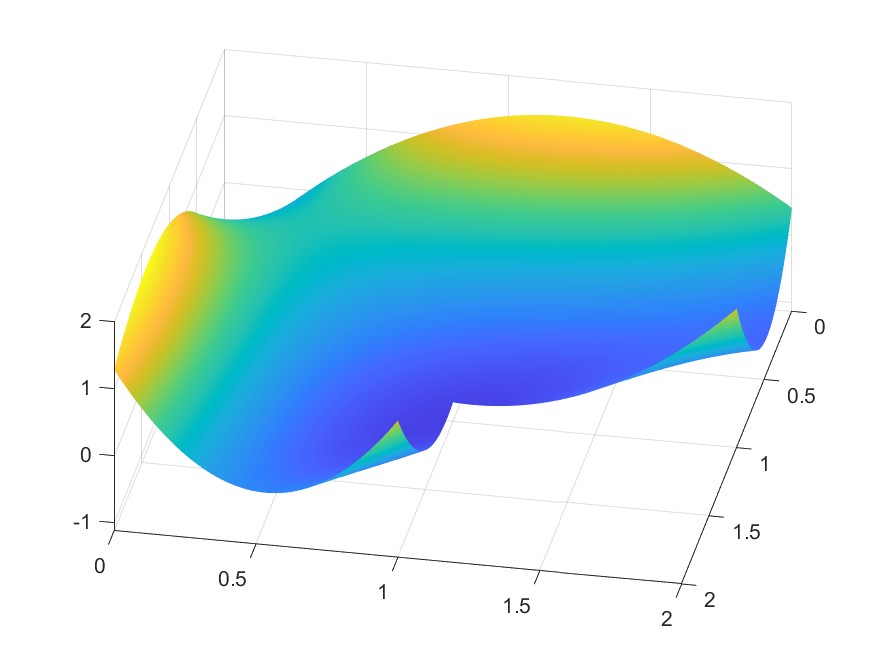} &
		\includegraphics[width=0.2\linewidth]{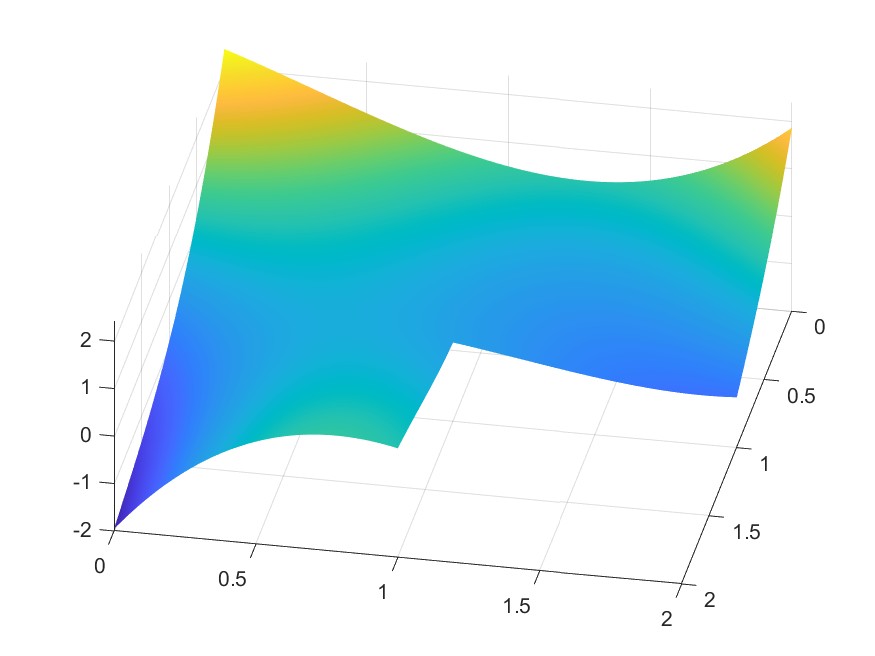} &
		\includegraphics[width=0.2\linewidth]{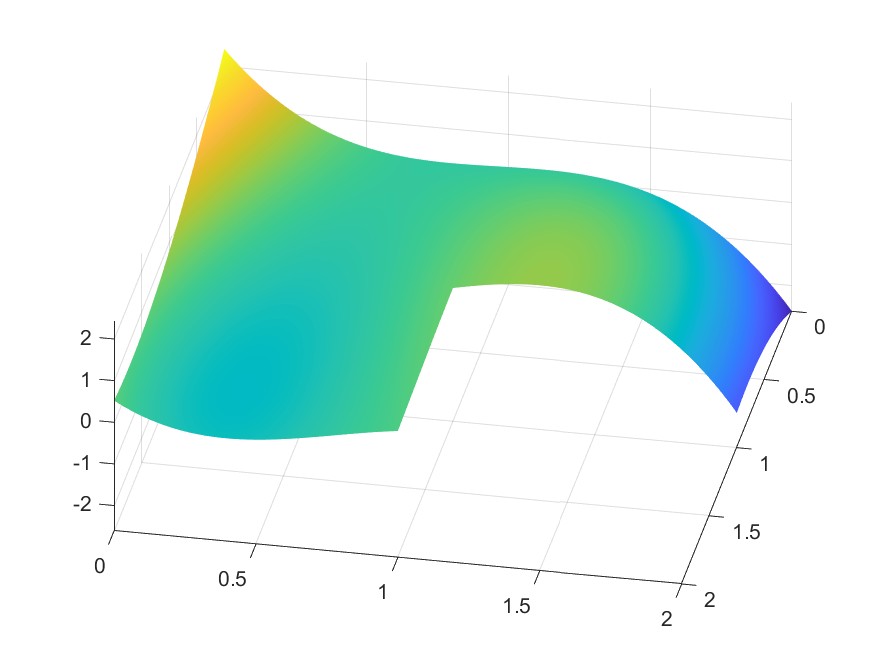}\cr
	 \includegraphics[width=0.2\linewidth]{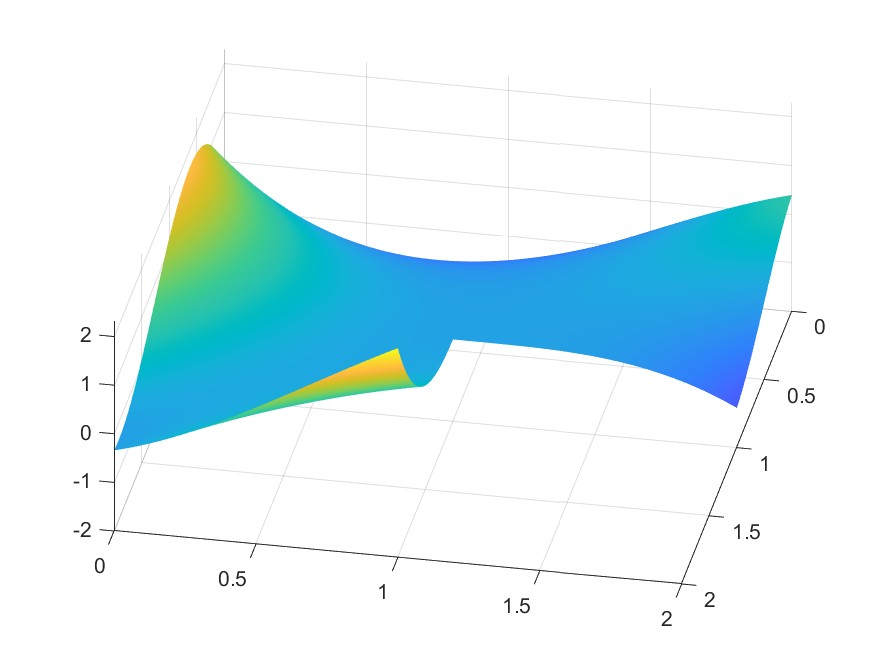} &
		\includegraphics[width=0.2\linewidth]{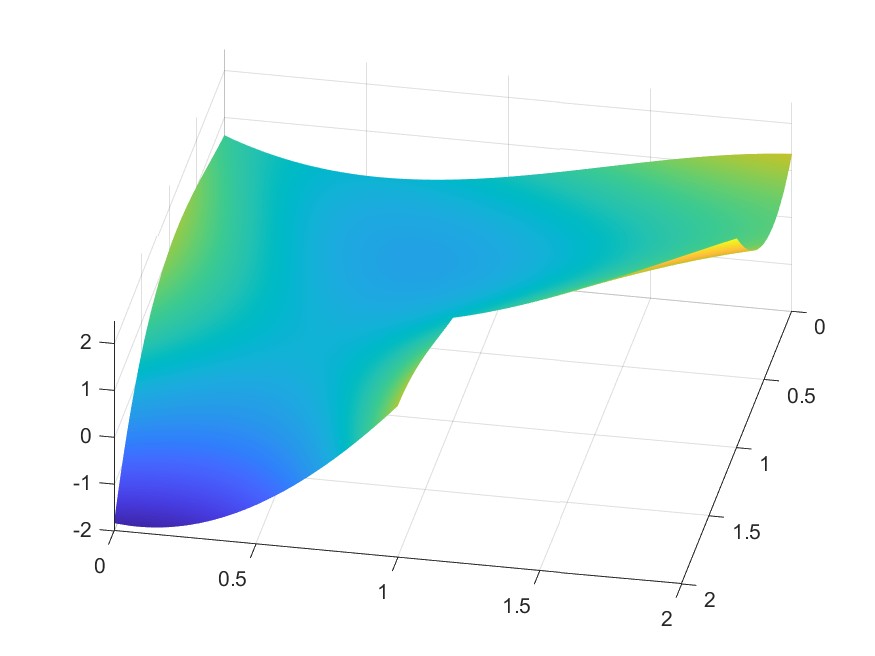}&
		\includegraphics[width=0.2\linewidth]{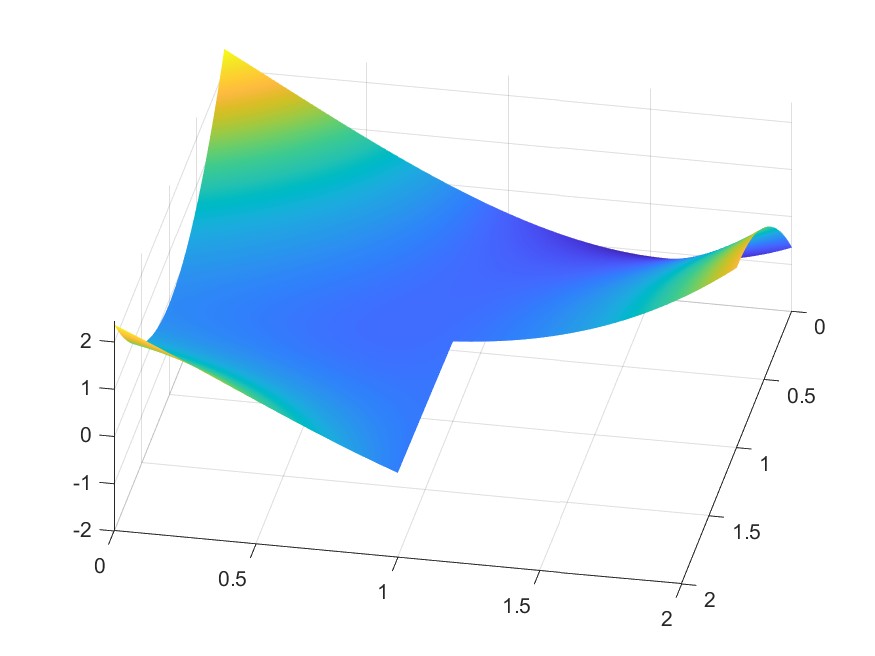} \cr
		\includegraphics[width=0.2\linewidth]{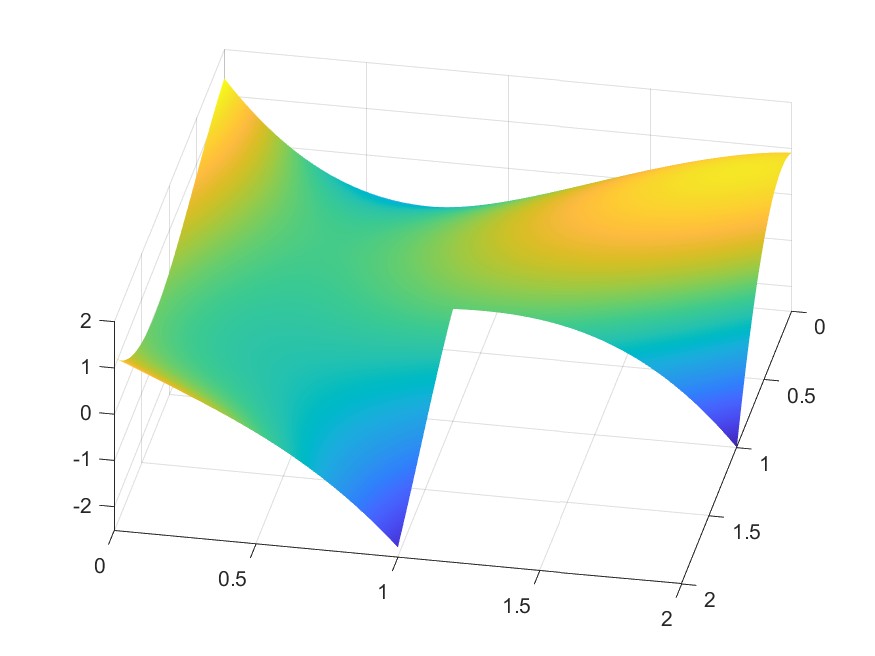}&
		\includegraphics[width=0.2\linewidth]{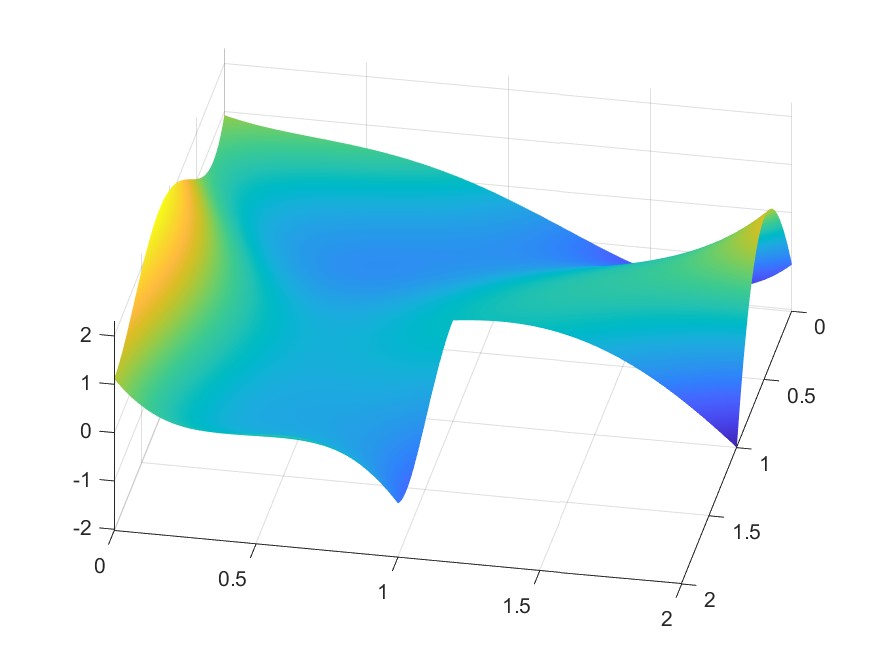}&
		\includegraphics[width=0.2\linewidth]{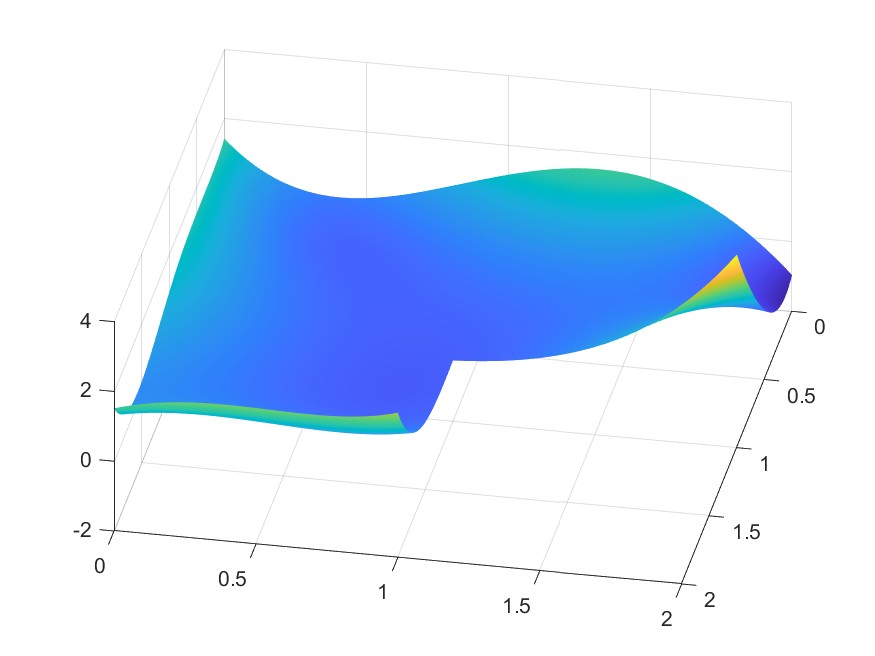}\cr
		\includegraphics[width=0.2\linewidth]{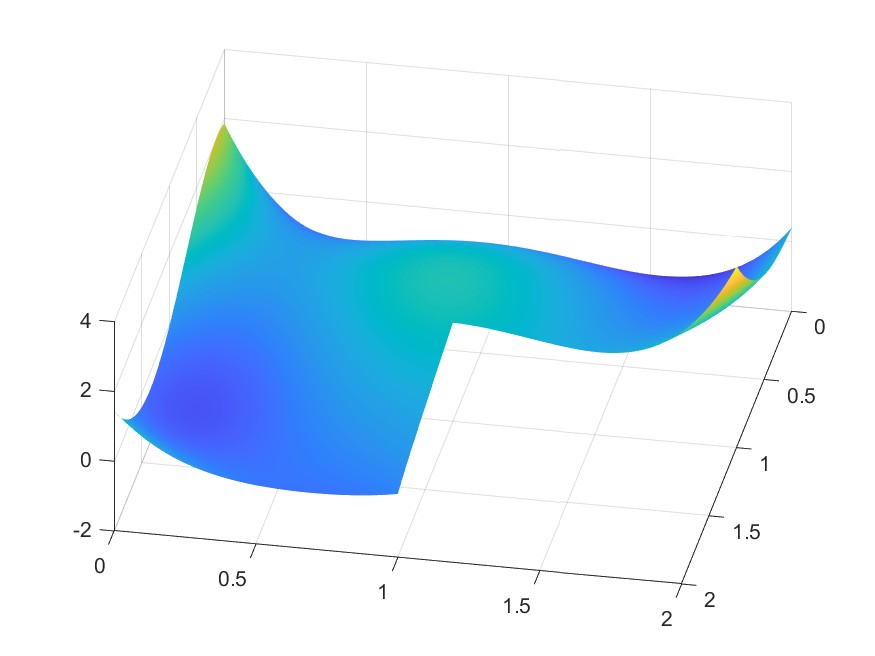}&
	\includegraphics[width=0.2\linewidth]{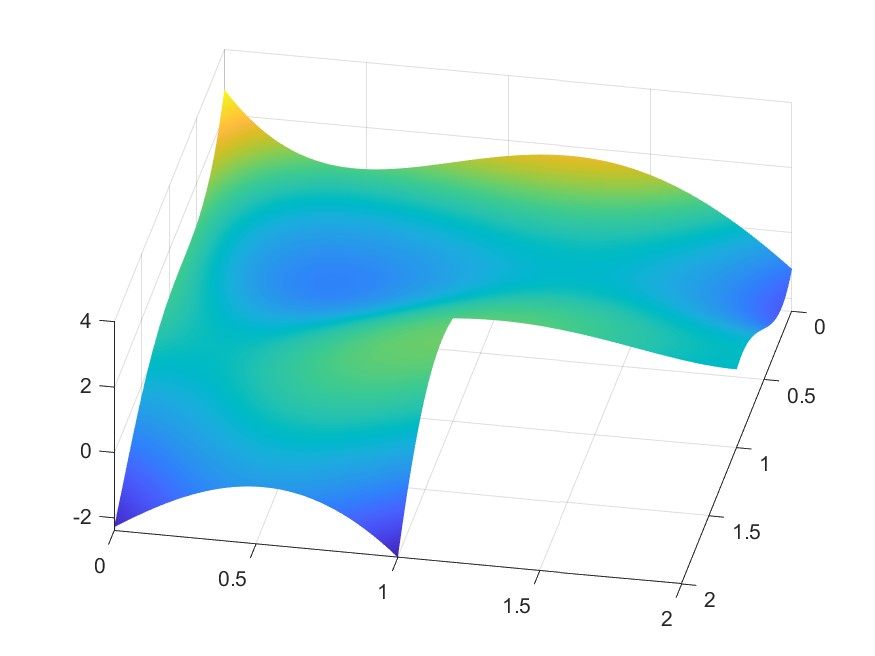}& 
	\includegraphics[width=0.2\linewidth]{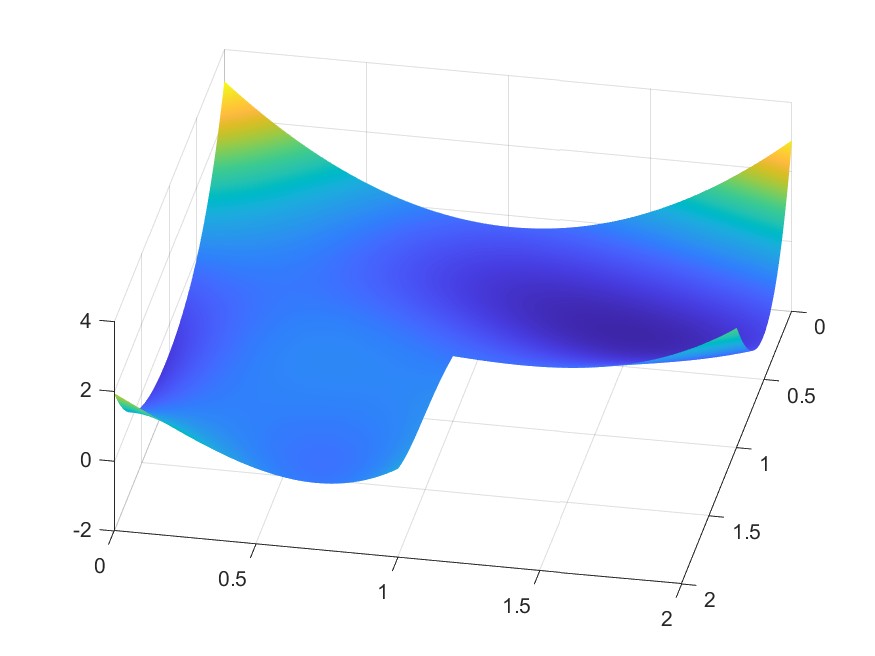}
	\end{tabular}
		\caption{Quartic Orthonormal Polynomials over L-shaped Domain \label{quaticop}}
	\end{figure}
\end{example}

When the polygon is not very complicated, 
the computation of these orthonormal polynomials is quick and easy.  
Only a few minutes are needed.  
When $P$ is complicated and $d\gg 1$, the computation can be long. 
Mainly, the major computational time is spent on the computation of 
the null space of $H_1$.  
Let us present orthonormal polynomials over a more complicated domain of interest. 

\begin{example}
Let us consider a domain which is a flower shape with one hole in the middle. We use 
Algorithm~\ref{alg1} to construct bivariate 
orthonormal polynomials of $d=3$.   Their graphs are presented 
in Figure~ \ref{OPflower2} together with the zero curves of these orthogonal polynomials.
	\begin{figure}[htpb]
	\begin{tabular}{ccc}
		\includegraphics[width=0.3\linewidth]{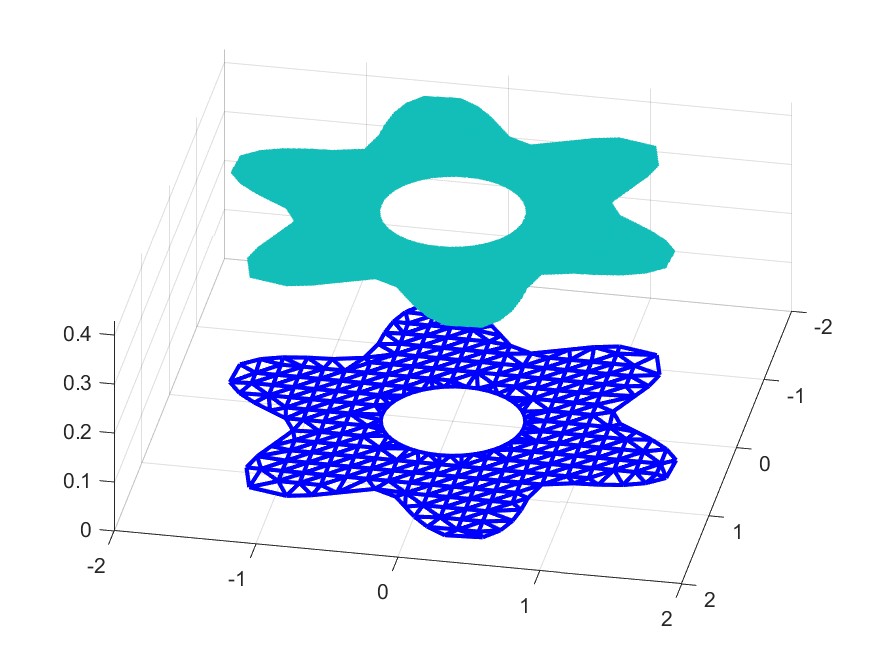} &
		\includegraphics[width=0.3\linewidth]{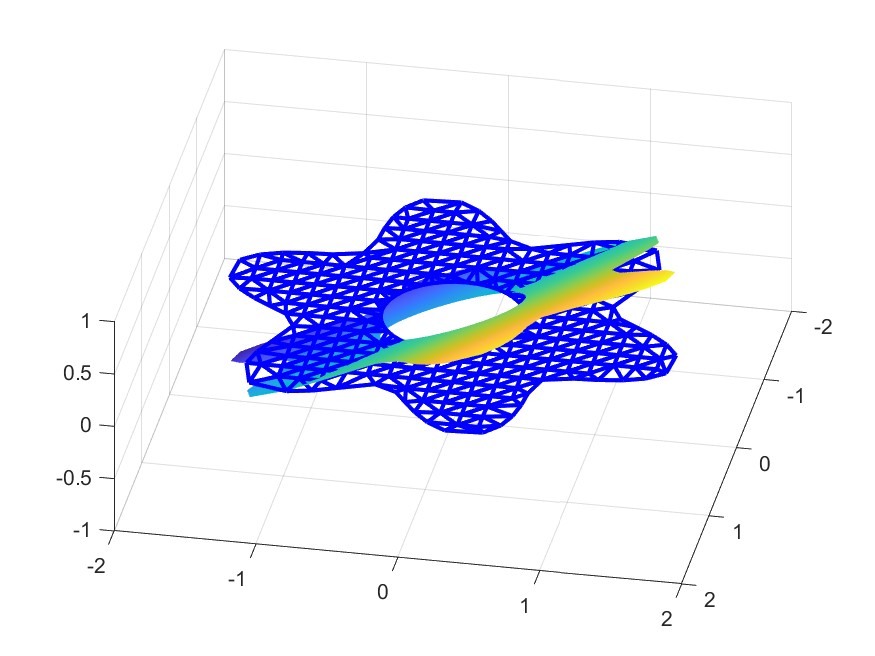} &
		\includegraphics[width=0.3\linewidth]{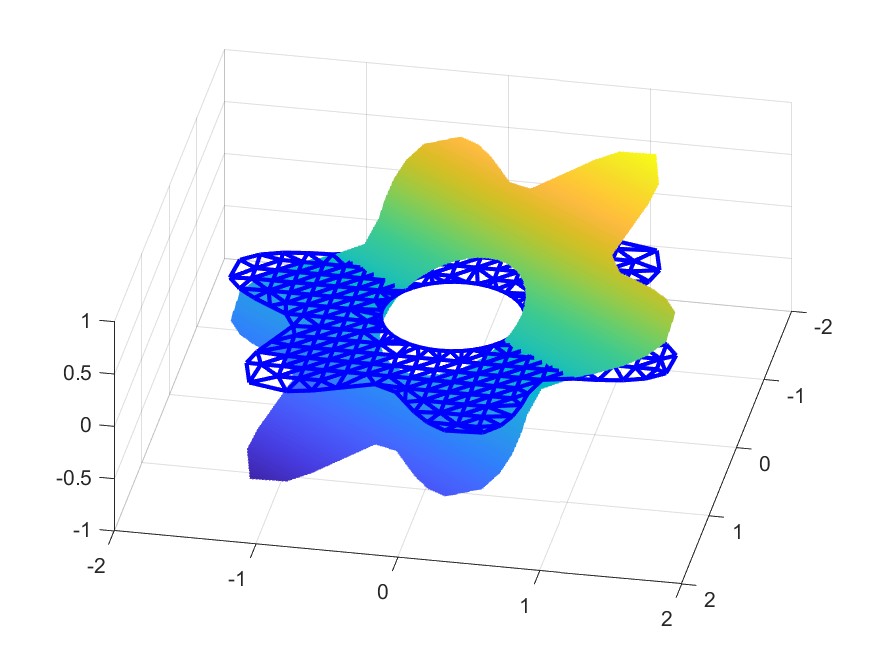}\\
		\includegraphics[width=0.3\linewidth]{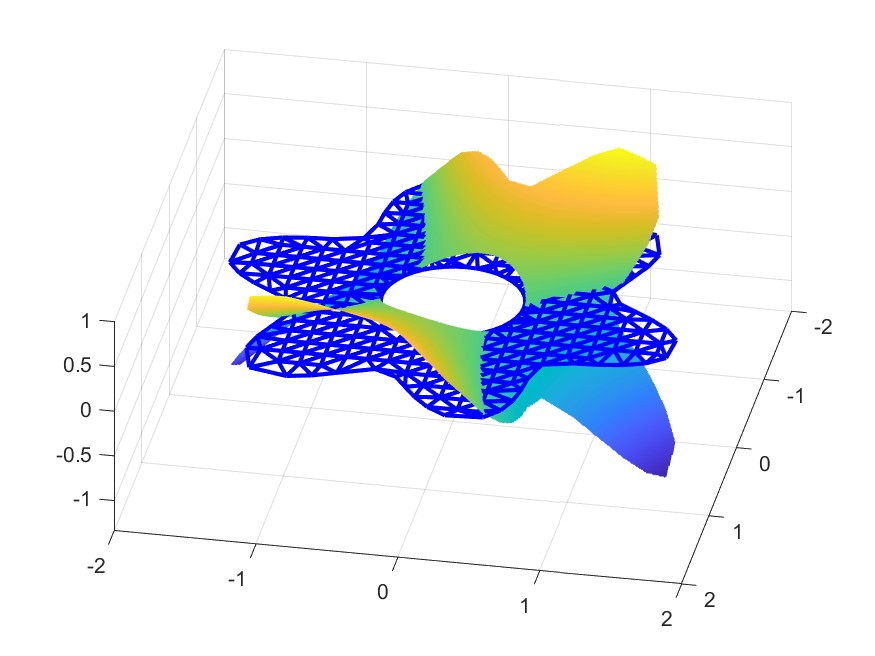} &
				\includegraphics[width=0.3\linewidth]{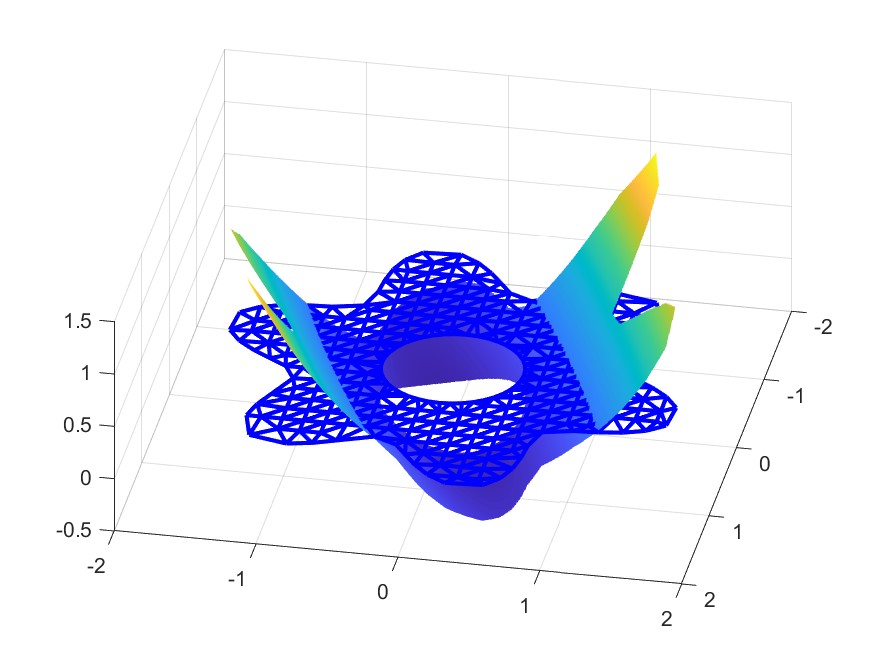} &
		\includegraphics[width=0.3\linewidth]{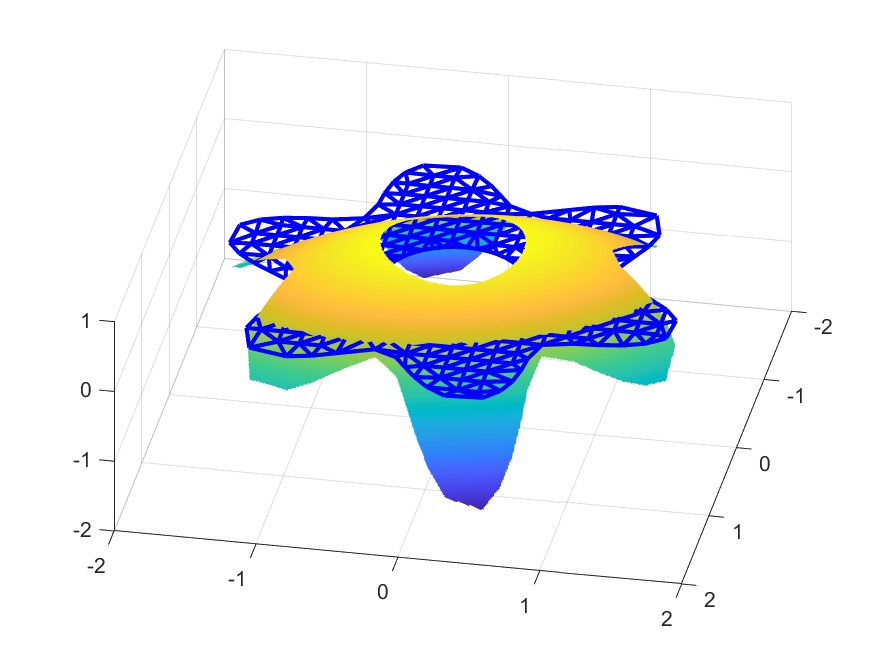}\\
		\includegraphics[width=0.3\linewidth]{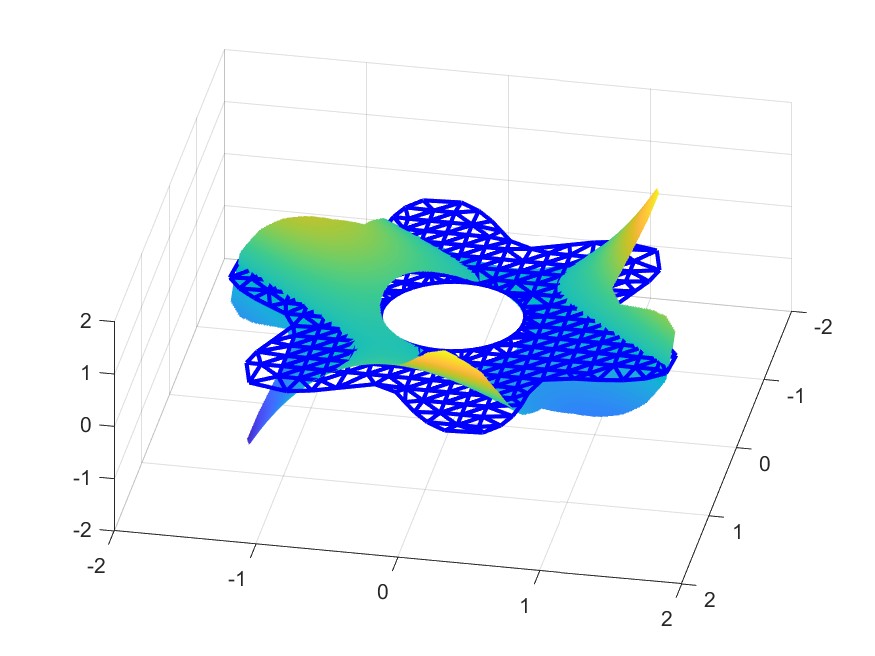}&
		\includegraphics[width=0.3\linewidth]{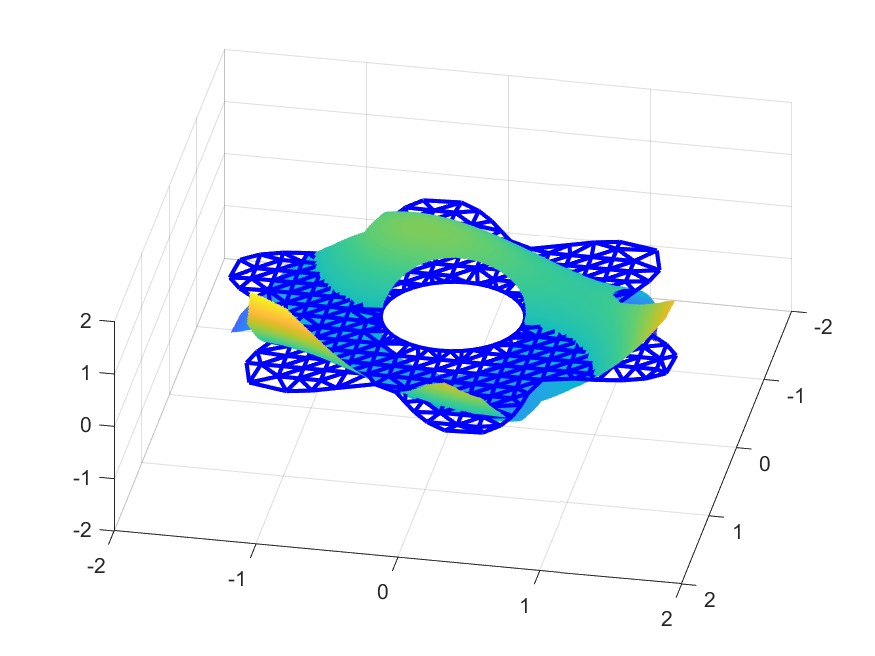}&
		\includegraphics[width=0.3\linewidth]{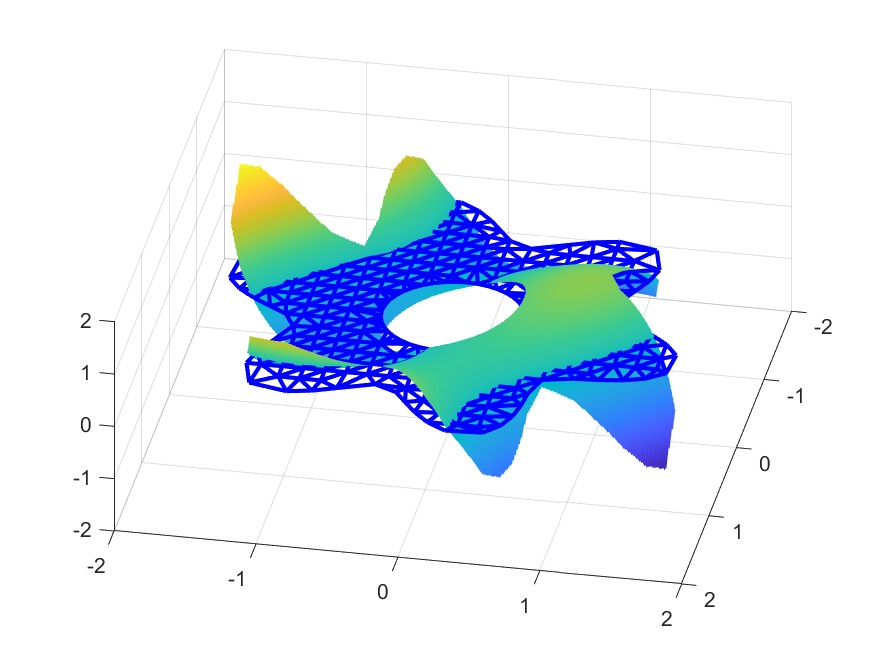}
			\end{tabular}
	\begin{tabular}{ccc}
		\includegraphics[width=0.3\linewidth]{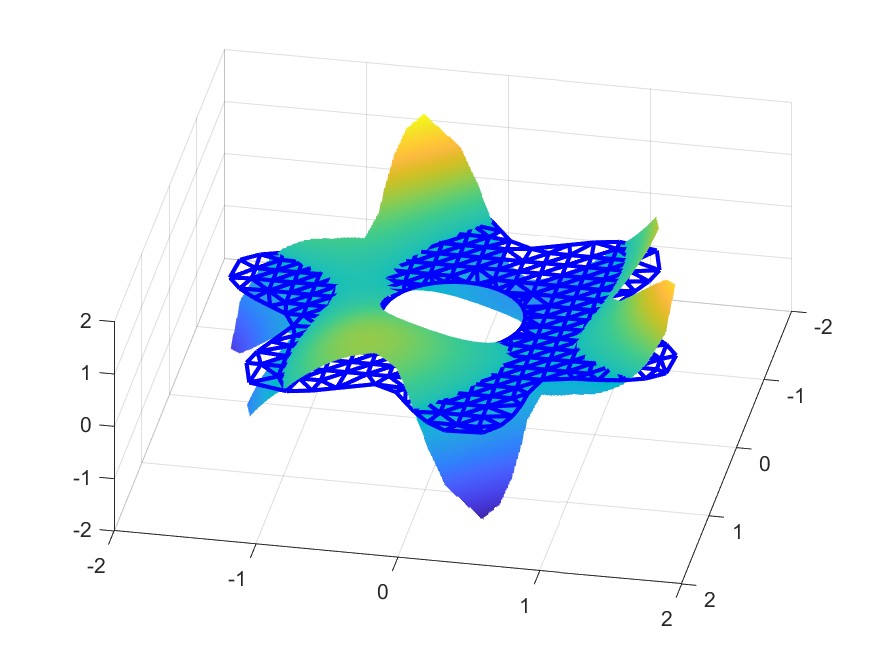}& \includegraphics[width=0.5\linewidth]{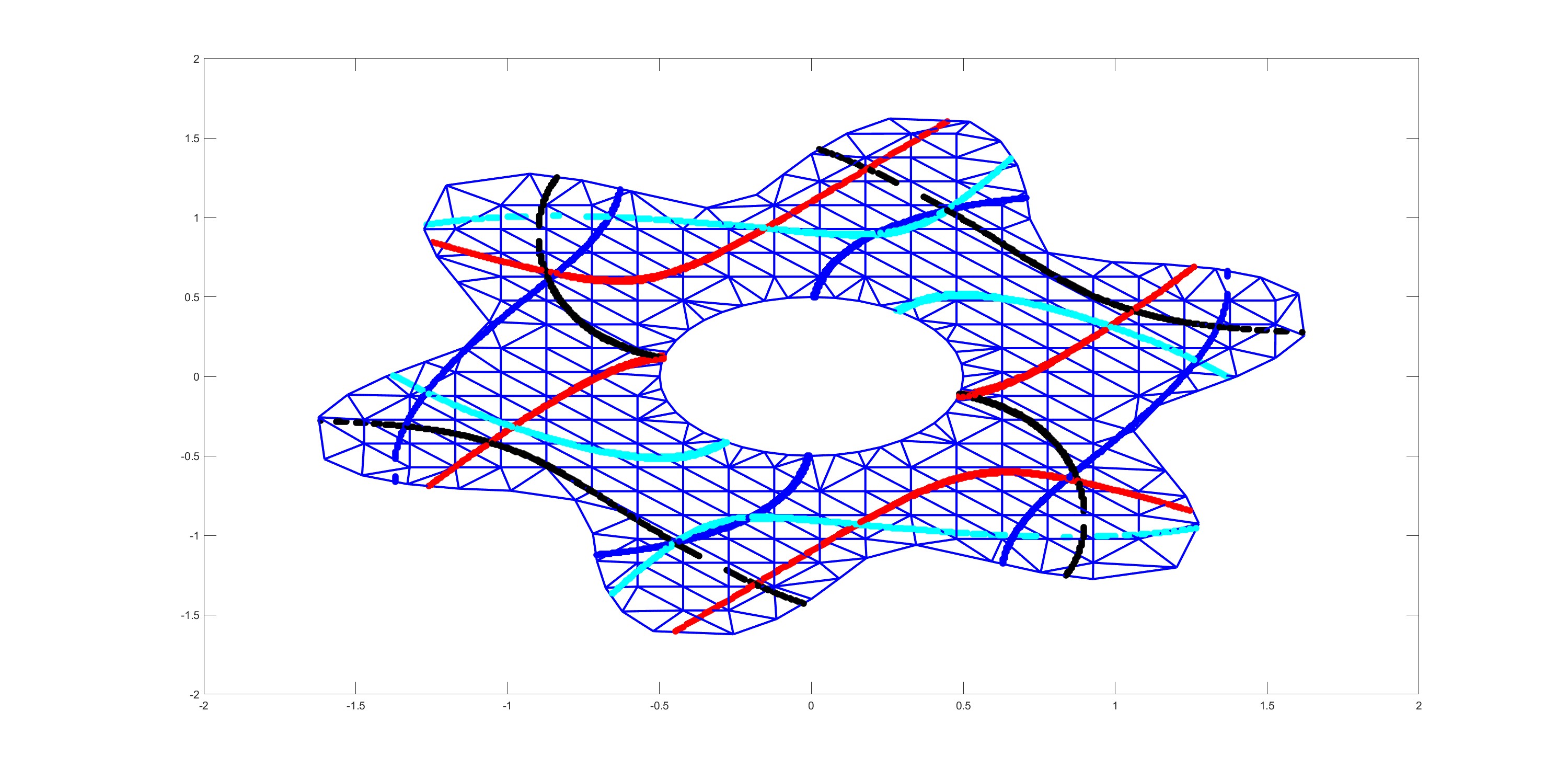}
	\end{tabular}
	\caption{Orthonormal Polynomials of Degree 3 over a Flower-shaped Domain and 
the Zero Curves of the Polynomials of Degree 3 \label{OPflower2}}
\end{figure}
\end{example}

\subsection{Zeros of Orthogonal Polynomials}
Let us first look for the common zeros of orthogonal polynomials over various 
polygonal domains. 
There are many examples that there is a point $p_1\in \Omega$ which is 
the intersection of the 
zero lines of linear orthogonal polynomials for various polygonal domain $\Omega$ 
in Figure~\ref{commonzeros} (except for the J-shaped domain).

	\begin{figure}[htpb]
	\begin{tabular}{cccc}
		\includegraphics[width=0.2\linewidth]{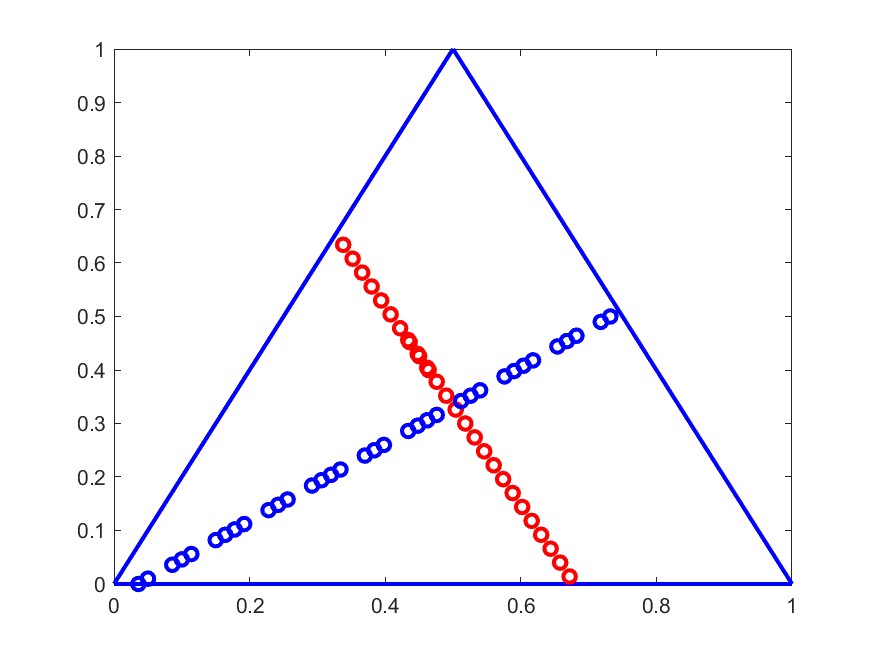}&
		\includegraphics[width=0.2\linewidth]{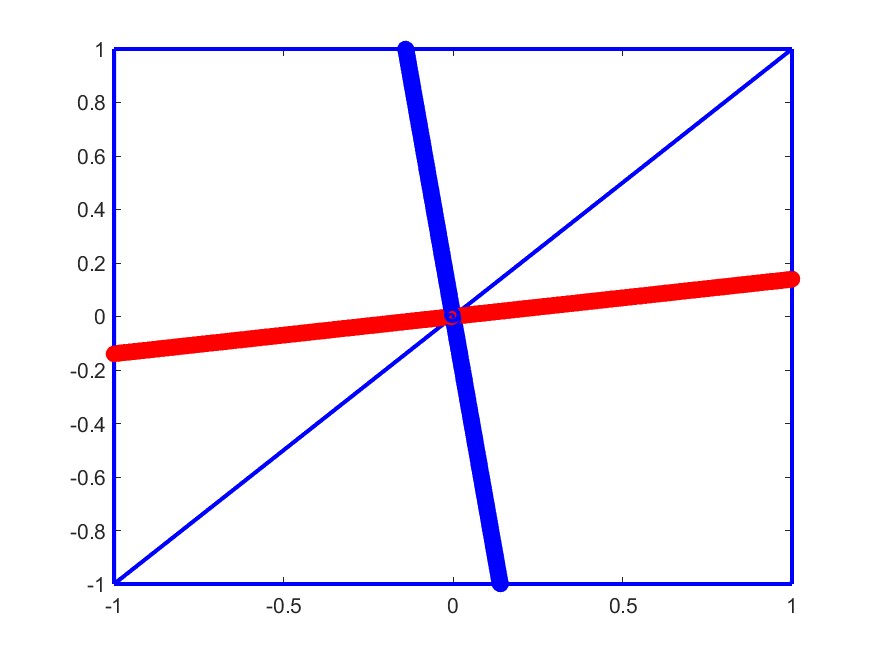}&
		\includegraphics[width=0.2\linewidth]{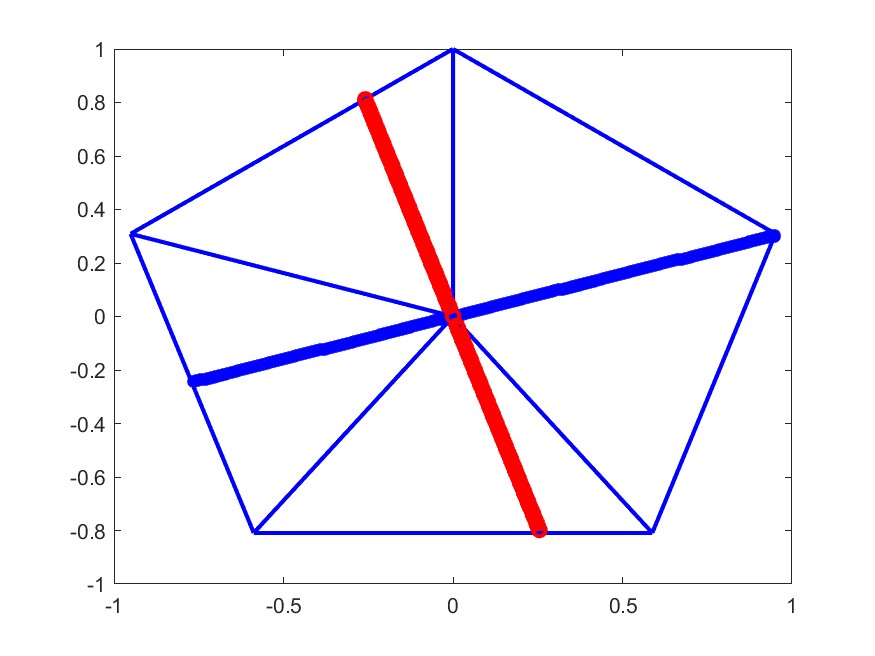}&
		\includegraphics[width=0.2\linewidth]{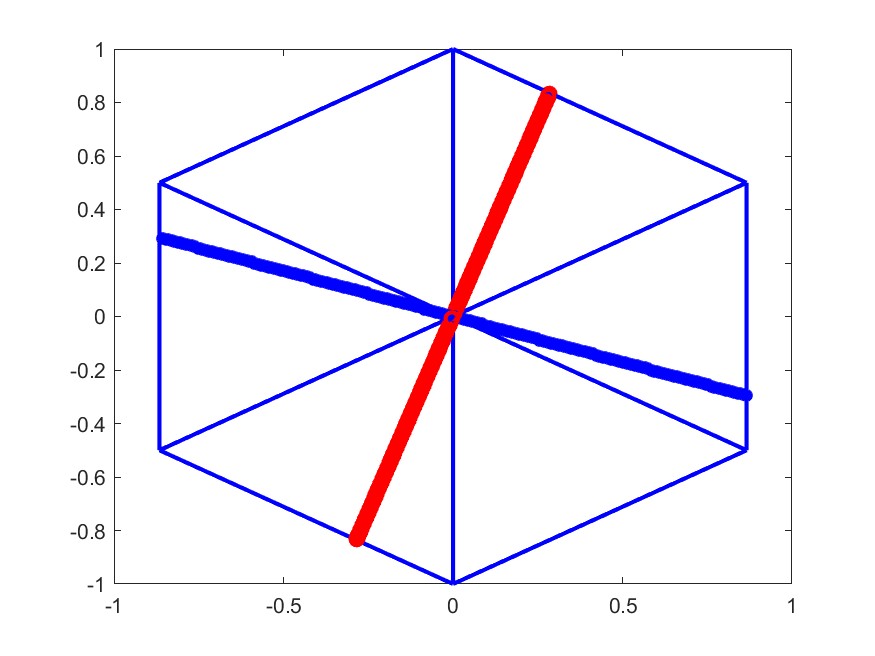}\cr
		\includegraphics[width=0.2\linewidth]{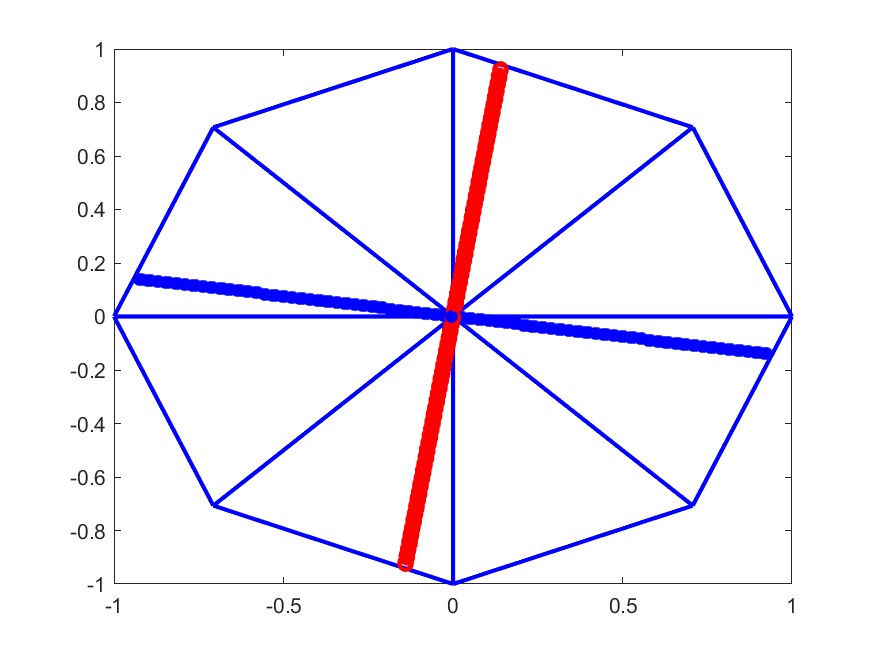}&
		\includegraphics[width=0.2\linewidth]{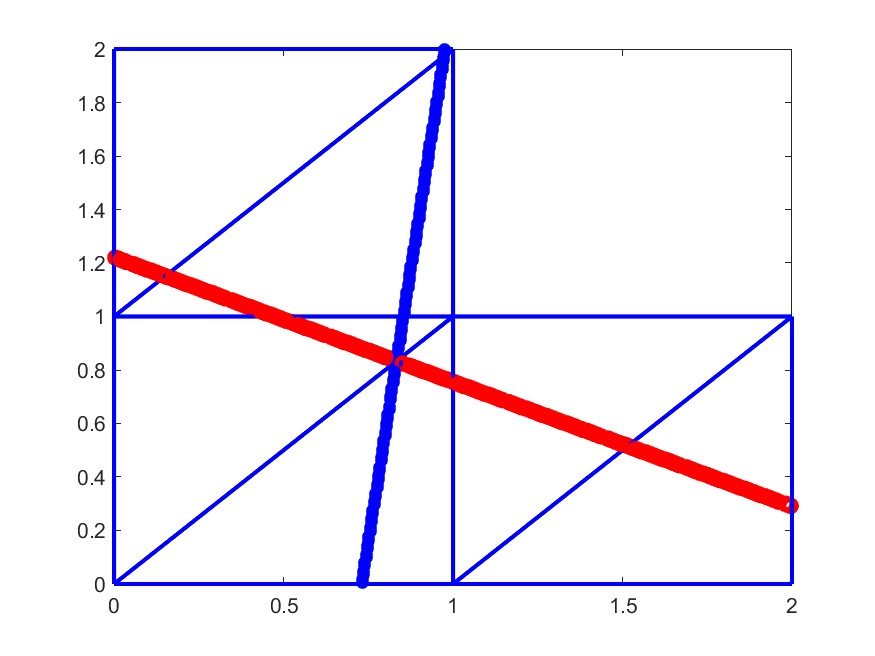}&
				\includegraphics[width=0.2\linewidth]{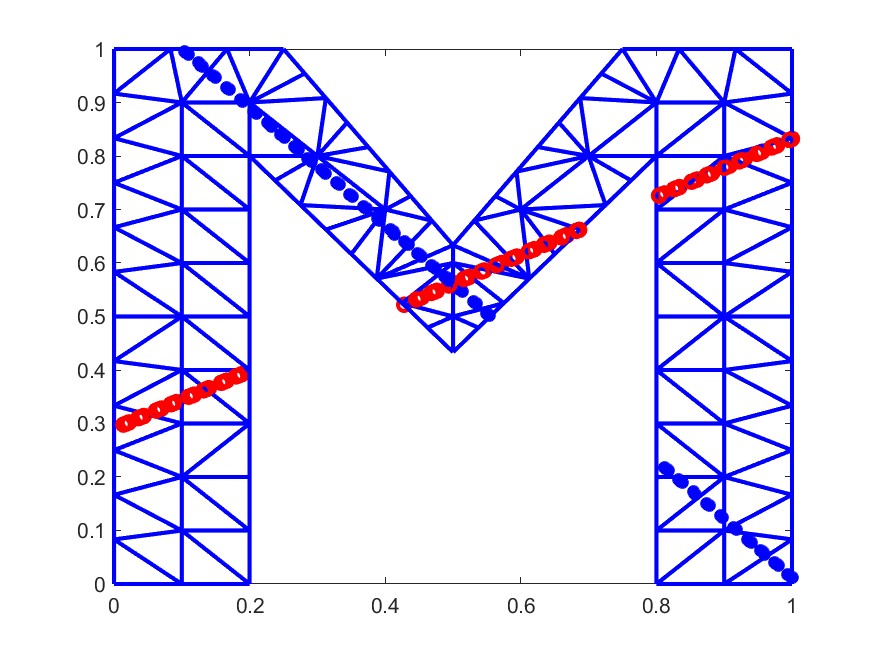}&
						\includegraphics[width=0.2\linewidth]{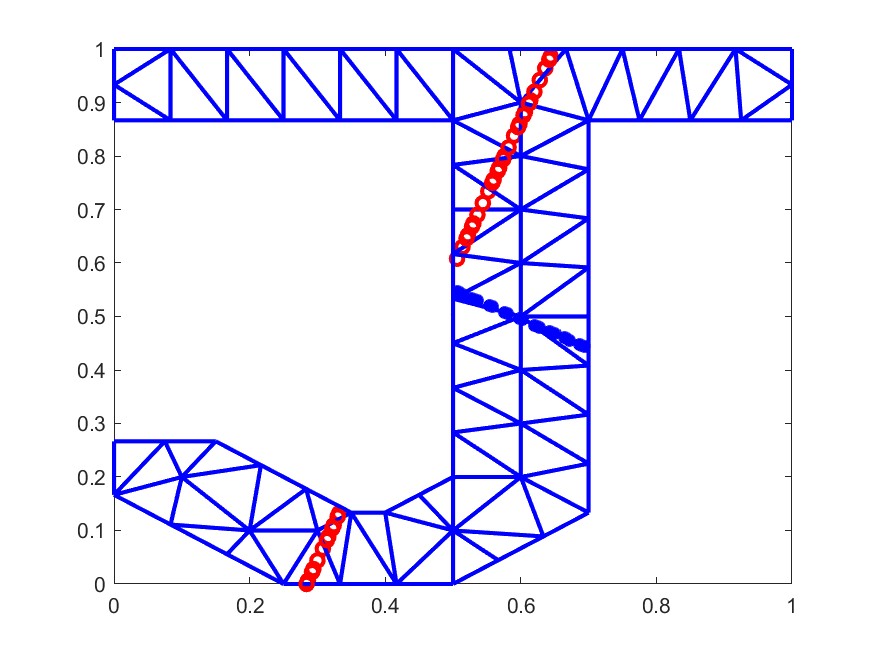}
		\end{tabular}
\caption{Zero lines (red and blue) of Two Orthogonal Linear Polynomials over 
Various Domains \label{commonzeros}}
	\end{figure}

Next let us look for common zero curves of orthogonal polynomials of higher degree 
than $1$.  
	\begin{figure}[htpb]
		\centering 
		\includegraphics[width=0.2\linewidth]{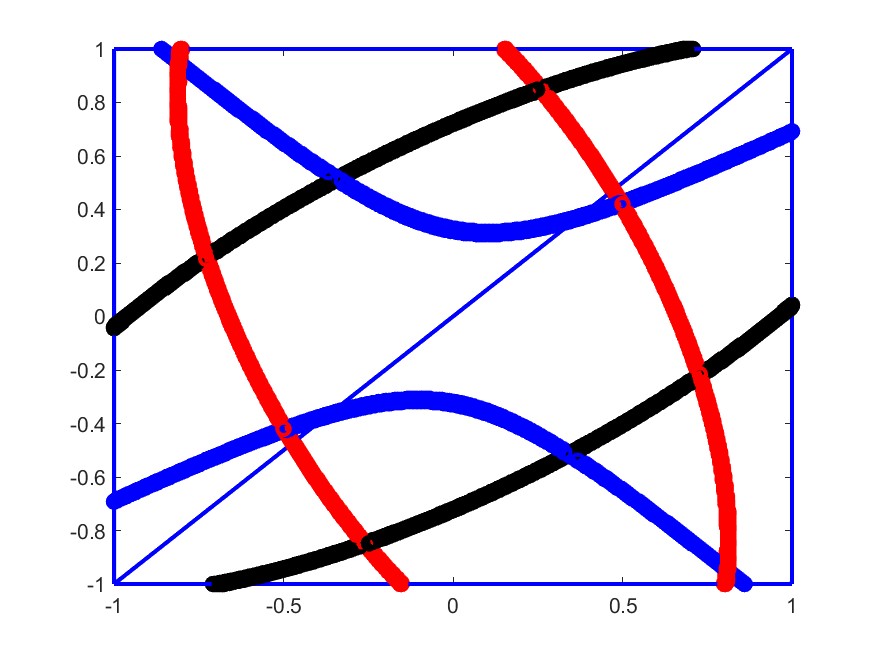}
		\includegraphics[width=0.2\linewidth]{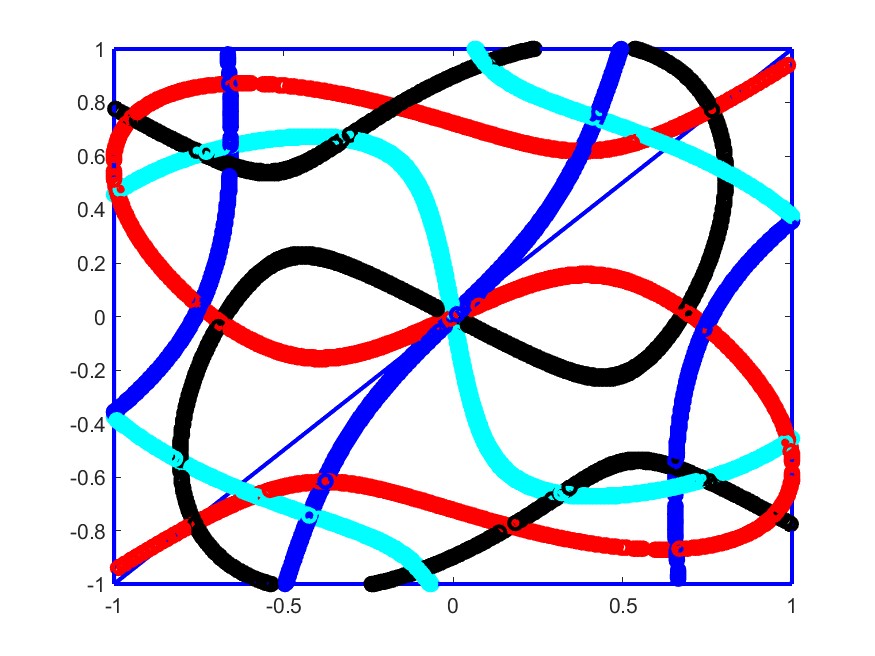}
				\includegraphics[width=0.2\linewidth]{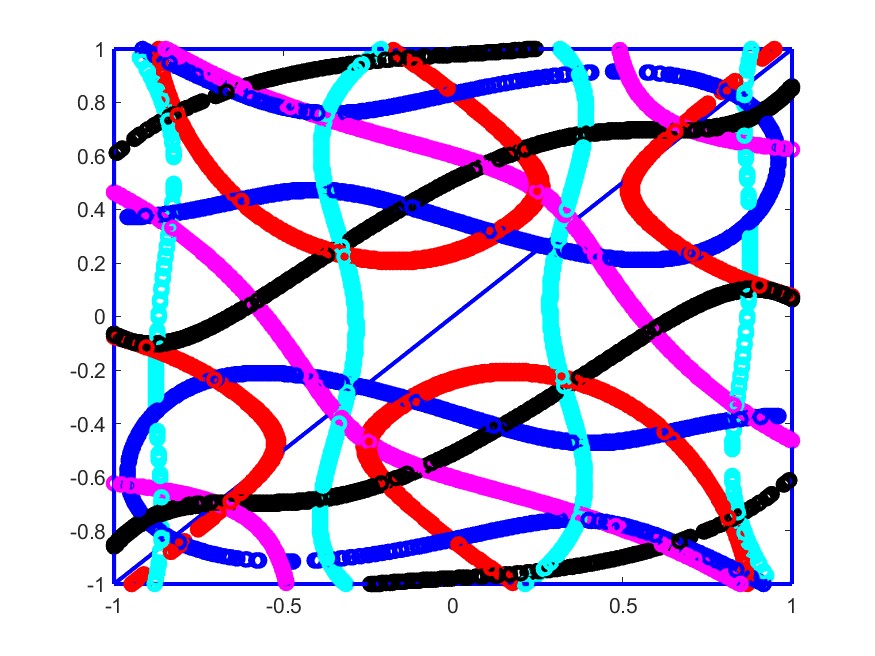}
		\includegraphics[width=0.2\linewidth]{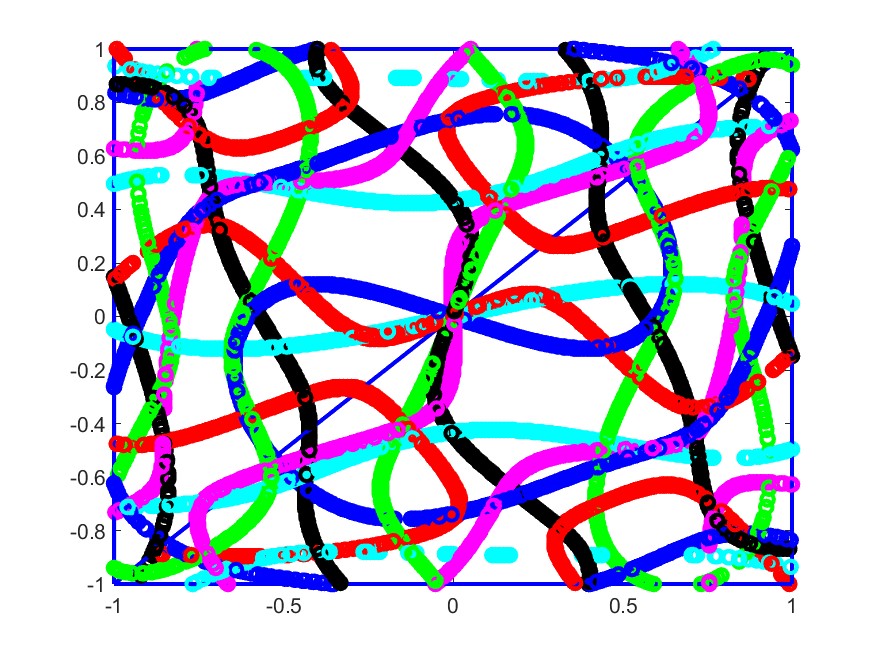}
	\caption{Zero Curves of Orthogonal Polynomials of degrees $2, 3, 4, 5$  \label{zeros1}}
	\end{figure}
 
These graphs in Figures~\ref{zeros1}, \ref{zerosP}, \ref{zerosH}, \ref{zerosO} show that 
there are not enough common zeros for polynomial degree $d=2$ to $d=5$. These
 ruin the hope to build up Gaussian quadrature like the one in the univariate setting. However, 
there is a common zero for all orthogonal polynomials of odd degrees, $d=1, 3, 5$ in
 Figures~\ref{zeros1}, \ref{zerosH}, and \ref{zerosO}.

\begin{example}
Consider a pentagonal domain $\Omega$ and plot the zero curves of orthogonal 
polynomials of degree $2, 3, 4, 5$ in Figure~\ref{zerosP}. 
	\begin{figure}[htpb]
		\centering 
		\includegraphics[width=0.2\linewidth]{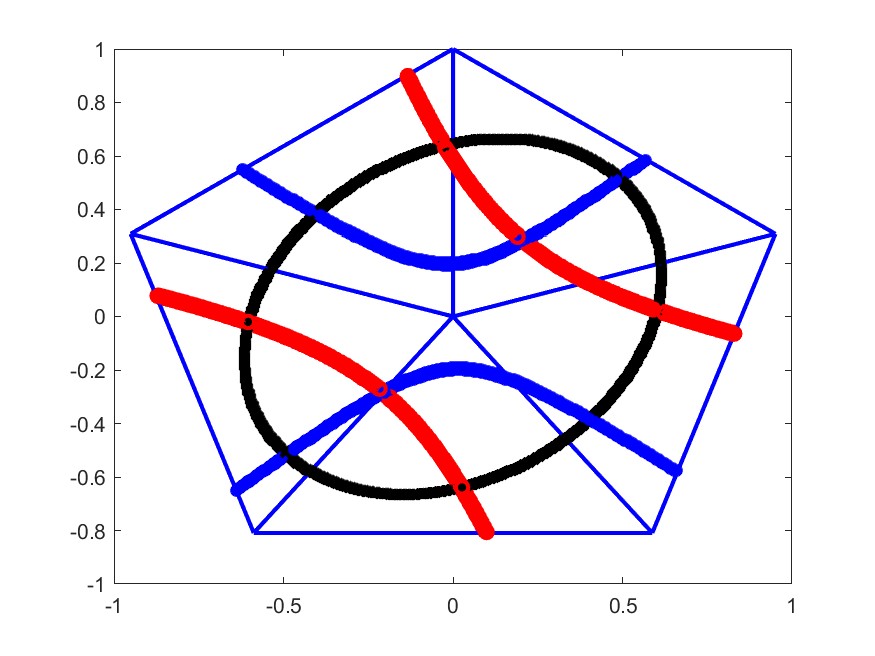}
		\includegraphics[width=0.2\linewidth]{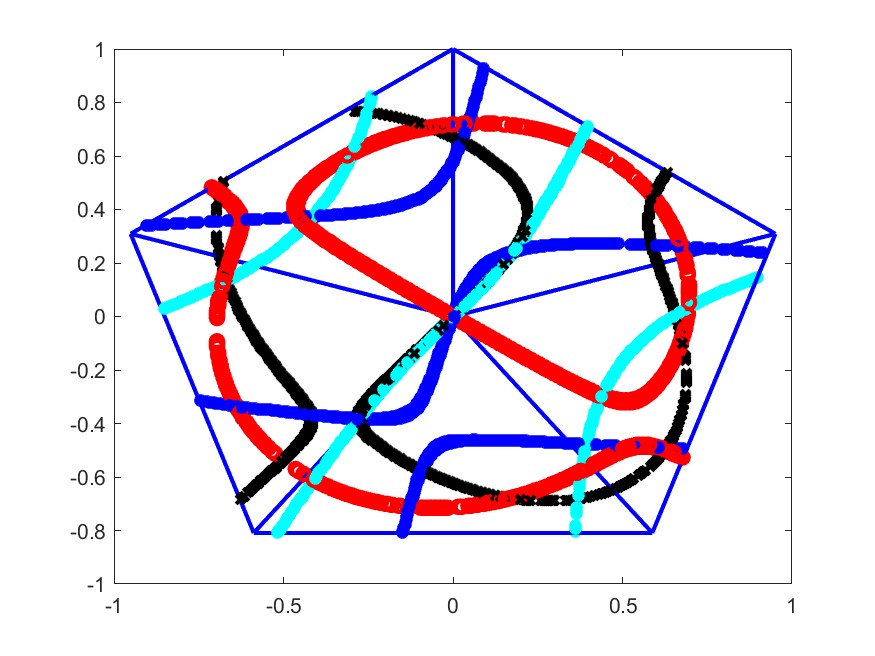}
				\includegraphics[width=0.2\linewidth]{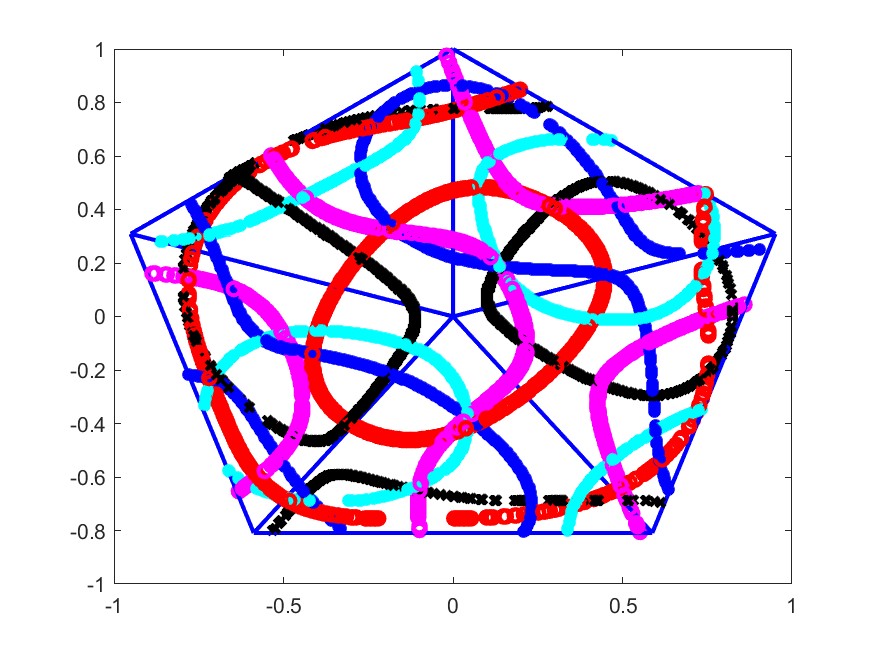}
		\includegraphics[width=0.2\linewidth]{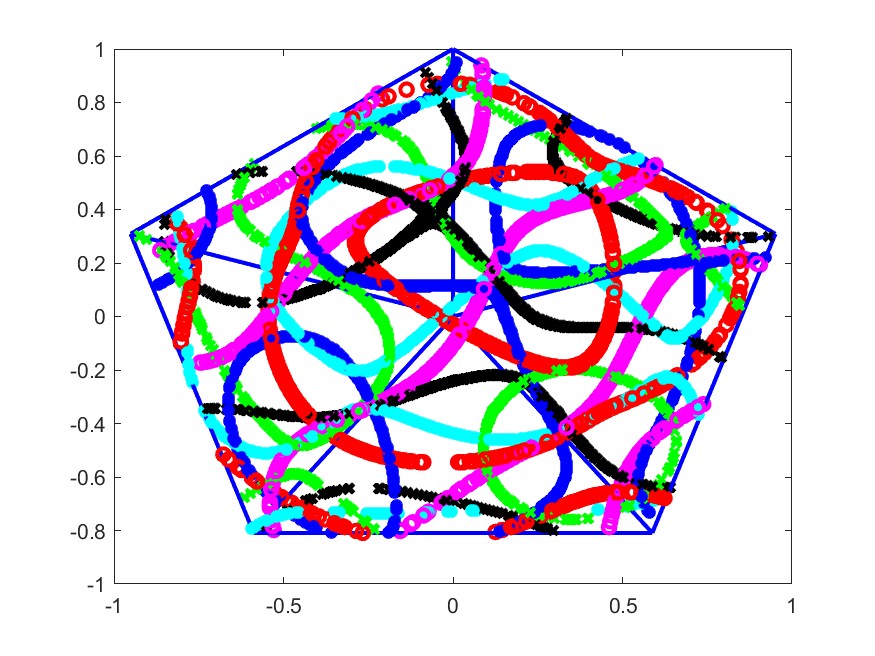}
	\caption{Zero Curves of Orthogonal Polynomials of degrees $2, 3, 4, 5$  \label{zerosP}}
	\end{figure}
We can see from Figure~\ref{zerosP} that there are no common zeros for $P_{d,j},j=0, 1, 
\cdots, d$ for $d=2, 3, 4,$ and $5$. 
\end{example}

\begin{example}
\label{hexagon}
Consider a hexagon $\Omega$ and we plot the zero curves 
of orthogonal polynomials of degree $2--5$ in Figure~\ref{zerosH}. 
	\begin{figure}[htpb]
		\includegraphics[width=0.18\linewidth]{commonzerosHd1.jpg}
		\includegraphics[width=0.18\linewidth]{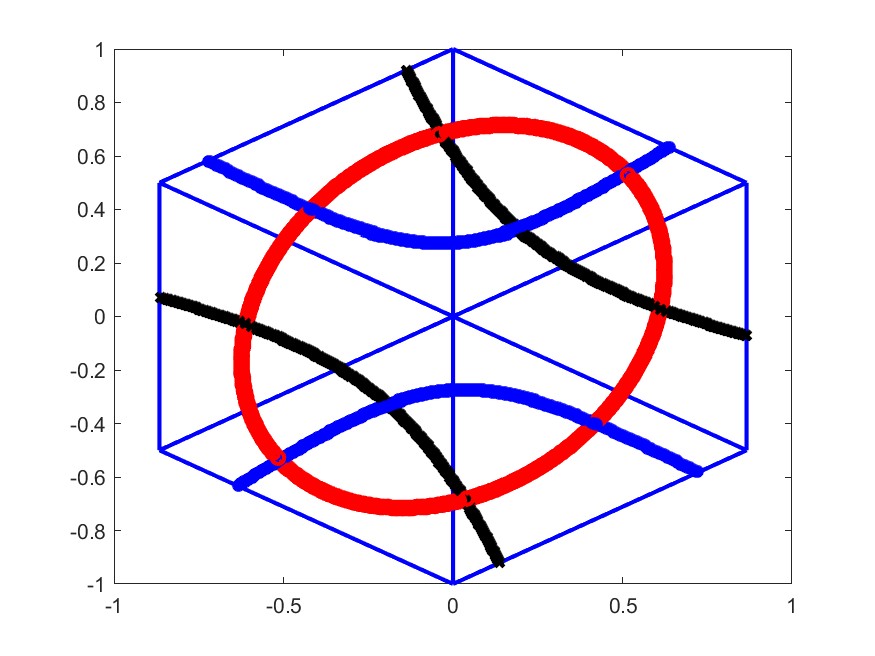}
		\includegraphics[width=0.18\linewidth]{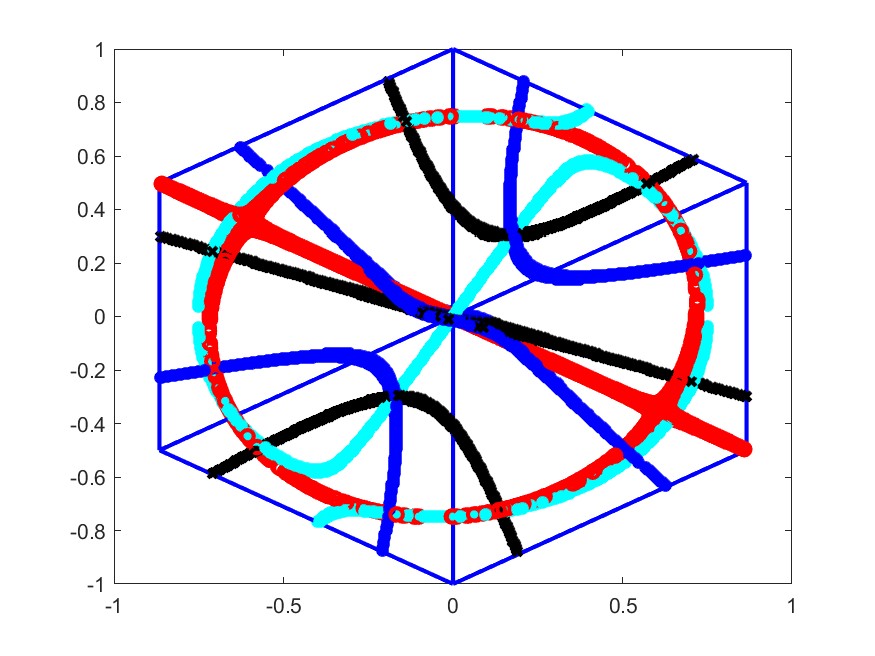}
				\includegraphics[width=0.18\linewidth]{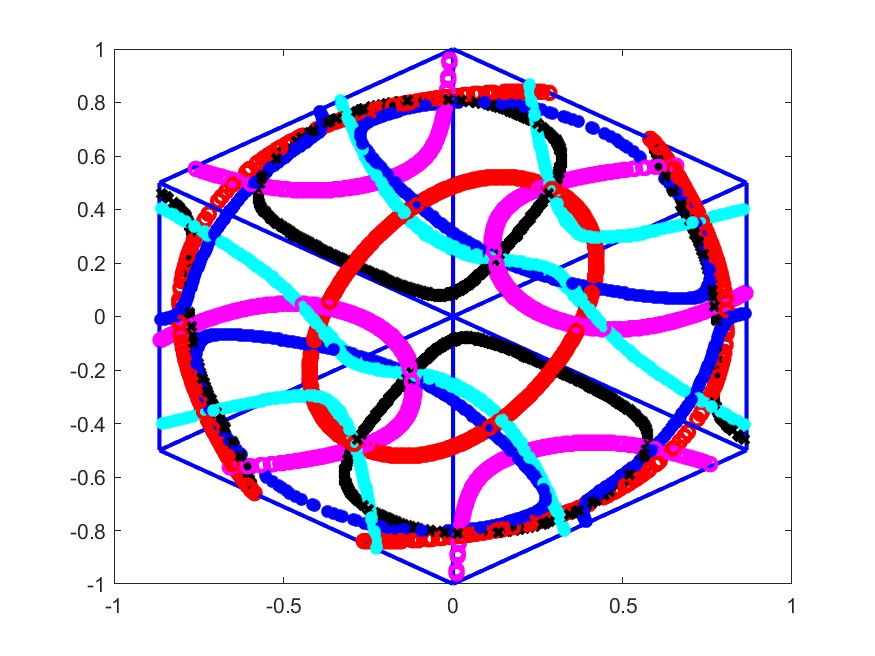}
		\includegraphics[width=0.18\linewidth]{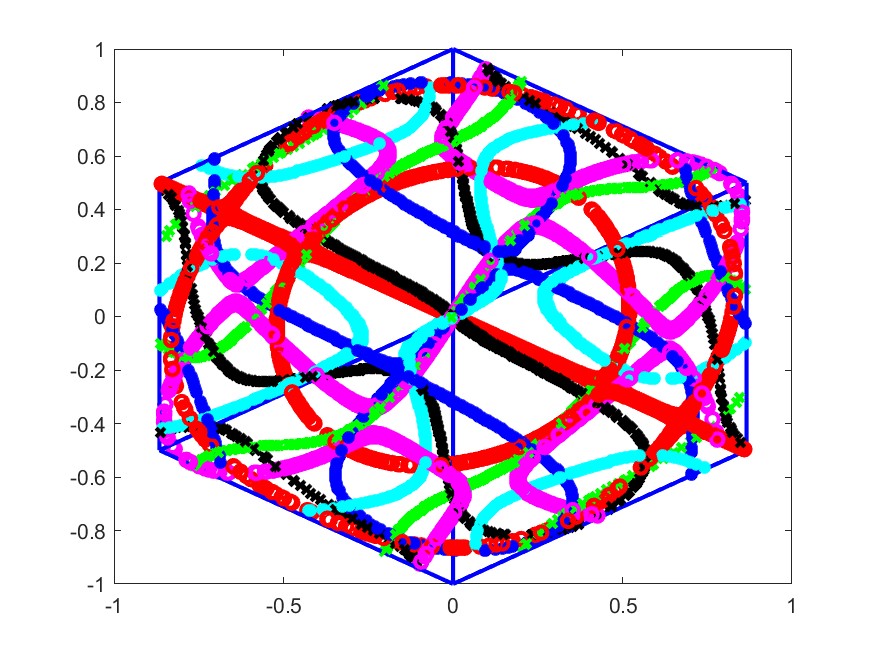}
	\caption{Zero Curves of Orthogonal Polynomials of degrees $2, 3, 4, 5$  \label{zerosH}}
	\end{figure}
We can see from Figure~\ref{zerosH} that there are a common zero for $P_{d,j},j=0, 1, \cdots, d$ for
 $d=1$, $d=3$ and $5$, but no common zeros for $d=2$ and $d=4$. 
The common zero for $d=3$ and $d=5$ is the origin $(0,0)$.  
\end{example}

\begin{example}
Consider an octagon $\Omega$ and we plot the zeros of orthogonal polynomials of degree 
$2--5$ in Figure~\ref{zerosO}. 
	\begin{figure}[htpb]
	\includegraphics[width=0.18\linewidth]{commonzerosOd1.jpg}
		\includegraphics[width=0.18\linewidth]{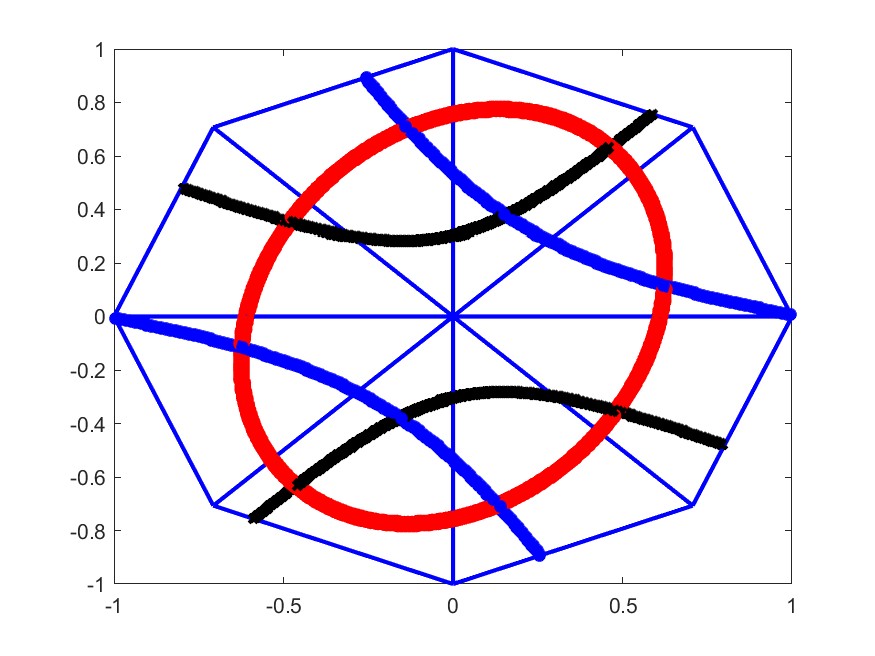}
		\includegraphics[width=0.18\linewidth]{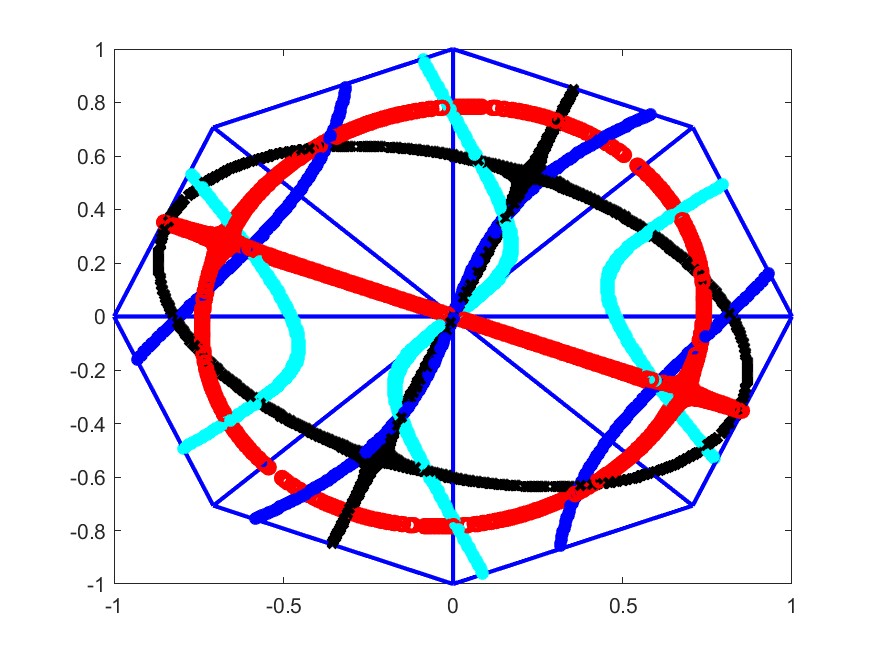}
				\includegraphics[width=0.18\linewidth]{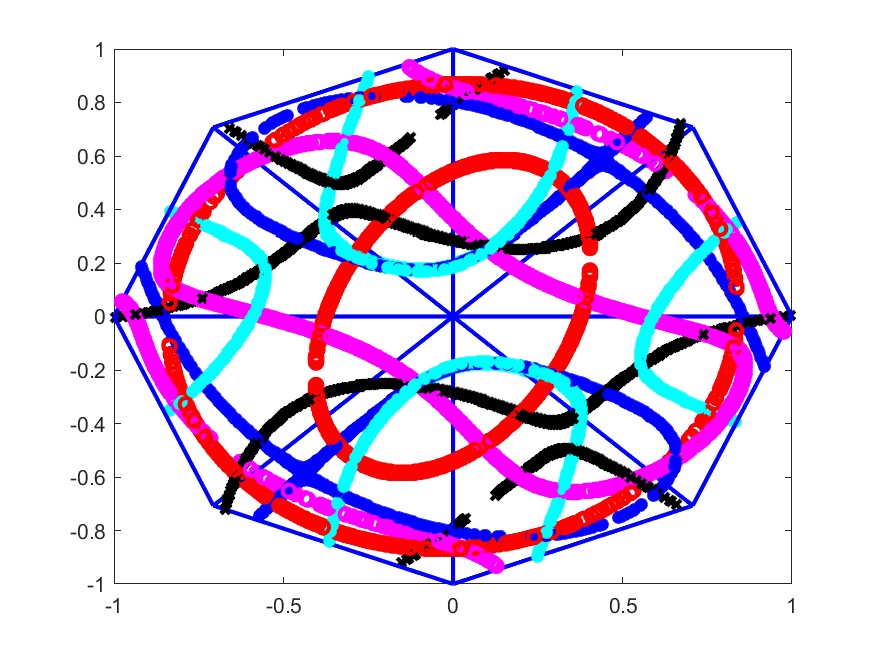}
		\includegraphics[width=0.18\linewidth]{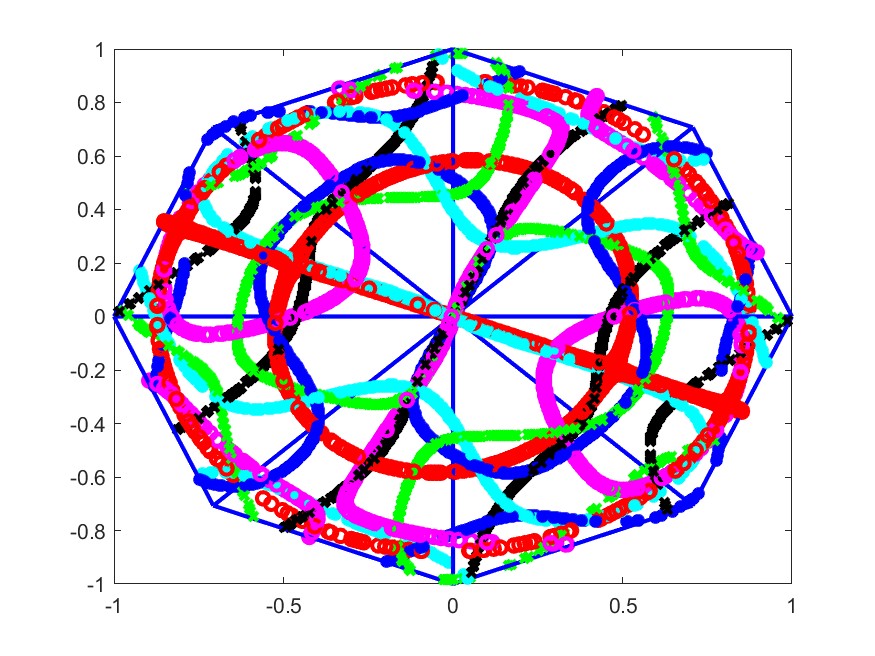}
	\caption{Zero Curves of Orthogonal Polynomials of degrees $2--5$  \label{zerosO}}
	\end{figure}
We can see from Figure~\ref{zerosH} that there are a common zero for $P_{d,j},j=0, 1, 
\cdots, d$ for $d=3$ and $5$, but no common zeros for $d=2$ and $d=4$. 
The common zero for $d=3$ and $d=5$ is the origin $(0,0)$.  
\end{example}

In fact, we did more calculation to see that the origin $(0,0)$ is the common zero of 
$P_{d,j}, j=0, \cdots, d$ for $d=3, 5, 7, $ etc.. when $\Omega$ is a hexagon and 
octagon.  More examples can be seen by using the MATLAB codes from the author.  

The above computational results echo the theoretical result discovered in \cite{X94}. 
See Theorem  3.1.1 in \cite{X94}. See also Corollary 3.7.7 in \cite{DX14}.  
\begin{thm}[Xu'1994 \cite{X94}]
If the domain $L$ is centrally symmetric and the dimension $d \ge 2$, then 
the orthogonal polynomial $\mathbb{P}_n$ has fewer than dim $\mathbb{P}_{n-1}$ 
common zeros. If, moreover, in $\mathbb{R}^2$, then orthogonal polynomials for 
$\mathbb{P}_d$ has no common zero if $d$ is even and has one common zero 
($\mathbf{x} = 0$) if $d$ is odd.
\end{thm}

\begin{remark}
One observation from the examples above is that when \textcolor{blue}{ we have } a polygonal domain, say J-shape, there
is no common zero for odd degree orthogonal polynomials.  
\end{remark}

\begin{remark}
Another observation from the construction above is when a domain is centrally 
symmetric, e.g. square, hexagon, octagon, etc., the 
orthogonal polynomials of an odd degree contains only monomials of odd degree and  the
orthogonal polynomials of an even degree contains only monomials of even degree. 
\end{remark}

\section{Applications of Orthogonal Polynomials}
A natural application of orthogonal polynomials is to build numerical quadrature over
various polygons motivated by the well-known Gauss quadrature in the univariate setting. 
Although it is known that no Gauss quadrature exists over even centrally symmetric 
domains, we will provide two approaches to use orthogonal polynomials for quadrature 
formula. One is to use polynomial reduction/division to reduce a polynomial 
function, e.g. an approximation of the given continuous function of interest, to a 
remainder of quadratic polynomials. One point quadrature formula can be built which 
will be exact for many more functions than linear polynomials. Another approach is to use
orthogonal polynomials for approximating the given continuous function of interest by 
interpolation based on its function values at some given locations over a domain 
of interest. 
Then the integration of the given function over the domain is simply the constant term of
the interpolating polynomial.   
These two approaches will be discussed in the following two subsections.  Finally, we 
discuss how to use less numbers of function values to achieve a higher order of 
polynomial accuracy. 

\subsection{Integration by Polynomial Reduction}
 We now extend some  ideas in the preliminary section to the two-dimensional setting.  
 Let $\Omega \subset \mathbb{R}^2$ be a bounded domain. For the convenience of 
computation to be discussed later, let $\Omega$ be a polygon. 
Even orthogonal polynomials are not unique, let us assume that we have a set of 
orthogonal polynomials over $\Omega$ based on the computation from the previous section. 

Recall that $\mathbb{P}_d$ is the polynomial space of total degree $\le d$ and write 
$P_{d,i}, i=0, \cdots, d$ for the orthogonal polynomials in 
 $\mathbb{P}_{d} \ominus \mathbb{P}_{d-1}$.  Then for any polynomial $P$ of degree 
$2d-1$, we can have a similar division formula as in the univariate setting 
 (\ref{Gauss}).    The following  result is well-known, e.g. see similar results in  
 \cite{AL}. 
For convenience, we present a proof. 
\begin{thm} 
\label{AGthm1}
Consider the bivariate polynomials in $\mathbb{P}_d, d\ge 1$.  
Suppose we have linearly independent (d+1) polynomials $P_{d,0}, ..., P_{d,d}$ of degree 
$d$ which span the space of all homogeneous polynomials of degree exactly $d$. Given 
a polynomial  $P$ of degree $\le 2d-1$ with $d\ge  2$,  there are   polynomials $q_i$ 
of degree $< d$ such that
\begin{equation*}
  P(x,y) = \sum_{i=0}^d q_i(x,y) P_{d,i}(x,y) + R(x,y),
\end{equation*}
where R(x,y) is a polynomial of degree $ d-1$ or less. 
\end{thm}
\begin{proof}
For convenience, let us give a proof. 
Without loss of generality, we assume that the given polynomial $P$ of degree $2d-1$ and 
the given $d+1$ linearly independent polynomials $P_{d,0}, \cdots, P_{d,d}$ are in Taylor
format, i.e. $P_{d,i}= x^iy^{d-i}$ for $i=0, \cdots, d$. Also, we write 
\begin{equation*}
P(x,y) = \sum_{i=0}^{2d-1} P_i(x,y),
\end{equation*}
where $P_i(x,y)$ is the part of polynomials of $P$ whose degree is exactly $i$. Let 
$R(x,y)= \sum_{i=0}^{d-1}P_i(x,y)$ be the residual portion of $P$. For $i\ge d$, we now write 
$P_i$ in terms of $P_{d,j}, j=0, \cdots, d$. It is easy to see when $i=d$, 
$P_d= \sum_{j=0}^d c_{d,j} P_{d,j}$ with constant coefficients $c_{d,j}$'s. 
Assume that $P_{i-1}$  a linear combination of $P_{d,j}, j=0, \cdots, d$ with polynomial 
coefficients $c_{i,j}$ when $i\ge d+1$. Let us consider $P_{i}$ with $i>d+1$. We write 
\begin{eqnarray*}
P_{i} &=& \sum_{k=0}^{i} a_k x^k y^{i-k} = \sum_{k=0}^d a_k x^k y^{d-k} (y^{i-d}) 
+ \sum_{k=d+1}^{i} a_k x^{k} y^{i-k}\cr 
&=& \sum_{k=0}^d a_k x^k y^{d-k} (y^{i-d})+ \sum_{k=0}^{i-d-1}a_{i-k} x^{d-k}y^k (x^{i-d})\\
&=& \sum_{k=0}^d (a_k y^{i-d}) x^k y^{d-k} + \sum_{l=2d+1-i}^{d}(a_{i-d+l}x^{i-d}) x^{l}y^{d-l}.
\end{eqnarray*}
That is, $P_i$ is a linear combination of $P_{d,k}, k=0, \cdots, d$ with polynomial 
coefficients with degree $\le d-1$.  By induction, all $P_i$ \textcolor{blue}{ are } a linear combination 
of $P_{d,j}, j=0, \cdots, d$ for $i\ge d$. Then $P$ can be written in the desired form.  
\end{proof}

Let us remark that when polynomials $P$ and $P_{d,i}$ are with rational coefficients, the 
above division can be done by using a MATLAB command PolynomialReduce.m.  
However, when the coefficients of these polynomials are irrational numbers, 
the MATLAB command does not work well. 
We now use the result above to derive a quadrature formula under some additional assumptions. 
\begin{thm} 
\label{mjlai10082024}
Let $P_{d,j},j=0, \cdots, d$ be orthogonal polynomials in $\mathbb{P}_d\ominus 
\mathbb{P}_{d-1}$ over polygonal domain
$\Omega$. Suppose that  there are $D_{d-1}$ locations: $p_1, \cdots, p_{D_{d-1}}\in 
\Omega$ which are some common zeros of all $P_{d,j},j=0, \cdots, d$.  Furthermore, suppose
that these locations $p_1, \cdots, p_{D_{d-1}}$ admit a unique polynomial 
interpolation of degree $d-1$, that is, for any continuous function $f$ 
defined on $\Omega$, there exists a unique polynomial $P_f$ satisfying 
$P_f(p_i)= f(p_i), i=1, \cdots, D_{d-1}$.  Writing 
\begin{equation}
P_f(x,y) = \sum_{i=1}^{D_{d-1}} f(p_i) L_i(x,y)
\end{equation}
for some Lagrange basis functions $L_i$ 
and letting $c_i = \int_\Omega L_i(x,y)dxdy$ for each $i$, we have 
\begin{equation}
\label{quadratureformula}
\int_\Omega f(x,y)dxdy \approx   \sum_{i=1}^{D_{d-1}} f(p_i)c_i =:Q(f).
\end{equation}
Then this quadrature $Q(f)$ is exact for all polynomials $f$ of degree $2d-1$.  
\end{thm}
\begin{proof}
Indeed, given a polynomial $f$ of degree $2d-1$, we use the theorem above to have
$$
f= \sum_{i=0}^d q_i P_{d,i}+ R.
$$
Note that the interpolating polynomial $P_f$ of
degree $d-1$ at these points $p_i$ satisfies 
$f(p_j)= P_f(p_j) = \sum_{i=0}^d q_i(p_j) P_{d,i}(p_j) + R(p_j)= R(p_j)$ for all $j$.   
So $P_f$ is the interpolating polynomial of degree $d-1$ with function  
values $R(p_i)$ instead and $R$ is of degree $d-1$, $P_f\equiv R$.     
 By the quadrature formula (\ref{quadratureformula}), 
$$
Q(R) = \int_\Omega R(x,y)dxdy = \int_\Omega (R(x,y)+ \sum_{i=0}^d q_i(x,y) 
P_{d,i}(x,y)) dxdy =\int_\Omega f(x,y)dxdy.
$$
That is, the quadrature is exact for polynomials of degree $2d-1$.   
\end{proof}

As we have seen in Figure~\ref{commonzeros} 
that there is a common zero of orthogonal polynomials in $\mathbb{P}_1
\ominus \mathbb{P}_0$ over many different polygons, we can use the above theorem 
to obtain the following quadrature formulas:  
\begin{example}
Let $\Omega$ be a polygonal domain and for $d=1$, let $P_{1,1}$ and $P_{1,2}$ be
two linear orthogonal polynomials which is also orthogonal to constant function 
$P_{0,0}$.  Since 
$P_{1,1}$ and $P_{1,2}$ are linear polynomials and orthogonal to $P_{0,0}$, $P_{1,1}$ is 
positive and negative over $\Omega$, and hence, the zeros of $P_{1,1}$ is a line inside
$\Omega$.  So is $P_{1,2}$. As these two lines are not parallel,  let $p_1$ be the 
interception  of these two lines.  If $p_1\in \Omega$, then 
\begin{equation}
\label{quadrature0}
\int_\Omega f(x,y)dxdy \approx f(p_1) A_\Omega,
\end{equation}
where $A_\Omega$ stands for the area of $\Omega$. By Theorem~\ref{mjlai10082024}, 
the quadrature will be exact for all linear polynomials.  For example
\begin{equation}
\label{GaussQuadrature0}
\int_{-1}^1\int_{-1}^1 f(x,y)dxdy \approx 4 f(0,0)
\end{equation}
which is exact for all linear polynomials. The quadrature (\ref{quadrature0}) is the
well-known one point quadrature in the bivariate setting. 
\end{example} 

\begin{example}
Consider a triangular domain $T$ with vertices $\{(0,0),(1,0),(1,1)\}$. Then the point
$p=(2/3,1/3)$ is the intersection of two orthogonal linear polynomials as 
shown in the first graph of Figure~\ref{zeros1} which is also 
orthogonal to constant functions. Then the following quadrature 
$$
\int_T f(x,y)dxdy \approx  \frac{1}{2} f(2/3,1/3)
$$
will  be exact for all linear polynomials. 
\end{example}

\begin{remark}
Note that Theorem~\ref{mjlai10082024} may be used to conclude that there is no Gaussian 
quadrature for $d>1$ as we have seen that there are not enough number of common zeros
of orthogonal polynomials in $\mathbb{P}_{d}\ominus \mathbb{P}_{d-1}$ for $d>1$ over a square domain, pentagonal domain,  
hexagonal domain,  octagonal domain,  L-shaped domain, etc. in the previous section.  
This observation was already pointed out in \cite{X94}, see also \cite{DX14} or \cite{X21}.   
\end{remark}

We now study the one point quadrature as it will be exact for more than just linear 
polynomials. The following is one of the main theorems in this paper. 
\begin{thm}
\label{mainresult1}
Let $\Omega$  be a  polygon, e.g.  square, hexagon or octagon, etc such that there 
exists a point $\mathbf{p} \in \Omega$ which  is the common zero of all orthogonal 
polynomials $P_{2d-1,j}, j=0, \cdots, 2d-1$ for $d=1, 2, \cdots, $.  
For any polynomial $P_n$ of degree $n$ with $2d-1 \le n\le 2d$,  let us write 
\begin{equation}
\label{mjlai10262024}
P_n(x,y) = \sum_{j=0}^{2d-1} q_j(x,y) P_{2d-1,j}(x,y) + \sum_{j=0}^{2d-3} \hat{q}_j(x,y) 
P_{2d-3,j}(x,y) + \cdots + \sum_{j=0}^1 \tilde{q}_j(x,y) P_{1,j}(x,y)+   R  
\end{equation}
based on orthogonal polynomials of odd degree, 
where the residual $R$ is a constant and quotients $q_j(x,y), \hat{q}_j(x,y), 
\tilde{q}_j(x,y)$  are  
linear polynomials.   If $\tilde{q}_j(x,y)$ are constant functions for  $j=0$ and $j=1$, then 
\begin{equation}
\label{surprise}
\int_\Omega P_n(x,y)dxdy  = P_n(\mathbf{p}) A_\Omega,
\end{equation}
where $A_\Omega$ is the area of $\Omega$.  If 
$\tilde{q}_j$ are indeed 
linear polynomials with $\tilde{q}_j= a_ix+b_jy+ c_j$, then 
\begin{equation}
\label{surprise2}
\int_\Omega P_n(x)dx  = P_n(\mathbf{p}) A_\Omega + \sum_{j=0}^1( a_j \alpha_j + b_j \beta_j),
\end{equation}
where $\alpha_j =\int_\Omega x P_{1,j} (x,y) dxdy$ and $\beta_j 
=\int_\Omega y P_{1,j}(x,y)dxdy$ for $j=0,1$ which can be precomputed beforehand.
\end{thm}
\begin{proof}
For convenience, let us say $d=3$. In this case, consider a polynomial function $P$   of 
degree $6$ or less.  We first expand 
$P(x,y) = \sum_{j=0}^{d} q_j(x,y) P_{5, j}(x,y) + R_1(x,y)$, 
where $q_j\in \mathbb{P}_1$ and $R_1\in \mathbb{P}_4$.
Then we express $R_1=  \sum_{j=0}^3 \hat{q}_j(x,y) P_{3,j}(x,y) + R_2$ with $\hat{q}_j\in 
\mathbb{P}_1$ and $R_2\in \mathbb{P}_2$.  
We write $R_2(x,y) = \sum_{j=0}^1 \tilde{q}_j(x,y) P_{1,j}(x,y)+ R_3(x,y)$, 
where $\tilde{q}_j$ and $R_3$ are linear polynomials.   
If $\tilde{q}_j$ are constants, then by the orthogonal property of $P_{1,j}$ with constant 
functions, we know 
$\int_\Omega R_2(x,y)dxdy  = \int_\Omega R_3(x,y)dx = R_3(\mathbf{p}) A_\Omega$ by 
Theorem~\ref{mjlai10082024}.  
Note that $R_3(\mathbf{p})= P(\mathbf{p})$ as $\mathbf{p}$ is the common zeros  
of all $P_{2k-1,j}, j=0, \cdots, 2k-1, k=1,2, \cdots, $.  Indeed, by the orthogonality, we have
\begin{eqnarray}
\int_\Omega P(x,y)dx &=&  \sum_{j=0}^{5} \int_\Omega q_j(x,y) P_{5,j}dx +  \sum_{j=0}^{3}  
\int_\Omega \hat{q}_j(x,y) P_{3,j}(x,y)dx + \cr
&& +\sum_{j=0}^{1}  \int_\Omega \tilde{q}_j(x,y) P_{1,j}(x,y)dxdy  +  \int_\Omega R_3(x,y)dxdy\cr
&=& \int_\Omega R_3(x,y)dxdy  = R_3(\mathbf{p}) A_\Omega = P(\mathbf{p}) A_\Omega .
\end{eqnarray}
When $\tilde{q}_j= a_j x+b_j y+ c_j$,  $\int_\Omega \tilde{q}_j(x) P_{1,j}(x,y)dxdy = a_j 
\int_\Omega x P_{1,j}(x,y)dxdy+ b_j\int_\Omega y P_{1,j}(x,y)dxdy 
+ c_j \int_\Omega P_{1,j}(x,y)dxdy= a_j \alpha_j + b_j \beta_j$.  
These complete the proof.  
\end{proof}

Similarly, we can have the following
\begin{thm}
\label{mainresult2}
Let $\Omega$  be a centrally symmetric  domain in the sense that $(\pm x, \pm y)\in \Omega$.  
For any polynomial $P_n$ of degree $n$ with $2d \le n\le 2d+1$,  let us write 
\begin{equation}
\label{mjlai10262025}
P_n(x,y) = \sum_{j=0}^{2d} q_j(x,y) P_{2d,j}(x,y) + \sum_{j=0}^{2d-2} \hat{q}_j(x,y) 
P_{2d-2,j}(x,y) + \cdots + \sum_{j=0}^2 \tilde{q}_j(x,y) P_{2,j}(x,y)+   R(x,y), 
\end{equation}
based on orthogonal polynomials of even degree, 
where the residual polynomial $R(x,y)$ and quotients $q_j(x,y), \hat{q}_j(x,y), 
\cdots, \tilde{q}_j(x,y)$  are  
linear polynomials.   If $R(x,y)= ax+ by+ c$, then 
\begin{equation}
\label{surprise3}
\int_\Omega P_n(x,y)dxdy  = c A_\Omega,
\end{equation}
where $A_\Omega$ is the area of $\Omega$.  Furthermore, if $\Omega$ is a general polygonal 
domain, then writing $R(x,y)=ax+by+c$, 
\begin{equation}
\label{surprise4}
\int_\Omega P_n(x,y)dxdy  = a \int_\Omega xdxdy+ b \int_\Omega ydxdy+ c A_\Omega.
\end{equation}
\end{thm} 
\begin{proof}
Due to the central symmetry of the domain $\Omega$, we have $\int_\Omega xdxdy=0=\int_\Omega
ydxdy$ and hence, 
$\int_\Omega R(x, y)dxdy = 
a\int_\Omega x dxdy + b\int_\Omega y dxdy + c\int_\Omega dxdy= cA_\Omega$. By using the 
orthogonality 
of $P_{2d,j}$'s, $P_{2d-2,j}$'s, and $P_{2,j}$'s, we have the result in (\ref{surprise3}).  
Otherwise, when $\Omega$ is a general polygonal domain, we have to compute 
$\int_\Omega x dxdy $ and $\int_\Omega y dxdy $ beforehand. Then (\ref{surprise4}) can be 
obtained after the polynomial reduction using the orthogonal polynomials of even degree. 
\end{proof}

The result above in Theorem~\ref{mainresult1} shows that 
the one point quadrature  (\ref{surprise})  for bivariate continuous functions 
will be exact for many nontrivial polynomials $P_n$ of higher degree  
than linear polynomials as long as the linear polynomial 
$\widetilde{q}_i$ mentioned in the theorem above are constants for $i=0$ and $i=1$.  When 
$\widetilde{q}_i$ are indeed linear polynomials, the coefficients
$\alpha_j, \beta_j$ can be computed by a brute force method and hence, 
(\ref{surprise2}) can be useful.  The result in Theorem~\ref{mainresult2} seems more convenient than
Theorem~\ref{mainresult1} if the domain is centrally symmetric. Even if the domain is 
not centrally symmetric, we  need to compute $\int_\Omega x dxdy,  \int_\Omega y dxdy, 
A_\Omega$ by a brute force method to obtain
a similar result in Theorem~\ref{mainresult1} which requires to compute the integration of 
all quadratic polynomials over $\Omega$. 

To use Theorem~\ref{mainresult1}, two questions remaining is (1) how to see 
from $P_n$ that the linear polynomials 
$\tilde{q}_j$ are constant functions and (2) how to compute 
$\tilde{q}_i$'s if  they are indeed nontrivial linear polynomials and find their 
coefficients $a_j, b_j$ for $\tilde{q}_j= a_jx+b_jy+ c_j, j=0, 1$.  
To use Theorem~\ref{mainresult2}, we have the similar questions. That is, how do we know the 
residual function $R(x,y)$ is a constant function? Certainly, when $p$ is an even polynomial and/or
the domain is centrally symmetric, the result (\ref{surprise2}) holds. 

One of the significances of Theorem~\ref{mainresult1} and Theorem~\ref{mainresult2} 
is that the integration of any polynomials is only dependent on the integration of their 
lower order terms, more precisely, the quadratic terms or linear terms if they are expanded in the forms
(\ref{mjlai10262024}) or (\ref{mjlai10262025}). 
Indeed, the results in Theorem~\ref{mainresult1} tell us that
if we divide the polynomial $P_n$ by orthogonal polynomials of odd degree to get the remainders, i.e. 
$a_j, b_j$'s, then the integration of $P_n$ is completely dependent on 
$\int_\Omega x P_{1,j}(x,y)dxdy$ and $\int_\Omega y P_{1,j}(x,y)dxdy$ for $j=0, 1$.  

For convenience, let us define 
\begin{equation}
\label{LaisSpace}
\mathcal{L}(\Omega) = \{   \sum_{d=2}^{d'} \sum_{j=0}^{2d-1} q_{d,j}(x,y) P_{2d-1,j}(x,y) 
+ R(x,y):  \hbox{ for all linear polynomials } q_{d,j}, R, \forall d'\ge 2\}
\end{equation}
Based on Theorem~\ref{mainresult1}, the one point quadrature (\ref{surprise}) is exact 
for all functions in $\mathcal{L}(\Omega)$ under the assumption that all orthogonal 
polynomials $P_{2d-1,j}, j=0, \cdots, 2d-1$ 
have the common zero $p\in \Omega$ for all $d\ge 1$.  We say that each function 
in $\mathcal{L}(\Omega)$ is quasi-divisible by orthogonal polynomials of odd degree. 
If a continuous function $f$ which can be approximated by using the quasi-divisible  
polynomials in  $\mathcal{L}(\Omega)$, the integration of $f$ can be approximated by the 
one point quadrature in (\ref{surprise}).  Note that the closure of  $\mathcal{L}(\Omega)$ 
has a lot of functions,  more than half of 
the continuous functions in $C(\Omega)$.  As in a previous section, we define 
\begin{equation}
\label{LaisSpace2D}
\mathbb{L}(\Omega)=  \hbox{span}\{ \mathcal{L}(\Omega), \mathbb{P}_2\}.
\end{equation}
Similar to the proof of Theorem~\ref{mainresult1}, we can establish the following 
\begin{thm}
The closure of $\mathbb{L}(\Omega)$ in the maximum norm is $C(\Omega)$, the space of all continuous
functions over $\Omega$.
\end{thm}
\begin{proof}
We leave the proof to the interested reader. 
\end{proof} 

In the same fashion, we can define
\begin{equation*}
\mathbb{L}_e(\Omega) = \{   \sum_{k=1}^{d} \sum_{j=0}^{2k} q_{k,j}(x,y) P_{2k,j}(x,y) 
+ R(x,y):  \hbox{ for all linear polynomials } q_{k,j}, R, \forall d\ge 1\}
\end{equation*} 
based on the orthogonal polynomials of even degree constructed in the previous section. 
We can establish the result that the closure of
$\mathbb{L}_e(\Omega)$ in the maximum norm is $C[\Omega]$.  
All these are left to the interested reader.  

In general, we can use the following class of functions to classify the functions 
such that the one point quadrature is exact:
\begin{eqnarray}
\label{LaisSpace3}
\mathbb{M}(\Omega) = \{ &&   \sum_{d=1}^{d'} \sum_{j=0}^{2d-1} q_{d,j}(x,y) 
P_{2d-1,j}(x,y) + R(x,y):  \cr
&& \hbox{ for  polynomials $ q_{d,j}$ of degree $\le 2d-2$ for all } j=0, \cdots, 2d-1,\cr 
&& R(x) \hbox{ is a constant}, \forall d'\ge 2\}.
\end{eqnarray}
It is easy to see that $\mathbb{L}(\Omega) \subset \mathbb{M}(\Omega)$. Also, many 
nontrivial quadratic polynomials are not inside  $\mathbb{M}(\Omega)$ defined in 
(\ref{LaisSpace3}).  Let 
$Q_{i}, i=0, \cdots, 5$ be quadratic monomials  
and suppose we have their integrals, e.g. we compute them beforehand numerically:
\begin{equation*}
I_i= \int_\Omega Q_{i}(x,y)dxdy, i=0, \cdots, 5.
\end{equation*}
Then we have 
\begin{thm}
For any polynomial $P$ of degree $n\ge 2$, there exists a polynomial 
$M\in \mathbb{M}(\Omega)$ and a quadratic polynomial $J$ such that 
\begin{equation}
P(x,y) = M(x,y)+ J(x,y)
\end{equation}
and hence,  $\int_\Omega P(x,y)dxdy $ is a linear 
combination of these values $I_i,i=0, \cdots, 5$.  
\end{thm} 
\begin{proof}
Let $P_n$ be a polynomial of degree $n$, say $n=100$. We use $P_{51,i}, i=0, \cdots, 51$ to
divide $P_n$. Since $P_{51,i}$'s are linearly independent, we have 
$$
P_n (x,y)=\sum_{i=0}^{51} q_{26,i}(x,y) P_{51,i}(x,y) + R_1(x,y)
$$
by Theorem~\ref{AGthm1}, where the remainder $R_1(x,y)$ is a polynomial  
of degree $\le 50$.  Then we apply Theorem~\ref{AGthm1} again to have
$$
R_1(x,y)= \sum_{i=0}^{27} q_{14,i}(x,y) P_{27,i}(x,y) + R_2(x,y),
$$
where  the remainder $R_2(x,y)$ is of degree $\le 26$.  We repeat the above computational
procedures to have
$$
R_3(x,y), R_4(x,y), R_5(x,y) \hbox{ and } R_6(x,y), 
$$
where the degree of $R_3(x,y)$ is less than or equal to $14$,  the degree of $R_4(x,y)$ 
is less than or equal to $8$, the degree of $R_5(x,y)$ is less
than or equal to $4$, and the degree of $R_6(x,y)$ is less
than or equal to $2$.  Letting $J(x,y)=R_6(x,y)$ and $M(x,y)$ be the summation of all
quotients above, we have
$$
P_n(x,y)= M(x,y)+ J(x,y).
$$    
Finally, we rewrite $J(x,y) =\sum_{j=0}^5 c_j Q_{j}(x,y)$.  Now the integral 
$$
\int_\Omega P_n(x,y)dxdy  =\int_\Omega M(x,y)dxdy+ \int_\Omega J(x,y)dxdy 
=\sum_{j=0}^5 c_j \int_\Omega Q_{j}(x,y)dxdy
$$
which is a linear  combination of the values $I_i,i=0, \cdots, 5$.  
\end{proof}

Finding the right quadratic polynomial $J$ for a general polynomial $P_n$, or a continuous
function is a key to find the integration of $\int_\Omega P_n(x,y)dxdy$ or 
$\int_\Omega f(x,y)dxdy$.  On the other hand, it is relatively 
easy to find the linear orthogonal polynomials for any fixed domain $\Omega$ based on 
the MATALB codes discussed in a previous section.    As
$\alpha_j$ and $\beta_j$ for $j=0, 1$ can be precomputed, the main quadrature formula 
(\ref{surprise2}) will be dependent on the computation 
of $a_j, b_j, j=0, 1$ for polynomial 
function $P$.  Similar to the quadrature formula given in 
Theorem~\ref{betterthanGaussianQuadrature}, the computation of $a_j, b_j,j=0, 1$ 
can lead to a 
quadrature formula for a general continuous function $f$.  Let us give an example. 

\begin{example}
Let $\Omega$ be a hexagon with the vertices on the unit circle.  
See Example~\ref{hexagon} for 
the zeros of orthogonal polynomials 
of various degree. We first present the four orthogonal polynomials of degree $3$ in 
$\mathbb{P}_3\ominus \mathbb{P}_2$. They are given in Table~\ref{Hd3}. 
Also we present the two orthogonal polynomials of degree 1 in  Table~\ref{Hd1}.  
 
\begin{table}[thbp]
\centering
\begin{tabular}{|c|cccccc|}\hline
& $x^3$ & $x^3y$ & $xy^2$ & $y^3$ & $x$ & $y$ \cr \hline 
$P_{3,0}$ &    963/269    &  4367/766   &   2203/908 &  4367/766&  -2822/1531 & -3100/971   \cr
$P_{3,1}$ &  -599/391   &   5297/1016  & 2241/443  &  -1413/1268  &    -467/7208     & -590/2253 \cr 
$P_{3,2}$ & -1670/903   & -2909/873    &  2222/359   &  3893/2097 &  -187/2083    & -83/265   \cr
$P_{3,3}$ &  1213/214  &   -3505/1069  &  25591/4285  &   -4842/1469  &   -1573/489  & 977/530 \cr  \hline
\end{tabular}
\caption{Orthogonal Polynomials of Degree 3 over the Hexagon Domain \label{Hd3}}
\end{table}
\begin{table}[thbp]
\centering
\begin{tabular}{|c|cc|}\hline
& $x$ & $y$ \cr \hline 
$P_{1,0}$ &    -804/625  &       428/975   \cr    
$P_{1,1}$ &    -428/975   &  -804/625    \cr 
\hline
\end{tabular}
\caption{Orthogonal Polynomials of Degree 1 over the Hexagonal Domain \label{Hd1}}
\end{table}

For any even polynomial $P$ of degree $4$, i.e. $P(x,y)=P(-x,-y)$, say
$$
P(x,y) = p_1 x^4+ p_2 x^3y+p_3 x^2 y^2+ p_4 xy^3+ p_5 y^4+ p_6 x^2 + p_7 xy +p_8 y^2+ p_9, 
$$ 
we can find coefficients $c_i, d_i, i=0,1, 2,3$ and $a_i, b_i, i=0, 1$ such that 
\begin{equation}
\label{interp}
P(x,y) -P(0,0) = \sum_{i=0}^3 (c_ix+d_iy)P_{3,i}(x,y) + \sum_{i=0}^1 (a_ix+b_iy) P_{1,i}(x,y).
\end{equation}
As we only need to match the coefficients of $P$ which has 8 coefficients 
$p_1, \cdots, p_8$ while we have 
 12 unknowns.  When we randomly choose 8 nodes $\mathbf{n}_i, i=1, \cdots, 8$ inside the 
hexagon given in Table~\ref{Hnodes}. 
 
 \begin{table}[thpb]
 \begin{tabular}{cccccccc}
 -429/3395  &      251/941 &      -935/2817     &  96/4867 &      266/2087    &  575/2688  &     -563/2908  &     -423/1790\cr  
     941/2262  &     101/878  &    -1371/3370 &      361/2827 &    -1186/3851 &      511/1845 &      230/631 &      -447/2686\cr
\end{tabular}
\caption{8 nodes inside the hexagonal domain \label{Hnodes}}
\end{table}

 The interpolating conditions (\ref{interp}) with $(x,y)$ being the 8 nodes 
listed in Table~\ref{Hnodes} 
 lead to a linear system which can be solved by 
 setting $a_0=0, b_0=0$ and $a_1=0$. The solution  $b_1$ can be found for any given $p_1, 
\cdots, p_9$.     
We therefore obtain the following quadrature:
\begin{equation}
\label{Hquadrature}
\int_\Omega f(x,y)dxdy \approx  \sum_{i=1}^8 w_i (f( \mathbf{n}_i)- f(\mathbf{0}))  \beta_1 + f(\mathbf{0}) A_\Omega, 
\end{equation}
where $A_\Omega= 1351/520$ and the weights $w_i$'s are given below.  

\begin{tabular}{cccccccc}
-2421/187   &   -8603/57 &     -22423/1729 &   -12315/29  &      7886/59  &      2501/165 &    -10964/219 &      21800/113
\end{tabular}

and $\beta_1 = \int_\Omega y P_{1,1}(x,y)dxdy=-431/619$.   If we use the quadrature 
(\ref{Hquadrature}) for even function $f(x,y)=\cos(x+y)$, 
we get a very accurate integral value $2.093802517305221$. Comparing it with the 
exact value $2.093390032732584$, the absolute error is  $4.124845726369841e-04$.
\end{example}

\subsection{Integration by Interpolation: Multi-Point Quadrature}  
We begin with another approach to compute the integration of a continuous function 
$f$ over any polygon $\Omega$ by using  orthogonal polynomials over $\Omega$. In 
particular, when the values of a continuous function are only available at several 
locations within 
$\Omega$, which is the common situation in practice, we build a quadrature based on the 
values and locations to approximate the integration of $f$ over $\Omega$.  

First of all, let us review the concept of domain points introduced in  \cite{LS07}.
Let $T\subset \Omega$ be a triangle in $\Omega$ and  let $\mathbf{v}_1, \mathbf{v}_2, 
\mathbf{v}_3$ be the vertices of $T$, Then 
\begin{equation}
\label{domainpoints}
\xi_{ijk} =\frac{1}{d}( i \mathbf{v}_1+j \mathbf{v}_2 +k \mathbf{v}_3), \forall 
i+j+k=d.
\end{equation}
are domain points of $T$. 
  Let $P_{d,1}, \cdots, P_{d,D_d}$ be the union of all iterative orthonormal 
polynomials of degree $\le d$, where $D_d= (d+1)(d+2)/2$ over $\Omega$. It is clear that   
we can write any polynomial $P$ of degree $d$ as $P= \sum_{i=1}^{D_d} c_i P_{d,i}$. 
Then  it is  known (e.g. Chung-Yao theorem in 1977 (cf. \cite{CL87} or \cite{LS07}) 
that for any continuous function $f$ over $T$, 
there exists a unique interpolating polynomial $P_f$ of degree $d$ such that 
\begin{equation}
\label{interpolation}
P_f(\xi_{ijk}) = f(\xi_{ijk}),  \quad \forall  i+j+k=d.
\end{equation}
That is, there exists a unique vector 
$\mathbf{c}=(c_1, \cdots, c_{D_d})^\top$ such that $P_f$ satisfies the following 
linear system of equations:
\begin{equation}
\label{linsys}
\begin{bmatrix}
P_{d,2}(\xi_{d00}) & P_{d,3}(\xi_{d00}) & \cdots & P_{d,D_d}(\xi_{d00}) & P_{d,1}(\xi_{d00})\cr
\vdots & \vdots & \cdots & \vdots & \vdots \cr
\vdots & \vdots & \cdots & \vdots & \vdots \cr
P_{d,2}(\xi_{00d}) & P_{d,3}(\xi_{00d}) & \cdots & P_{d,D_d}(\xi_{00d}) & P_{d,1}(\xi_{00d})\cr
\end{bmatrix} \mathbf{c}
=  \begin{bmatrix} f(\xi_{d00}\cr \vdots \cr \vdots \cr f(\xi_{00d})
\end{bmatrix}.
\end{equation}
Note that  we have arranged the first entry $c_1$ is in the end of the 
columns of the system  (\ref{linsys}) above.
The discussion above leads to the following numerical integration formula.
\begin{thm}
\label{main2}
Let $f$ be a continuous function over $T\subset \Omega$. Suppose that $c_1$ is the 
coefficient satisfying the above system of linear equations together with other 
coefficients $c_2, \cdots, c_{D_d}$. Then 
\begin{equation}
\label{mjlai09252024}
c_1 \approx \int_\Omega f(x,y)dxdy,
\end{equation}
where $P_{d,1}$ is so normalized that $\int_\Omega P_{d,1}(x,y)dxdy=1$.  
\end{thm}
\begin{proof}
Let us solve the whole system (\ref{linsys}) to obtain all the coefficients. Then
$$
P_f(x,y) = \sum_{i=1}^{D_d} c_i P_{d,i}(x,y)
$$
is a good approximation of $f$ satisfying the interpolation condition 
(\ref{interpolation}). Then it follows 
$$
 \int_\Omega f(x,y)dxdy  \approx \int_\Omega P_f(x,y)dxdy =\int_\Omega \sum_{i=1}^{D_d} 
c_i P_{d,i}(x,y)dxdy 
=\int_\Omega c_1 P_{d,1}(x,y)dxdy = c_1,
$$
where we have used the fact that all $P_{d,i}$ are orthogonal to 
$P_{d,1}$ which is a constant, that is, 
$
\int_\Omega P_{d,i}(x,y)dxdy =0
$ 
for all $i\ge 2$ while $\int_\Omega P_{d,1}(x,y)dxdy=1$ due to the normalization.  
Hence, we have (\ref{mjlai09252024}).  
\end{proof}

Note that the coefficient $c_1$ is a linear combination of function values 
$f(\xi_{ijk}), i+j+k=d$, that is,  
\begin{equation*}
c_1= \sum_{i+j+k=d} c_{d,ijk} f(\xi_{ijk})
\end{equation*}
which is a numerical quadrature over triangle $\Omega$, where $c_{d, ijk}$  are the 
coefficients which can be found from the linear system (\ref{linsys}). In particular, we can use 
Cramer's rule to find these coefficients. 

\begin{remark}
To obtain $c_1$ from the linear system~(\ref{linsys}), we only need to do Gaussian 
elimination and solve for $c_1$ without further 
backward substitution to get the remaining coefficients $c_2, \cdots, c_n$. 
Recall that polynomial integration over $T$ was discussed in \cite{LS07}. Once we solve the
linear system (\ref{interpolation}) in terms of Bernstein-B\'ezier polynomial basis, we 
simply add them all up and multiply it by the area of $T$ and divide it by 
${d+2\choose 2}$. See Theorem 2.33 on page 45 of \cite{LS07}.  
The method above offers a simpler computational method 
to obtain a numerical quadrature formula. 
\end{remark}

\begin{remark}
If we use the traditional polynomial interpolation method, we write 
$$P_f(x,y)= \sum_{i+j+k=d} f(\xi_{ijk}) L_{ijk}(x,y)$$
with $L_{ijk}$ being Lagrange basis function of degree $d$ associated with each 
$\xi_{ijk}$. Then 
$$
\int_\Omega f(x,y)dxdy \approx \int_\Omega P_f(x,y)dxdy =  \sum_{i+j+k=d} f(\xi_{ijk}) 
\int_\Omega L_{ijk}(x,y)dxdy.
$$
To obtain a quadrature formula, 
one has to find out the integrals $w_{ijk}=\int_\Omega L_{ijk}(x,y)dxdy$ which are 
nontrivial for a general polygonal domain. The results of Theorem~\ref{main2} offers a 
simpler approach
by solving a linear system once these orthogonal polynomials are found. 
\end{remark}

When $\Omega$ is a triangle, it is known how well the interpolating polynomial 
$P_f$ approximates $f$ as discussed in \cite{LS07}(See Theorem 1.12 on page 13.) 
After carefully examining the proof, we can conclude the following 
\begin{thm}
When $T\subset \Omega$ and $f\in C^{d+1}(\Omega)$, 
\begin{equation}
\label{2007}
\| f - P_f\|_{\Omega,\infty} \le K |\Omega|^{d+1} |f|_{d+1,\Omega,\infty}
\end{equation}
for a positive constant $K$ independent of $f$, where $|\Omega|$ stands for the diameter
of $\Omega$.  So the error of the quadrature formula ~(\ref{mjlai09252024}) with 
$\Omega$ is 
\begin{equation}
\label{error2007}
|\int_\Omega  f(x,y)dxdy - c_1| \le K |T|^{d+2}|f|_{d+1,\Omega,\infty}.  
\end{equation}
\end{thm} 

\begin{example}
\label{magic1}
Consider the integration $\int_\Omega f(x,y)dxdy$ over an hexagon domain $\Omega$
with area is $2.598076211353319$.  
We now present the computation in MATLAB format as follows.  
\begin{verbatim}
%The following command to find the integration based on orthogonal polynomials.
a=rand(28,1);
p1=@(x,y) a(1)+a(2)*x+a(3)*y;
p2=@(x,y) a(4)*x.^2+a(5)*x.*y+a(6)*y.^2;
p3=@(x,y) a(7)*x.^3+a(8)*x.^2.*y+a(9)*x.*y.^2+a(10)*y.^3;  
fname=@(x,y) p1(x,y)+p2(x,y)+p3(x,y); %or other function.
%fname=@(x,y) sin(pi*x+pi*pi*y); 
v1=[0 0];v2=[0.25 0];v3=[0,0.25];
d=3;
[x,y]=domain_pts(v1,v2,v3,d); %this code generates the domain points explained in [18]. 
P=[x y]; %P is a list of domain points.
%The rest of the code is to compute int_P f(x,y)dxdy. 
load Hexgond3 %This matlab data file contains the orthogonal polynomials in Taylor format.
P1=@(x,y)  1/2.598076211353319+x-x; %The MATLAB way to define P1 as a function of x.
P2=@(x,y)  X(2,2)*x+X(2,3)*y;
P3=@(x,y)  X(3,2)*x+X(3,3)*y;
P4=@(x,y) X(4,1)+X(4,4)*x.^2+X(4,5)*x.*y+X(4,6)*y.^2;
P5=@(x,y) X(5,1)+X(5,4)*x.^2+X(5,5)*x.*y+X(5,6)*y.^2;
P6=@(x,y) X(6,1)+X(6,4)*x.^2+X(6,5)*x.*y+X(6,6)*y.^2;
P7=@(x,y) X(7,2)*x+X(7,3)*y+X(7,7)*x.^3+X(7,8)*x.^2.*y+X(7,9)*x.*y.^2+X(7,10)*y.^3;
P8=@(x,y) X(8,2)*x+X(8,3)*y+X(8,7)*x.^3+X(8,8)*x.^2.*y+X(8,9)*x.*y.^2+X(8,10)*y.^3;
P9=@(x,y) X(9,2)*x+X(9,3)*y+X(9,7)*x.^3+X(9,8)*x.^2.*y+X(9,9)*x.*y.^2+X(9,10)*y.^3;
P10=@(x,y) X(10,2)*x+X(10,3)*y+X(10,7)*x.^3+X(10,8)*x.^2.*y+X(10,9)*x.*y.^2+X(10,10)*y.^3;
B=[]; b=[];
for i=1:10
x=P(i,1);y=P(i,2); b=[b;fname(x,y)];
A=[P1(x,y),P2(x,y) P3(x,y) P4(x,y) P5(x,y) P6(x,y)];
A=[A P7(x,y) P8(x,y) P9(x,y) P10(x,y)];
B=[B;A];
end
c=B\b; 
disp(`The integration is `)
c(1)
\end{verbatim}
To build up a numerical quadrature formula,  we simply use the first row of the inverse of
$B$ above and obtain the following 
\begin{equation}
\label{inthexagon}
\int_\Omega f(x,y)dxdy \approx \sum_{i=1}^{10} w_i f(x_i,y_i),
\end{equation}
where $w_1, \cdots, w_{10}$ are given below:
$$
4596/29 ,     -18901/97 ,       4053/26 ,      -4053/104  ,   -18901/97 , 0, 0, 
4053/26 , 0,         -4053/104  
$$
and $(x_i,y_i), i=1, \cdots, 10$ are $(0,0), (1/12,0),(1/6,0),(1/4,0),(0,1/12), (1/12,1/12), 
(1/6,1/12),$ \par $(0,1/6),(1/12,1/6),(0,1/4)$.  
Note that there are three zeros for the weights 
and hence, the function values are not all needed.  That is we only need 7 function values to have
the integration to be exact for all cubic polynomials.  
\end{example}

\subsection{Quadrature Formulas with Higher Polynomial Precision}
As seen from Example~\ref{magic1}, we can use a less number of function values to 
make the integration exact for polynomials of degree $d$.  In other words, we
can have a quadrature formula with higher polynomial precision.  Construction of such quadrature formulas
has been studied for many decades. See many quadrature formulae in 2D and cubature formulas in 3D in 
\cite{HS56}, \cite{S71}, \cite{C97}, \cite{HKA12}, \cite{P16},  and etc..  
In this subsection, we present   a few more examples.  
\begin{example}
\label{magic}
Let us give the quadrature formula for degree $2$ which is 
\begin{equation}
\label{quadrature2}
\int_T f(x,y)dxdy \approx \frac{1}{6} (f(1/2,0)+f(1,1/2)+f(1/2,1/2)),
\end{equation}
where $T$ is the triangle with vertices $(0,0), (1,0), (1,1)$.  The formula will be 
exact for all quadratic polynomials although the formula needs only three function 
values.  
\end{example}

\begin{example}
In this example, we recall an elementary quadrature formula over $[0, 1]^2$:
\begin{eqnarray}
\label{sformua1}
\int_{-1}^1\int_{-1}^1 f(x,y) dxdy &\approx &  4f(0,0)
\end{eqnarray}
Note that the formula use only one function value, but will be exact for all linear 
polynomial functions and many even and odd polynomials as discussed before. 
\end{example}

\begin{example}
\label{square}
Let us consider to use orthogonal polynomials of degree $2$ to find quadrature formulas.
There are many of them as we can use different points in $[0, 1]^2$ to find a quadrature 
formula. The following is just one of them. 
\begin{eqnarray}
\label{sformula2}
\int_0^1\int_0^1 f(x,y) dxdy &\approx &  (f(0,0)+f(0,1)+f(1,0)+f(1,1))/12+ 2f(1/2,1/2)/3.
\end{eqnarray}
Although it is constructed based on orthogonal polynomials of degree $2$, the above
quadrature will be exact for all polynomials of degree $3$.  
This formula can be compared with 
a standard quadrature formula based on 15 function values for cubic polynomial precision. 
\end{example}

\begin{example}
\label{quintic}
Next we present the quadrature formula based on our iterative 
orthogonal polynomials of degree $5$ which is 
\begin{eqnarray*}
\int_T f(x,y)dxdy &\approx &\frac{1}{2016} 
(11f(0,0)+25f(1/5,0)+25f(2/5,0)+25f(3/5,0)+25f(4/5,0) + \\
&& 11f(1,0)+25f(1,1/5)+25f(1,2/5)+25f(1,3/5)+25f(1,4/5) +\\
&& 11f(1,1)+25f(4/5,4/5)+25f(3/5,3/5)+25f(2/5,2/5)+25f(1/5,1/5) +\\
&& 200f(2/5,1/5)+25f(3/5,1/5)+200f(4/5,1/5)+25f(4/5,2/5)+\\
&& 200f(4/5,3/5)+25f(3/5,2/5)),  
\end{eqnarray*}
where $T$ is the triangle with vertices $(0,0), (1,0), (1,1)$.  The formula will be 
exact for all quintic polynomials. Note that the coefficients are positive and symmetric 
with respect to $T$.
\end{example}

\section{Remarks}
We have a few remarks in order
\begin{remark}
It is possible to extend the ideas in the paper  to compute orthogonal polynomials with 
respect to a weight $w$, in particular, when $w$ is a polynomial or piecewise polynomials, 
there is a formula which is exact for the following integration:
\begin{equation}
\label{keyformula2}
\int_T w(x,y)B^T_{ijk}(x,y) B^T_{i',j',k'}(x,y) dxdy
\end{equation}
for any $B^T_{ijk}$ and $B^T_{i',j',k'}$ with $i+j+k=d$ and $i'+j'+k'=d$, where the weight function $w(x,y)$ is 
polynomial, e.g. Jacobi weights with integer parameters $W(x,y)=\prod_{i=1}^n(1-x_i)^\alpha(1+x_i)^\beta$
over $[-1,1]^n$ for integers $\alpha$ and $\beta$. More generally, if a weight function is a piecewise polynomial function over a triangulation 
$\triangle$ of polygonal domain $\Omega$ of interest, we can build up $M_\triangle$ 
piece-wisely over each
triangle in $\triangle$. Similarly, there
is such a formula in the 3D setting. With these formula, we are able to compute orthogonal 
polynomials with respect to a weight $w$ when $w$ is a polynomial. When $w$ is not a 
polynomial, 
we can approximate $w$ by a polynomial which leads to approximate orthogonal polynomials. 
We leave the discussion to the interested reader. 
\end{remark}

\begin{remark}
	When a given domain $\Omega$ is not a polygon, e.g. the domain used by 
Koornwinder (cf. \cite{DX14}), we can approximate $\Omega$ by using a polygon with piecewise linear boundary 	and then apply our 
	algorithms to obtain orthonormal polynomials.  
\end{remark}

\begin{remark}
There are some recurrence relationship between the orthogonal polynomials like Legendre 
formula in the univariate setting.  It is interesting to explore the relationships to 
obtain some clean formulas due to the fact that the coefficients of orthogonal 
polynomials are irrational numbers (containing some radical numbers).  We leave the 
detail to the interested reader.  
\end{remark}

\begin{remark}
We can extend the ideas in this paper to construct orthonormal spline basis for a Sobolev 
space $H^1(\Omega)$ for any polygonal domain. That is we combine 
$M_\triangle$ and $S_\triangle$ together, where $S_\triangle$ stands for the stiffness 
matrix over $\triangle$.  We do not know any application of such a basis. 
We leave the detail to the interested reader.  
\end{remark}

\begin{remark}
We can extend the ideas in this paper to construct orthonormal wavelets/framelets over 
any triangulation $\triangle$ of the domain $\Omega$ of interest. This will 
extend the study in \cite{GL13}, where box splines are used.   We leave the detail to the 
interested reader. 
\end{remark}

\begin{remark}
	Certainly, we can extend the computational algorithm to the multi-dimensional setting 
when the dimensionality
	$n\ge 4$. So far we do not have a MATLAB code to generate n-dimensional $M_\triangle$ 
over an $n$-simplex.  We leave the detail 	to the interested reader.  
\end{remark}

\bigskip
\noindent
{\bf Acknowledgment: }
The author would like to thank both reviewers for their comments and several references to 
improve the readability of the paper. 

\medskip
\noindent
{\bf Declarations: }

\noindent {\bf Funding: } The author is supported by the Simons Foundation for collaboration
grant \#864439.

\noindent {\bf Completing Interest: }  
The author declared that he has no conflict of interest.

\end{document}